\documentclass[10pt]{article}

\usepackage{amsmath}
\usepackage{amsthm}
\usepackage{amssymb}
\usepackage{amsfonts}
\usepackage{graphicx}
\usepackage{verbatim}
\usepackage{mathrsfs}
\usepackage{subfigure}
\usepackage{fancyhdr}
\usepackage{latexsym}
\usepackage{graphicx}
\usepackage{dsfont}
\usepackage{color}
\usepackage{epstopdf}
\usepackage{bm}

\theoremstyle{plain}
\newtheorem{theorem}{Theorem}[section]
\newtheorem{proposition}[theorem]{Proposition}
\newtheorem{lemma}[theorem]{Lemma}

\newtheorem{remark}[theorem]{Remark}
\newtheorem{definition}[theorem]{Definition}

\newcommand{\R}{\mathbb R}
\newcommand{\N}{\mathbb N}
\newcommand{\Z}{\mathbb Z}
\newcommand{\C}{\mathbb C}
\newcommand{\B}{\mathcal B}
\newcommand{\E}{\mathcal E}
\newcommand{\X}{\mathcal X}
\newcommand{\CP}{\mathcal P}
\newcommand{\CC}{\mathcal C}
\DeclareMathOperator{\sgn}{sgn}

\DeclareFontFamily{U}{matha}{\hyphenchar\font45}
\DeclareFontShape{U}{matha}{m}{n}{
      <5> <6> <7> <8> <9> <10> gen * matha
      <10.95> matha10 <12> <14.4> <17.28> <20.74> <24.88> matha12
      }{}
\DeclareSymbolFont{matha}{U}{matha}{m}{n}
\DeclareFontSubstitution{U}{matha}{m}{n}
\DeclareFontFamily{U}{mathx}{\hyphenchar\font45}
\DeclareFontShape{U}{mathx}{m}{n}{
      <5> <6> <7> <8> <9> <10>
      <10.95> <12> <14.4> <17.28> <20.74> <24.88>
      mathx10
      }{}
\DeclareSymbolFont{mathx}{U}{mathx}{m}{n}
\DeclareFontSubstitution{U}{mathx}{m}{n}
\DeclareMathDelimiter{\VERT}{0}{matha}{"7E}{mathx}{"17}

\title{Existence and instability of steady states for a triangular cross-diffusion system: a computer-assisted proof}

\author{Maxime Breden \thanks{CMLA, ENS Cachan, CNRS, Université Paris-Saclay, 61 avenue du Pr\'esident Wilson, 94230 Cachan, France \& D\'epartement de Math\'ematiques et de Statistique, Universit\'e Laval, 1045 avenue de la M\'edecine, Qu\'ebec, QC, G1V 0A6, Canada. maxime.breden@ens-cachan.fr}
\and Roberto Castelli \thanks{VU University Amsterdam, Department of Mathematics, De Boelelaan 1081, 1081 HV Amsterdam, The Netherlands. r.castelli@vu.nl}
}

\date{}

\begin{document}
\maketitle

\begin{abstract}
In this paper, we present and apply a computer-assisted method to study steady states of a triangular cross-diffusion system. Our approach consist in an \emph{a posteriori} validation procedure, that is based on using a fixed point argument around a numerically computed solution, in the spirit of the Newton-Kantorovich theorem. It allows us to prove the existence of various non homogeneous steady states for different parameter values. In some situations, we get as many as 13 coexisting steady states. We also apply the \emph{a posteriori} validation procedure to study the linear stability of the obtained steady states, proving that many of them are in fact unstable.
\end{abstract}

\begin{center}
{\bf \small Key words} 
{\small Rigorous numerics $\cdot$ Cross-diffusion $\cdot$ Steady states \\
Eigenvalue problem $\cdot$ Spectral analysis $\cdot$ Fixed point argument}
\end{center}

\begin{center}
{\bf \small 2010 AMS Subject Classification} {\small 35K59 $\cdot$ 35Q92 $\cdot$ 65G20 $\cdot$ 65N35}
\end{center}

\section{Introduction}

The primary goal of describing physical systems with mathematical models is to be able to explain and predict natural phenomena, within some range of approximation.  In some circumstances the mathematical prediction and the experimental evidence don't agree, and a more trustful model is then required. Typically, one can add nonlinear or non homogeneous terms to get a more refined model, but this often seriously complicates the mathematical analysis of the system, which can become very hard, if not impossible, to study analytically.
In this situation, numerical simulations allow insight of the phenomena and provide approximate, often very accurate, solutions. Aiming at formulating theorems, a powerful tool to validate approximate solutions into rigorous mathematical statements is provided by the rigorous computational techniques. 

The diffusive Lotka-Volterra system, a well known model for population dynamics to study the competition between two species, is paradigmatic of  the situation discussed above. It consists in the system
\begin{equation}
\label{eq:without_crossdiff}
\left\{
\begin{aligned}
&\frac{\partial u}{\partial t}=d_1\Delta u + (r_1-a_1u-b_1v)u, \quad &\text{on }\R_+\times\Omega,\\
&\frac{\partial v}{\partial t}=d_2\Delta v + (r_2-b_2u-a_2v)v, \quad &\text{on }\R_+\times\Omega,\\
&\frac{\partial u}{\partial n}=0=\frac{\partial v}{\partial n},\quad \text{on }\R_+\times\partial\Omega,
\end{aligned}
\right.
\end{equation}
where $\Omega$ is a bounded domain of $\R^N$, and $u(t,x),v(t,x)\geq 0$ represent the population densities of two species at time $t$ and position $x$. The non negative coefficients $d_i$, $r_i$, $a_i$ and $b_i$ ($i=1,2$) describe the diffusion, the unhindered growth of the species, the intra-specific competition and the inter-specific competition respectively. 

One of the fundamental problems is to determine if and under which assumptions the two species coexist, that converts into proving   the existence or non-existence of stable positive equilibrium solutions. Several  works has been produced to classify and analyse the stability of the equilibria for \eqref{eq:without_crossdiff} and of related systems. We refer for instance to \cite{MimEiFang91} for a short review. Of particular interest for our discussion is the result presented in \cite{KisWei85}.  If the domain $\Omega$ is convex, in that paper it is proved that any spatially non-constant equilibrium solution  of \eqref{eq:without_crossdiff}  is unstable,  if it exists. This  implies that if the two species coexist, their densities must be homogeneous in the whole domain.

However, biological observations suggest that two competing species could coexist by forming pattern to avoid each other (a phenomenom called \emph{spatial segregation}). Therefore we would like the model to exhibit stable non homogeneous steady states, but the quoted result shows that this is excluded (at least for convex domains). We point out that stable non homogeneous equilibria have been shown to exist for non convex domains \cite{MatMim83}, or for systems involving more than two species \cite{Kis82}.

\medskip

In the case of two competing species, to account for the expected stable inhomogeneous steady sates, a generalization of~\eqref{eq:without_crossdiff} was proposed in~\cite{ShiKawTer79}:
\begin{equation}
\label{eq:with_crossdiff}
\left\{
\begin{aligned}
&\frac{\partial u}{\partial t}=\Delta ((d_1+d_{12}v)u) + (r_1-a_1u-b_1v)u, \quad \text{on }\R_+\times\Omega,\\
&\frac{\partial v}{\partial t}=\Delta ((d_2+d_{21}u)v) + (r_2-b_2u-a_2v)v, \quad \text{on }\R_+\times\Omega,\\
&\frac{\partial u}{\partial n}=0=\frac{\partial v}{\partial n},\quad \text{on }\R_+\times\partial\Omega,
\end{aligned}
\right.
\end{equation}
where the added \emph{cross-diffusion} terms $\Delta(uv)$ model that the two species try to avoid each other, by diffusing more when more individuals of the other species are present.

Since its introduction the system~\eqref{eq:with_crossdiff} has been studied extensively, one of the main reason being that it seems to exhibit a much wider variety of steady states than~\eqref{eq:without_crossdiff}, especially non homogeneous ones, in accordance to laboratory experiments. Several numerical studies have been presented, displaying intricate bifurcation diagrams of steady states  (see for instance \cite{IidMimNim06,IzuMim08} and also Figure~\ref{fig:bifurcation_diagram}). Consider for instance the homogenous equilibria 
\begin{equation}
\label{eq:homogeneous_eq}
(u_{eq},v_{eq}):=\left(\frac{r_1a_2-r_2b_1}{a_1a_2-b_1b_2},\frac{r_2a_1-r_1b_2}{a_1a_2-b_1b_2}\right)
\end{equation}
in the {\it strong intra-specific} case, i.e. when $\frac{b_1}{a_2}<\frac{r_1}{r_2}<\frac{a_1}{b_2}$ (the other case $\frac{a_1}{b_2}<\frac{r_1}{r_2}<\frac{b_1}{a_2}$ is known as the \emph{strong inter-specific competition} case). While $(u_{eq},v_{eq})$ is stable for~\eqref{eq:without_crossdiff} for any diffusion coefficient $d_1,d_2\geq 0$, adding strong enough cross-diffusion can destabilize this equilibria, from which new non homogeneous steady states can bifurcate. A link between this \emph{cross-diffusion induced instability} and the standard \emph{Turing instability} for reaction-diffusion systems is made in~\cite{IidMimNim06}.

\medskip

While from one side the addition of these nonlinear cross-diffusion terms yields a more reliable model, on the other it seriously complicates the analytical treatment of the system. Even the existence of global classical solutions of~\eqref{eq:with_crossdiff} (completed with non negative initial data) is a challenging question that is still fairly open. Local in time existence can be obtained by the theory of quasilinear parabolic systems \cite{Ama89}, but to then get bounds that prevent blowups requires restrictions on the coefficients, as for instance $d_{21}=0$ \cite{ChoLuiYam03}. In this particular case, sometimes called \emph{triangular cross-diffusion system}, some recent progresses were also made in~\cite{DevTre15} by combining entropy methods with a 3 component reaction-diffusion system without cross-diffusion (first introduced in~\cite{IidMimNim06}), that is used to approach~\eqref{eq:with_crossdiff}. Entropy methods were also used to improve upon the existing results for the \emph{full} cross-diffusion system, see~\cite{Jun16,DevLepMouTre15} and the references therein.

\medskip

\begin{figure}[h!]
\vspace{-0.5cm}
\centering
\includegraphics[width=12cm]{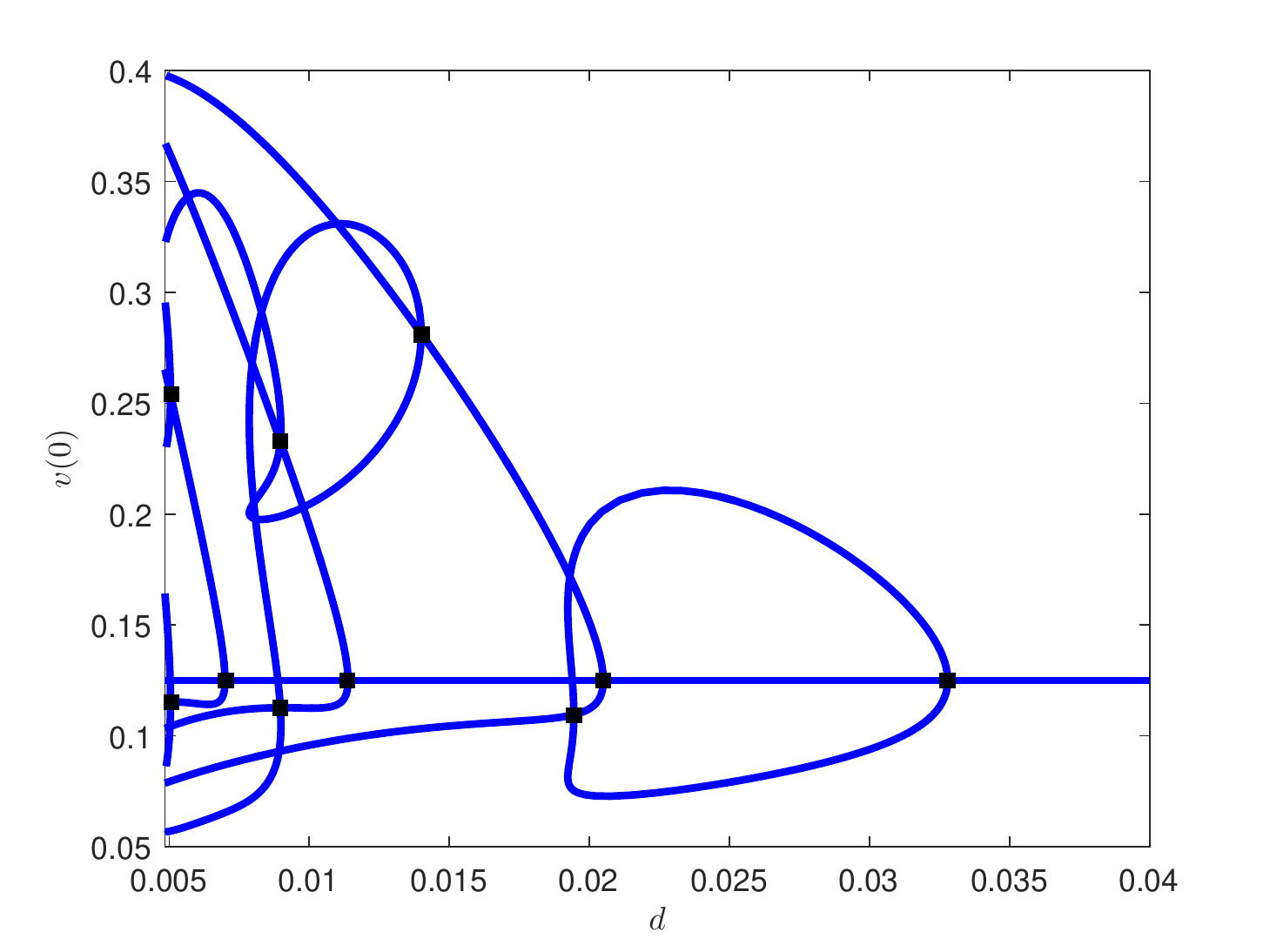}
\vspace{-0.5cm}
\caption{A numerical bifurcation diagram of steady states of~\eqref{eq:with_crossdiff}, in the strong intra-specific case. The space domain $\Omega$ is $(0,1)$, $r_1=5$, $r_2=2$, $a_1=3$, $a_2=3$, $b_1=1$, $b_2=1$, $d_{12}=3$, $d_{21}=0$ and $d_1=d_2=d$ is left as the bifurcation parameter.}
\label{fig:bifurcation_diagram}
\vspace{-0.2cm}
\end{figure}

Existence and stability results for non homogeneous steady states of~\eqref{eq:with_crossdiff} have also been established following different  approaches including bifurcations techniques~\cite{MimKaw80}, singular perturbation techniques~\cite{MimNisTesTsu84,LouNiYot04} and fixed point index theory~\cite{RyuAhn03}. Because of the presence of the cross-diffusion term, the analytical studies are limited to those solutions that are either close the homogeneous steady states, or in very specific parameter ranges. However, the numerically computed bifurcation diagram reveals a very rich structure that includes coexistence of many different steady states for a given set of parameter values, as well as secondary bifurcations. To the best of the author's knowledge, even the existence of these solutions is not yet  proved and it seems  out of reach of purely analytical techniques.

\medskip

The aim of this paper is to prove existence and study the linear stability  of several non-homogenous steady states of~\eqref{eq:with_crossdiff}, significantly far from being perturbations of the homogeneous equilibria, also showing multiplicity of solutions for the same set of parameters. The kind of technique adopted here is often referred to as \emph{validated numerics}, because the goal is to prove the existence of a genuine solution of the problem in a sharp and explicit neighborhood of a numerical one, hence, in this sense, to validate the approximate solution (more details in Section~\ref{sec:validated_numerics}). More precisely,  we follow the so called  {\it radii polynomial } approach, a quite general technique  based on the contraction mapping argument that has been adapted to solve several differential problems in areas ranging from dynamical systems to ordinary and partial differential equations through delay differential equation and chaotic dynamics. This is the first time that rigorous computational techniques are applied to PDE system with cross interactions in the leading differential operator. The cross-diffusion terms are indeed a major technical hurdles, since they enfeeble the smoothing effect of the higher order differential operator.

In this work we restrict ourself  to the triangular case and assume that the space dimension is 1 (i.e. we fix $d_{21}=0$ and $\Omega=(0,1)$). The method can easily be extended to  higher space dimension (say 2 or 3), but of course the computational cost would increase. The generalization to a full cross-diffusion system is less straightforward. Indeed, as mentioned above, it will become apparent in the next sections that the cross-diffusion structure hinders the use of our validation method, and that we take advantage of the triangular configuration to overcome this difficulty. 

\medskip
The system we are dealing with is the following:
\begin{equation}
\label{eq:steady_states}
\left\{
\begin{aligned}
&\!\left((d_1+d_{12} v)u\right)''+(r_1-a_1u-b_1v)u=0, \quad &\text{on }(0,1),\\
&d_2v''+(r_2-b_2u-a_2 v)v=0, \quad &\text{on }(0,1),\\
&u'(0)=u'(1)=0, \\
&v'(0)=v'(1)=0.
\end{aligned}
\right.
\end{equation}
Figure~\ref{fig:bifurcation_diagram} depicts a  bifurcation diagram of solutions of~\eqref{eq:steady_states}, with given values for the parameters $r_i$, $a_i$, $b_i$ and $d_{12}$. This diagram was first obtained numerically in~\cite{IidMimNim06}, using a 3-component system without cross-diffusion that approaches~\eqref{eq:steady_states}.  We point out that even in the somewhat restricted framework with $\Omega=(0,1)$ and $d_{12}=0$, the steady states of~\eqref{eq:with_crossdiff} already manifest very complex and interesting behavior when the parameters vary.

The first result concerns the existence of steady states.
\begin{figure}[h!]
\vspace{-0.3cm}
\centering
\includegraphics[width=12cm]{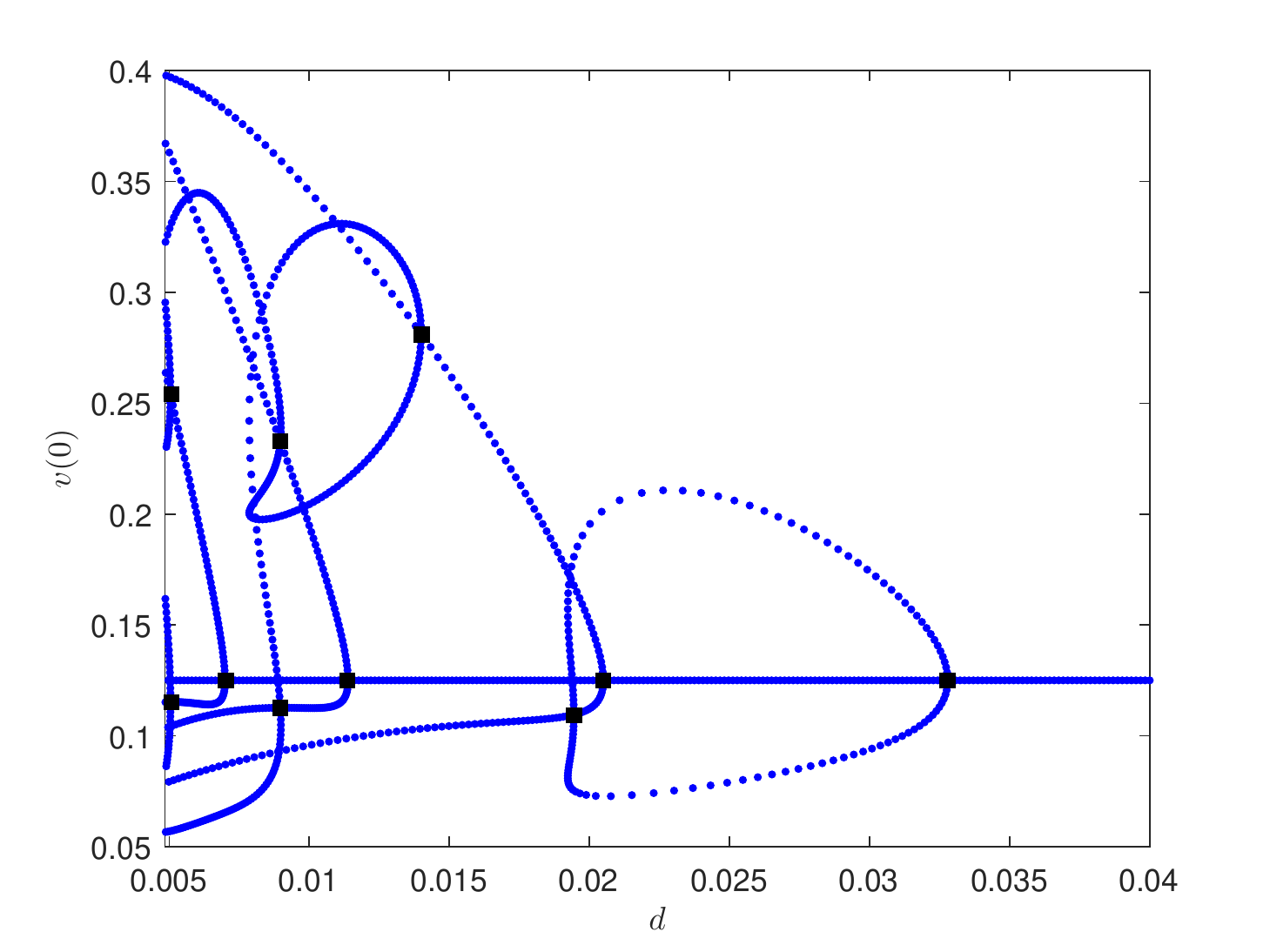}\vspace{-0.5cm}
\caption{Validated bifurcation diagram of solutions of~\eqref{eq:steady_states}. The space domain $\Omega$ is $(0,1)$, $r_1=5$, $r_2=2$, $a_1=3$, $a_2=3$, $b_1=1$, $b_2=1$, $d_{12}=3$, and $d_1=d_2=d$ is left as the bifurcation parameter. Each blue dot represents a proved solution. The black squares indicate bifurcations, while the other \emph{apparent} crossings are just due to the projection (i.e. $v(0)$) we used to represent the solutions.}
\label{fig:proved_diagram}
\vspace{-0.2cm}
\end{figure}

\begin{theorem}
\label{th:exists_13}
Referring to Figure~\ref{fig:proved_diagram}, each bullet represents a solution of~\eqref{eq:steady_states}, for the parameter values $r_1=5$, $r_2=2$, $a_1=3$, $a_2=3$, $b_1=1$, $b_2=1$, $d_{12}=3$, and $d_1=d_2=d$. In particular there exists at least 13 different solutions when $d_1=d_2=0.005$.
\end{theorem}  

The proof of each steady states also provides precise qualitative  informations about the  solution, in terms of explicit bounds  of the distance (in some function space, see Section~\ref{sec:framework_steady_states}) between the genuine solution that is proved to exist and a numerically computed approximation (more details in Section~\ref{sec:results_steady_states}).

In~\cite{IidMimNim06}, the linear stability of the obtained steady states was also studied (still numerically), suggesting that most of the solutions displayed in the bifurcation diagram of Figure~\ref{fig:proved_diagram} are unstable, while others seems to be stable. In this direction, the second contribution  of this paper is a rigorous computational approach to the  study of the spectral properties of the equilibria presented above.  

\begin{figure}[h!]
\vspace{-0.5cm}
\centering
\includegraphics[width=12cm]{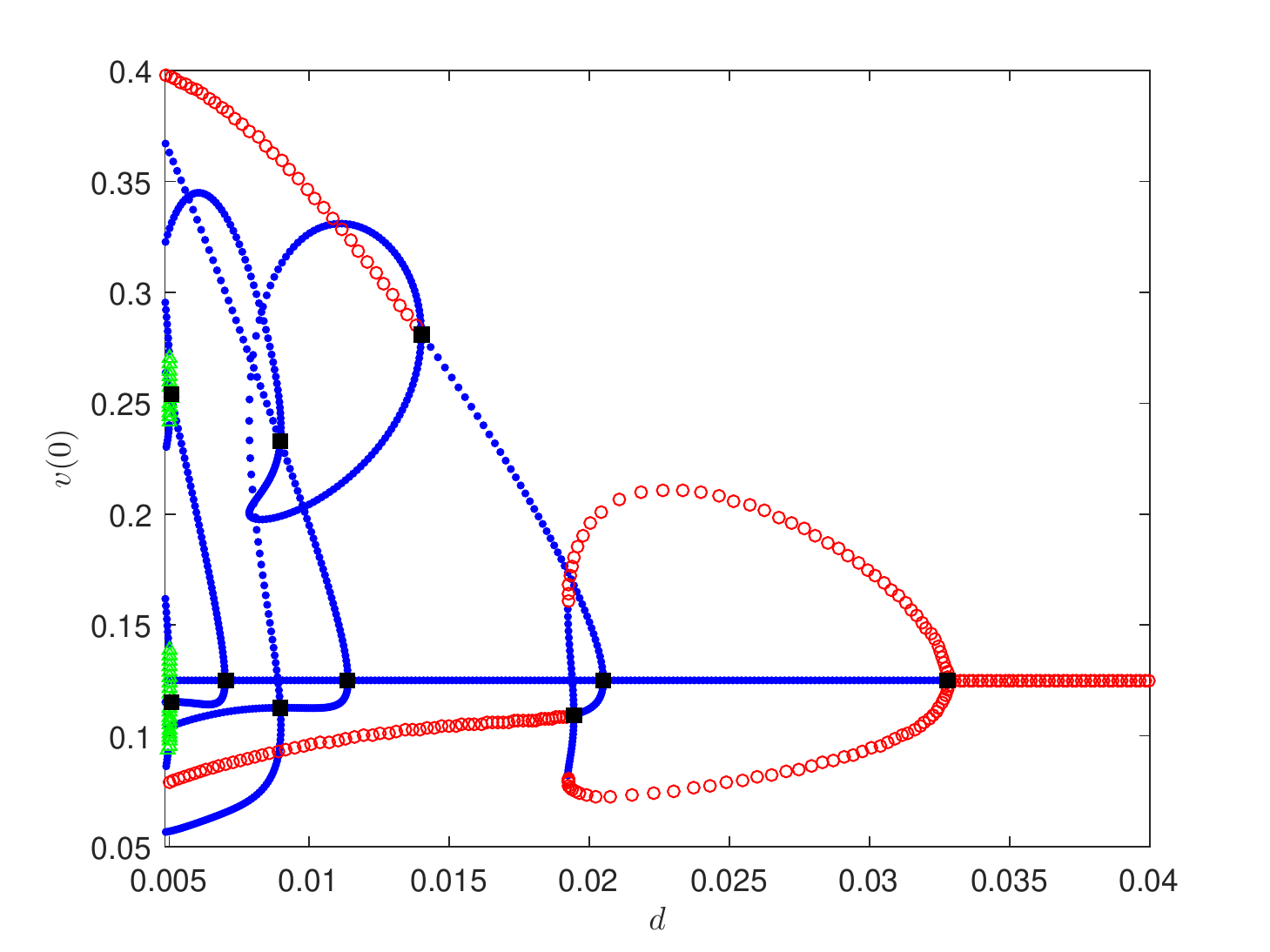}\vspace{-0.5cm}
\caption{Validated bifurcation diagram of solutions of~\eqref{eq:steady_states}. The space domain $\Omega$ is $(0,1)$, $r_1=5$, $r_2=2$, $a_1=3$, $a_2=3$, $b_1=1$, $b_2=1$, $d_{12}=3$, and $d_1=d_2=d$ is left as the bifurcation parameter. Each blue dot represents a proved solution, for which we also proved instability. Each green triangle represents a proved solution, that seems unstable numerically but for which we were not able to prove instability. Each red circle represents a proved solution, that seems stable numerically.}
\label{fig:stability_diagram}
\vspace{-0.2cm}
\end{figure}

\begin{theorem}
\label{th:unstable_11}
Referring to Figure~\ref{fig:stability_diagram}, each blue bullet represents an unstable steady state.
Out of the 13 solutions at parameter value $d=0.005$,  at least 11 are unstable.
\end{theorem}
The steady states marked in red in Figure~\ref{fig:stability_diagram} and in particular the two solutions out of 13 that are not concerned by the above Theorem seem to be stable. However, at the moment  we are not yet able to use our validation method to prove linear stability, as this requires to control the whole spectrum and not just a single eigenvalue. Still, we point out that the straight line of solutions at $v(0)=0.125$ corresponds to the homogeneous steady state~\eqref{eq:homogeneous_eq}, for which the linear stability could of course be studied analytically. In particular it could be proven that, before the bifurcation occuring at $d\simeq 0.0328$ the homogeneous steady state is linearly stable. Validated numerics techniques were used successfully to prove stability in other situations (\cite{CaiZen14,KinWatNak14}, see also~\cite{Mir16}), but the adaptation to our problem presents several challenges (mainly due to the cross-diffusion terms, which {\it muddle} the asymptotic structure of the eigenvalue problem, see Section~\ref{sec:framework_instability}) and will be the object of future investigations. % and is out of the scope of the present work.

\medskip

The paper is organized as follows. In Section~\ref{sec:validated_numerics}, we give a brief  exposition of the validated numerics techniques we apply in this work, as well as additional references on the subject.  In particular we state the Theorem \ref{th:radii_pol}, that serves as common reference and guideline for both the rigorous computation of the steady states and the rigorous enclosure of the eigenvalues. Section~\ref{sec:def} is devoted to the introduction of some notations and elementary estimates that are used throughout the paper. In Sections~\ref{sec:framework_steady_states} to~\ref{sec:results_steady_states}, we then prove the existence of steady states. More precisely, in Section~\ref{sec:fun_F} we expose how to reformulate the problem of existence of solutions of~\eqref{eq:steady_states} into a framework suitable for Theorem~\ref{th:radii_pol}. In Section~\ref{sec:bounds_steady_states}, we then derive explicit  and implementable formulas for the bounds involved in Theorem~\ref{th:radii_pol}, and finally give examples of results in Section~\ref{sec:results_steady_states}. Sections~\ref{sec:framework_instability} to~\ref{sec:results_instability} are dedicated to proving the instability of some of these steady states, following the same procedure: suitable reformulation in Section~\ref{sec:F2}, bounds in Section~\ref{sec:bounds_instability} and results in Section~\ref{sec:results_instability}.

\section{Overview of the rigorous computational method}
\label{sec:validated_numerics}

In this section we briefly explain the strategy for both solving \eqref{eq:steady_states} and computing the linear stability of the steady states by means of validated numerics techniques.
Each problem is formulated as solving an equation $F(X)=0$ defined on a suitable Banach space. 
The core of the method, first presented in~\cite{Yam98}, consists in the introduction of an operator $T$ whose fixed points are in one-to-one correspondence with the   zeros of $F(X)$. The existence and enclosure of the  solution follow by the Banach fixed point theorem once the operator $T$ is proven to be a contraction on some complete set. The explicit determination of the neighborhood on which the operator is a contraction is done efficiently using the \emph{radii polynomial approach} (see~\cite{DayLesMis07}), which is reminiscent of the Newton-Kantorovich Theorem. The technique can be summarized in the following statement.
\begin{theorem}
\label{th:radii_pol}
Let $(\X,\left\Vert\cdot\right\Vert_\X)$, $(\mathcal Y,\left\Vert\cdot\right\Vert_\mathcal Y)$ be  Banach spaces and $F:\X\to \mathcal Y$ a $\CC^1$ function. Let $A:\mathcal Y\to \X$ and $A^{\dag}:\X\to \mathcal Y $ be linear operators, so that $AF$ maps $\X$ into itself. Let $\bar X\in\X$ and assume there exist positive constants $Y$, $Z_0$, $Z_1$ and a positive function $r\mapsto Z_2(r)$ such that
\begin{align}
\left\Vert AF(\bar X) \right\Vert_\X &\le Y \label{def:Y}\\
\left\VERT I-AA^{\dag} \right\VERT_\X &\le Z_0 \label{def:Z_0}\\
\left\VERT A\left(DF(\bar X)-A^{\dag}\right) \right\VERT_\X &\le Z_1 \label{def:Z_1}\\
\left\VERT A\left(DF(X)-DF(\bar X)\right) \right\VERT_\X &\le rZ_2(r) \quad \forall~X\in\B_\X(\bar X,r) \label{def:Z_2},
\end{align}
where $\B_\X(\bar X,r)$ is the closed ball of $\X$, centered at $\bar X$ and of radius $r$, and $\VERT\cdot\VERT_\X$ denotes the operator norm on $\X$. Define the function $P$ as
\begin{equation}
\label{def:radii_pol}
P(r)=Z_2(r) r^2 -\left(1-(Z_0+Z_1)\right)r + Y.
\end{equation}
If there exists $r>0$ such that $P(r)<0$, then the operator $T:\X\to\X$ defined as
\begin{equation}
\label{def:T}
T=I-AF
\end{equation}
has a unique fixed point in $\B_\X(\bar X,r)$. Moreover if $A$ is injective, then $F$ has a unique zero in $\B_\X(\bar X,r)$.
\end{theorem}
We omit the proof of the theorem, that can be found for instance  in \cite{DayLesMis07}. Nevertheless, in the next remark we explain the role of the  different operators involved in the theorem and some instructions on how to define them. These considerations are detailed and made more explicit in Section~\ref{sec:bounds_steady_states} (resp. Section~\ref{sec:bounds_instability}), where we derive the bounds $Y$, $Z_0$, $Z_1$ and $Z_2(r)$ for an $F$ associated to the existence of solutions of~\eqref{eq:steady_states} (resp. their instability). We also mention that in practice, because of the way we define $A$, its injectivity is in fact implied by the existence of a $r>0$ such that $P(r)<0$ (see Proposition~\ref{prop:radii_pol}).

\begin{remark}\label{rmk:thm}
\begin{itemize}
\item $\bar X$ is chosen as an approximate  solution for $F(X)=0$, to be computed numerically as zero for a  finite dimensional approximation of $F$.  The constant $Y$ is the {\it defect bound} and measures how far is $\bar X$ from being a fixed point of $T$.  Depending on the accuracy of the approximate solution $\bar X$, we expect $Y$ to be small.
 
\item The $Z_i(r)$, $i=0,1,2$, are meant as bounds for the rate of contraction of the operator $T$ in the ball  $\B_\X(\bar X,r)$. More precisely, $Z_0+Z_1$ provides a bound for the derivative of $T$ at $\bar X$, while $Z_2(r)$ gives a correction for the derivative in the whole ball  $\B_\X(\bar X,r)$.  Assume for a moment  that $Z_2$ is constant. Necessary and sufficient conditions for the existence of an $r>0$ such that $P(r)<0$ are given by
\begin{equation*}
Z_0+Z_1<1 \quad \text{and}\quad \left(1-(Z_0+Z_1)\right)^2> 4Z_2Y.
\end{equation*}
The two conditions imply that $T$ is a contraction on the ball $\B_\X(\bar X,r)$. In order to obtain a small $Z_1$, the operator  $A^{\dag}$ is conceived as an  approximation of $DF(\bar X)$ (again based on a finite dimensional approximation of $F$). Similarly, the operator  $A$ is constructed as an  approximate inverse of $A^{\dag}$, which will then make $Z_0$ small.
 A key point is to define $A^{\dag}$ and $A$ in a smart way, to have good enough approximations while being able to derive tight bounds for $Z_1$. Typically $\mathcal Y$ is a space of functions less regular than $\X$. To this extent, the operator $A$ acts as a smoothing operator.
\item If the nonlinearity of $F$ is a  polynomial  (say of degree $d$), $Z_2$ can be constructed as a polynomial (of degree $d-2$), and therefore $P(r)$ is indeed a polynomial (of same degree $d$ as $F$). 
\item All the  bounds are obtained through a combination of analytic estimates (because the spaces involved  are naturally infinite dimensional) and numerical computations (since they depend on the approximate solution $\bar X$). To ensure that all possible round off errors are controlled during the computations, we use an interval arithmetic package (in our case INTLAB~\cite{Rum99}).
\end{itemize} 
\end{remark}
As said in the Introduction,  this work is far from being the first application of this kind of rigorous computational  techniques to solve systems of PDEs, see for instance \cite{ MR3353132, Cas16, CasTei16,DayLesMis07, MR2776917, GamLes11}. Particularly related to our work is the result presented in~\cite{BreLesVan13}. In that paper a similar method was used to rigorously validate a bifurcation diagram of steady states of a 3-component reaction diffusion system (without cross-diffusion term). The system considered in~\cite{BreLesVan13} depends on a parameter $\varepsilon$  and it has the property that  its steady states approach the solutions of~\eqref{eq:steady_states} as $\varepsilon$ goes to $0$, see ~\cite{IzuMim08}. However, the proof could only be made for a fixed (small) $\varepsilon$ and  the limit case $\varepsilon= 0$ is in some sense singular. Therefore the cross-diffusion case could not be handled.

More broadly, techniques similar to the one presented in this work were developed to prove the existence of fixed points, periodic orbits, invariant manifolds and connecting orbits for ordinary differential equations, infinite dimensional maps, partial differential equations and delay differential equations (see for instance~\cite{BerLes15,BerMirRei15,BerShe15,CasLes13,CasLesMir16,dAmb15,GamLes16,Les10,LlaMir16}).
We also mention the existence of comparable techniques where, instead of computing $A$ by using a finite dimensional truncation as we do, a bound on the norm of the inverse of $DF(\bar X)$ is obtained via spectral estimations (see~\cite{Plu01, McK09}    and the references therein). Instead of the contraction mapping principle, computer assisted proofs in dynamical systems are frequently based on topological tools  as covering relations,  the Brouwer degree, the fixed point index, the Conley index, see for instance \cite{Cap15, Gid04, Mis05, Mro96, Zgl01}. To conclude this paragraph, we refer the interested reader to \cite{Ari15, Bah11, Gal16, Yam01}  for a list, surely not exhaustive, of rigorous computational techniques developed to solve a variety of problems, not necessarily in the area of   dynamical systems.

Theorem \ref{th:radii_pol} is the cornerstone for all the proofs that are presented in this paper.  Whatever problem we want to solve, once the system is rephrased as a zero finding problem and the hypothesis of the theorem are verified, then the proof follows as application of the theorem. Thus, for a given problem $(\CP)$, we proceed as follows:
\begin{enumerate}
\item Introduce a Banach space $(\X,\left\Vert\cdot\right\Vert_\X)$ and a  $\CC^1$ function $F$ defined on  $\X$ so that the solutions of $F(X)=0$ correspond to solutions of $(\CP)$;
\item Compute a numerical approximation $\bar X\in \X$ so that  $F(\bar X)\approx 0$ and define the linear operators $A$ and $A^\dag$;
\item Define and compute the bounds $Y,Z_i(r)$ satisfying \eqref{def:Y}-\eqref{def:Z_2};
\item Check that $P(r)$ given in \eqref{def:radii_pol} is negative for some $r>0$.
\end{enumerate}
If the last condition is met, the existence of a solution $X$ for $F(X)=0$ is proved in the form specified in the theorem.

In the sequel we detail  each step of the above list for the problem of proving existence of steady states (Section \ref{sec:framework_steady_states}) and for the problem of proving their linear instability (Section \ref{sec:framework_instability}).

\section{Sequence space, convolutions and norm estimates}
\label{sec:def}

The solutions of system~\eqref{eq:steady_states} as well as the eigenfunctions of the linearised system are sought in the form of Fourier series,  which is fairly natural given the boundary conditions included in~\eqref{eq:steady_states}. This approach also provides a very convenient setting to apply our validated numerics technique. In this section  we introduce  the sequences space  relevant for our analysis and  we recall some useful properties. The material presented here is standard and mainly included for the sake of completeness and to fix some notations.

\begin{definition}\label{def:ell1}
Let $\nu>1$. For any sequence $u=\left(u_k\right)_{k\geq 0} \in\C^{\N}$ we define the $\nu$-norm of $u$ as 
\begin{equation*}
\left\Vert u\right\Vert_{\nu} = \vert u_0\vert + 2\sum_{k\geq 1}\left\vert u_k\right\vert \nu^{|k|},
\end{equation*}
and introduce the space 
\begin{equation*}
\ell^1_{\nu} = \left\{ u\in\C^{\N},\ \left\Vert u\right\Vert_{\nu} < \infty \right\}.
\end{equation*}
We also define $\ell^1_\nu(\R)$ the subspace of $\ell^1_\nu$ made of real sequences.
\end{definition}
\begin{definition}
For any $u,v\in\ell^1_{\nu}$, we define the sequences $(u\ast v)$, $(u\star v)$ and $(u\bullet v)$ as
\begin{equation*}
(u\ast v)_k= \sum_{\substack{k_1,k_2\in \Z \\ k_1+k_2=k}} u_{\vert k_1\vert}v_{\vert k_2\vert},\quad (u\star v)_k= \sum_{\substack{k_1,k_2\in \Z \\ k_1+k_2=k}} \sgn(k_1)u_{\vert k_1\vert}v_{\vert k_2\vert},
\end{equation*}
\begin{equation*}
(u\bullet v)_k= \sum_{\substack{k_1,k_2\in \Z \\ k_1+k_2=k}} \sgn(k_1)\sgn(k_2)u_{\vert k_1\vert}v_{\vert k_2\vert},
\end{equation*}
where $\sgn(k)$ denotes the sign of $k$ and  $\sgn(0)=0$.
\end{definition}
The reason for introducing three different convolution products is that we will deal with multiplications of  both even and odd functions. The role played by each of the above operations is as follows.

Let $u$ and $v$ in $\ell^1_{\nu}$ and consider the even functions (still denoted $u$ and $v$) defined by
\begin{equation*}
u(x)=u_0+2\sum_{k\geq 1}u_k\cos(kx), \quad v(x)=v_0+2\sum_{k\geq 1}v_k\cos(kx),
\end{equation*}
then $(u\ast v)$ is the sequence of Fourier coefficients of the product function $uv$, i.e.
\begin{equation*}
u(x)v(x)=(u\ast v)_0+2\sum_{k\geq 1}(u\ast v)_k\cos(kx).
\end{equation*}
If instead,  $u$ is the odd  function given by 
\begin{equation*}
u(x)=2\sum_{k\geq 1}u_k\sin(kx),
\end{equation*}
then  $(u\star v)$  provides the sequence of Fourier coefficients of the product function $uv$, i.e.
\begin{equation*}
u(x)v(x)=2\sum_{k\geq 1}(u\star v)_k\sin(kx).
\end{equation*}
Finally, if both  the functions  $u$ and $v$ are odd
\begin{equation*}
u(x)=2\sum_{k\geq 1}u_k\sin(kx),\quad v(x)=2\sum_{k\geq 1}v_k\sin(kx),
\end{equation*}
then $(u\bullet v)$ is the sequence of Fourier coefficients of the product function $uv$, i.e.
\begin{equation*}
u(x)v(x)=(u\bullet v)_0+2\sum_{k\geq 1}(u\bullet v)_k\cos(kx).
\end{equation*}
We also recall that $\ell^1_\nu$ equipped with any of the three convolution products $\ast$, $\star$ or $\bullet$ is a Banach algebra. More precisely we have the following estimate.
\begin{lemma}
\label{lem:convolution_algebra}
Let $u,v\in\ell^1_\nu$ and  $\circ\in\{\ast,\star, \bullet\}$ any of the convolution products. Then
\begin{equation*}
\left\Vert u\circ v\right\Vert_{\nu}\leq \left\Vert u\right\Vert_{\nu} \left\Vert v\right\Vert_{\nu}.
%\left\Vert u\ast v\right\Vert_{\nu}\leq \left\Vert u\right\Vert_{\nu} \left\Vert v\right\Vert_{\nu},\quad \left\Vert u\star v\right\Vert_{\nu}\leq \left\Vert u\right\Vert_{\nu} \left\Vert v\right\Vert_{\nu},\quad \left\Vert u\bullet v\right\Vert_{\nu}\leq \left\Vert u\right\Vert_{\nu} \left\Vert v\right\Vert_{\nu}.
\end{equation*}
\end{lemma}

Consider now $B:\ell^1_\nu\to \ell^1_\nu $ a bounded linear operator. To the operator $B$ is associated a infinite dimensional matrix (still denoted by $B$) so that $(Bu)_k=\sum_{j\geq 0}B(k,j)u_j$ for all $u$ in $\ell^1_\nu$. We denote by $\left\VERT B\right\VERT_\nu$ the operator norm of $B$, i.e.
\begin{equation}\label{eq:oper_norm}
\left\VERT B\right\VERT_\nu=\sup_{\left\Vert u\right\Vert_\nu=1} \left\Vert Bu \right\Vert_\nu .
\end{equation}

\begin{lemma}
\label{lem:operator_norm_nu}
Let $B:\ell^1_\nu\to \ell^1_\nu $ be a linear operator and consider $B(k,j)$ the matrix representation. Then
\begin{equation*}
\left\VERT B\right\VERT_\nu=\sup_{j\geq 0}\frac{1}{\nu^{j}}\sum_{k\geq 0}|B(k,j)|\nu^{k}. 
\end{equation*}
\end{lemma}
A linear functional $b:\ell^1_\nu \to \R$ is a particular case of the general operator $B$ given above, when $(B(v))_k =0 $ for any $k\neq 0$. The linear functional $b$ acts on $u$ as  $bu=\sum_{j\geq 0}b_j u_j$ and the operator norm of $b$ is then given by , $\left\VERT b\right\VERT_\nu=\sup_{j\geq 0}\frac{|b(j)|}{\nu^{|j|}}$. We point out that this last formula is linked to the fact that the dual space of $\ell^1$ with weight $\nu$ is isometric to the space $\ell^{\infty}$ with weight $\nu^{-1}$.

When dealing with numerical computations, we need to consider only a finite number of coefficients $u_k$ in the infinite sequence $u\in \ell^1_\nu$, that is we consider a finite dimensional projection of $u$. 
\begin{definition}
\label{def:truncation}
Let $u\in\ell^1_{\nu}$. For $m\in\N$  we denote $\hat u^m$ the \emph{truncated part} (i.e. the finite $m$-dimensional projection) of $u$ and $\check u^m$ the \emph{tail part} (i.e. infinite dimensional complement) of $u$, given as
\begin{equation*}
\hat u^m_k = \left\{
\begin{aligned}
& u_k, &k<m,\\
& 0, &k\geq m,
\end{aligned}
\right.
\quad \text{and}\quad 
\check u^m_k = \left\{
\begin{aligned}
& 0, &k<m, \\
& u_k, &k\geq m.
\end{aligned}
\right.
\end{equation*}
\end{definition}
By a slight abuse of notation, we also refer to $\hat u^m$ as the finite dimensional vector $\left(u_k\right)_{0\le k<m}$.  Moreover, when there is no possible confusion about the dimension of the projection we may drop the exponent $m$ and simply use $\hat u$ and $\check u$. 

We end this section with an estimate used to bound a convolution with a \emph{tail term}.
\begin{lemma}
\label{lem:convolution_truncation}
Let $m\in\N$ and $u,v\in\ell^1_{\nu}$ and $\circ\in\{\ast,\star, \bullet\}$ any of the convolution products. Then, for all $k\geq 0$,
\begin{equation*}
\vert u\circ \check v^m \vert_k \leq \Phi^m_k(u,\nu) \left\Vert v\right\Vert_{\nu},
\end{equation*}
where
\begin{equation*}
\Phi^m_k(u,\nu) = \sup\limits_{\vert l\vert \geq m } \frac{\left\vert u_{\vert l-k\vert}\right\vert}{\nu^{\vert l\vert}}.
\end{equation*}
%The same result holds with $u\star \check v^m$ instead of $u\ast \check v^m$.
\end{lemma}

\section{Framework for the existence of steady states}
\label{sec:framework_steady_states}

We are now concerned with  the proof of existence of non homogeneous solutions of system~\eqref{eq:steady_states}. According to the algorithm outlined in Section~\ref{sec:validated_numerics}, the first step is to reformulate  the problem in the form  $F(X)=0$, where $F$ is defined on a proper Banach space. Then we introduce the linear operators $A$ and $A^\dag$, and finally we provide the definition of the bounds $Y,Z_0,Z_1,Z_2(r)$. The latter are then combined to define the radii polynomial $P(r)$.

\subsection{Existence of steady states: the function $\bm{F}$}\label{sec:fun_F}

We preventively need to transform \eqref{eq:steady_states} into an equivalent system that is more amenable to the application of validated numerics techniques. We introduce further unknown functions and change of coordinates in order to remove the cross-diffusion nonlinearity and to obtain a system with only polynomial nonlinearities (which will be useful when deriving the validation estimates). Then we discretize the obtained system by using Fourier series and finally we introduce $F$ and the Banach space $\X$ in which we look for the solutions.

\subsubsection{ Auxiliary functions and polynomial system}
\label{sec:automatic_differentiation}

In order to remove the cross-diffusion nonlinearity we introduce the  function $w$ defined as
\begin{equation*}
w=(d_1+d_{12} v)u.
\end{equation*}
Expressing~\eqref{eq:steady_states} in term of the unknowns $v$ and $w$ gives simpler higher order terms, but the nonlinear terms become rational functions. To keep the nonlinearity in the form of polynomials, define the function $p$ as 
\begin{equation}
\label{eq:p}
p=\frac{1}{d_1+d_{12} v},
\end{equation}
so that $u=pw$. In term of  $w,v,p$, the system~\eqref{eq:steady_states} takes the form
\begin{equation*}
\left\{
\begin{aligned}
&w''+(r_1-a_1pw-b_1v)pw=0, \quad &\text{on }(0,1),\\
&d_2 v''+(r_2-b_2pw-a_2 v)v=0, \quad &\text{on }(0,1),\\
&v'(0)=v'(1)=0, \\
&w'(0)=w'(1)=0
\end{aligned}
\right.
\end{equation*}
with $p$ given  above as a function of $v(x)$. However, we want to also treat $p$ as an independent unknown, on the same footing as $w$ and $v$. For this, it is enough to append a differential equation and some initial conditions uniquely satisfied by the requested function $p(x)$. Let us consider the equation
\begin{equation*}
p'=-d_{12} p^2 v',
\end{equation*}
together with the constraint
\begin{equation*}
p(0)(d_1+d_{12} v(0))=1.
\end{equation*}
It is straightforward to check that the function $p(x)$ in \eqref{eq:p} is the only solution of such initial value problem. 
Finally, introducing a variable $s$ for $v'$, we obtain the system
\begin{equation}
\label{eq:steady_states_extended}
\left\{
\begin{aligned}
&w''+(r_1-a_1pw-b_1v)pw=0, \quad &\text{on }(0,1),\\
&d_2 s'+(r_2-b_2pw-a_2 v)v=0, \quad &\text{on }(0,1),\\
&p'+d_{12} sp^2=0, \quad &\text{on }(0,1),\\
&v'-s=0, \quad &\text{on }(0,1),\\
&p(0)(d_1+d_{12} v(0))=1,\\
&v'(0)=v'(1)=0, \\
&w'(0)=w'(1)=0.
\end{aligned}
\right.
\end{equation}
to be solved in the unknowns $w,v,p,s$. We point out that the usage of the variable $p$ to recover polynomial nonlinearities is inspired from~\cite{LesMirRan16}, where this technique was introduced in the context of validated numerics.

\subsubsection{Algebraic system in Fourier space}
We now expand the unknown functions and we project the differential system \eqref{eq:steady_states_extended}  onto the Fourier basis.
Because of the boundary conditions and \eqref{eq:p}, $v(x)$, $w(x)$ and $p(x)$ are written as cosine series.  On the opposite, since $s(x)=v'(x)$, the function $s(x)$ is expanded on the  sines basis. Precisely, consider 
\begin{equation}\label{def:fourier_expansion}
\begin{array}{ll}
{\displaystyle w(x)=w_0+2\sum_{k\geq 1}w_k\cos(k\pi x), \quad} &{\displaystyle v(x)=v_0+2\sum_{k\geq 1}v_k\cos(k\pi x),}\\
{\displaystyle p(x)=p_0+2\sum_{k\geq 1}p_k\cos(k\pi x),\quad} &{\displaystyle s(x)=s_0+2\sum_{k\geq 1}s_k\sin{(k\pi x)}.}
\end{array}
\end{equation}
The addition of the $s_0$ coefficient (which is clearly zero since $s=v'$) is deliberate, because we  see each of the sequence of Fourier coefficients $\left(v_k\right)_{k\geq 0}$, $\left(w_k\right)_{k\geq 0}$, $\left(p_k\right)_{k\geq 0}$ and $\left(s_k\right)_{k\geq 0}$ as an element of $\ell^1_\nu$. Plugging these series expansions into \eqref{eq:steady_states_extended} and then projecting back onto the cosine/sine basis, we obtain the following (infinite dimensional) algebraic system
\begin{equation}
\label{eq:algebraic_system}
\left\{
\begin{aligned}
&-(\pi k)^2w_k+r_1(p\ast w)_k-a_1(p\ast p\ast w\ast w)_k-b_1(p\ast v\ast w)_k=0,\quad &\forall~k\in \N,\\
&d_2 \pi k s_k+r_2v_k-a_2(v\ast v)_k-b_2(p\ast v\ast w)_k=0,\quad &\forall~k\in \N,\\
&-\pi k p_k+d_{12}(s\star p\ast p)_k=0,\quad &\forall~k\in \N,\\
&-\pi k v_k-s_k=0,\quad &\forall~k\in \N,\\
&\left(p_0+2\sum_{k\geq 1}p_k\right)\left(d_1+d_{12}\left(v_0+2\sum_{k\geq 1} v_k\right)\right)-1=0.
\end{aligned}
\right.
\end{equation}
to be solved for the unknown sequences $w=\{w_k\}_{k\geq 0},v=\{v_k\}_{k\geq 0}, p=\{p_k\}_{k\geq 0}, s=\{s_k\}_{k\geq 0}$.

Notice that because of the equation $-\pi k v_k-s_k=0$ with $k=0$, any solution of this system does indeed satisfy $s_0=0$. 
Notice also that for $k=0$, the equation $-\pi k p_k+d_{12}(s\star p\ast p)_k=0$ is an identity. Indeed $-0\pi p_0=0$, and it follows from the definition of the convolution products $\star$ that  $(s\star p\ast p)_0=0$. In other words, since $p(x)$ is even and $s(x)$ is odd, the product $sp^2$ is odd, hence the $0$-th Fourier coefficients vanishes.
Therefore this equation can be removed and we are then left with a {\it square} system.

\subsubsection{The $\bm{F=0}$ formulation}

For $\nu>1$, we define $\X_\nu=\left(\ell^1_\nu(\R)\right)^4$, where $\ell^1_\nu$ is given in Definition \ref{def:ell1}, and denote by $X=(v,w,p,s)$ any element in $\X_\nu$. We also use the notation $X_k$ to denote $(v_k,w_k,p_k,s_k)$. We endow $\X_\nu$ with the norm
\begin{equation}\label{eq:normXnu}
\left\Vert X\right\Vert_{\X_\nu} = \left\Vert v\right\Vert_\nu + \left\Vert w\right\Vert_\nu + \left\Vert p\right\Vert_\nu + \left\Vert s\right\Vert_\nu,
\end{equation}
which makes it a Banach space. We then define the function $F=(F^{(v)},F^{(w)},F^{(p)},F^{(s)})$ acting on $\X_{\nu}$ by
\begin{align}\label{eq:map_F}
&F_k^{(v)}(X)=-\pi k v_k-s_k,\quad &\forall~k\in \N, \\
&F_k^{(w)}(X)=-(\pi k)^2w_k+r_1(p\ast w)_k-a_1(p\ast p\ast w\ast w)_k-b_1(p\ast v\ast w)_,\quad &\forall~k\in \N, \\
&F_0^{(p)}(X)=\left(p_0+2\sum_{k\geq 1} p_k\right)\left(d_1+d_{12}\left(v_0+2\sum_{k\geq 1} v_k\right)\right)-1, \\
&F_k^{(p)}(X)=-\pi k p_k+d_{12}(s\star p\ast p)_k,\quad &\forall~k\ge 1, \\
&F_k^{(s)}(X)=d_2\pi ks_k+r_2v_k-a_2(v\ast v)_k-b_2(p\ast v\ast w)_k,\quad &\forall~k\in \N.
\end{align}

The next Lemma summarizes and justifies in a precise statement all the formal computations and substitutions made previously in this section, with the goal of solving system \eqref{eq:steady_states}.
\begin{lemma}
\label{lem:rigorous_justification}
Let $\nu>1$. Assume that there exists $X\in\X_{\nu}$ such that $F(X)=0$ and consider as in~\eqref{def:fourier_expansion} the functions $v$, $w$, $p$ and $s$. Assume also that the coefficients $\left(v_k\right)_{k\geq 0}$ and $\left(w_k\right)_{k\geq 0}$ are such that the functions $v$ and $w$ are positive. Define the function $u=pw$. Then $u$ and $v$ are smooth positive functions that solve~\eqref{eq:steady_states}.
\end{lemma}
\begin{proof}
First notice that since $X\in\X_\nu$ with $\nu>1$, the Fourier coefficients are decaying exponentially fast to 0, and thus the functions $v$, $w$, $p$ and $s$ are well defined and smooth (in fact analytic) $2$-periodic functions. Then, having $F(X)=0$ means exactly that the sequences $\left(v_k\right)_{k\geq 0}$, $\left(w_k\right)_{k\geq 0}$, $\left(p_k\right)_{k\geq 0}$ and $\left(s_k\right)_{k\geq 0}$ solve~\eqref{eq:algebraic_system}, which in turn implies that the functions $v$, $w$, $p$ and $s$ solve~\eqref{eq:steady_states_extended}. All the derivatives needed in~\eqref{eq:steady_states_extended} are legitimate thanks to the exponential decay of the coefficients. Besides, since $p$ satisfies the differential equation $p'+d_{12}p^2v'=0$ and $p(0)(d_1+d_{12}v(0))=1$, by uniqueness we have 
\begin{equation*}
p=\frac{1}{d_1+d_{12}v}.
\end{equation*}
Therefore $w=\frac{u}{p}=(d_1+d_{12}v)u$ and $(u,v)$ does indeed solve~\eqref{eq:steady_states} (the boundary condition for $u$ is also satisfied since $u'(0)=p'(0)w(0)+p(0)w'(0)=0$ and $u'(1)=p'(1)w(1)+p(1)w'(1)=0$).
\end{proof}

\subsection{Existence of steady states: the operators $\bm{A}$ and $\bm{A^\dag}$}\label{sec:fixed_point}

As outlined in Remark \eqref{rmk:thm}, the definition of the operators $A$ and $A^\dag$ is based on some approximate solution $\bar X$, which  is computed as numerical zero of a finite dimensional projection of $F(X)=0$.

Extending the notations introduced in Definition~\ref{def:truncation}, for $X\in\X_{\nu}$ we denote $\hat X^m$ the vector of truncated sequences, i.e. 
\begin{equation*}
\hat X^m=(\hat v^m,\hat w^m,\hat p^m,\hat s^m).
\end{equation*}
Similarly,
\begin{equation*}
\hat F^m = \left(\left(F^{(v)}_k\right)_{0\leq k<m},\left(F^{(w)}_k\right)_{0\leq k<m},\left(F^{(p)}_k\right)_{0\leq k<m},\left(F^{(s)}_k\right)_{0\leq k<m}\right).
\end{equation*}   
We consider $\hat F^m$ as acting on truncated sequences $\hat X^m$ only, so that we can see it as a function mapping $\R^{4m}$ to itself. Therefore, finding $\hat X^m$ such that $\hat F^m(\hat X^m)=0$ is a finite dimensional problem that can be solved numerically. We now assume to have computed numerically a zero of $\hat F^m$,  denoted by $\bar X$. 
 
The linear operator $A^{\dag}$ is defined as an approximation of $DF(\bar X)$.  However, since we will also need to construct   an approximate inverse of $A^{\dag}$,  $A^{\dag}$ is required to have a \emph{simple} structure. In practice we impose that $A^\dag$ acts diagonally on the {\it tail} $\{X_k\}_{k\geq m}$. More precisely, we define $A^{\dag}$ (acting on $X=(v,w,p,s)\in\X_{\nu}$), as
\begin{equation}\label{eq:Adag_m}
\widehat{A^{\dag}X}^m= D\hat F^m(\bar X)\hat X^m,
\end{equation}
and
\begin{equation*}
\left(A^{\dag}X\right)_k = \left(-\pi k v_k, -(\pi k)^{2} w_k, -\pi k p_k, d_2 \pi k s_k\right),\quad \forall~k\geq m.
\end{equation*}

The operator $A$ is then constructed  as an approximate inverse of $A^{\dag}$. We consider $\hat A^m$ a numerically computed inverse of $D\hat F^m(\bar X)$ and define $A$ (acting on $X=(v,w,p,s)\in\X_{\nu}$), as
\begin{equation*}
\widehat{AX}^m= \hat A^m\hat X^m,
\end{equation*}
and
\begin{equation*}
\left(AX\right)_k = \left(-(\pi k)^{-1} v_k, -(\pi k)^{-2} w_k, -(\pi k)^{-1} p_k, (d_2 \pi k)^{-1} s_k\right),\quad \forall~k\geq m.
\end{equation*}

The definition of $A$ and the fact that $\ell^1_\nu$ is a algebra for both convolution products $\ast$ and $\star$ (see Lemma~\ref{lem:convolution_algebra}) ensure that $AF$ does map $\X_\nu$ into itself, as requested in the hypothesis of Theorem~\ref{th:radii_pol}. 

\begin{remark}
\label{rem:diagonal}
To define the action  of $A^{\dag}$ on the tail space, we simply kept the asymptotically dominant terms of the derivative $DF(\bar X)$. Since these terms act  diagonally, we are able to easily and analytically  invert the tail of $A^{\dag}$ and hence to define $A$. However, the fact that the dominant terms of the derivative are diagonal is not a mere happenstance, rather it is the result of  the various reformulations performed in Section~\ref{sec:automatic_differentiation}. Had we not introduced the function $w$, the Fourier expansion of the cross-diffusion term would create a messy dominant expression that we would not be able to invert analytically.
\end{remark}

\subsection{Existence of steady states: the bounds $\bm{Y}$ and $\bm{Z_i(r)}$}
\label{sec:bounds_steady_states}

Having the Banach space $\X_\nu$, the function $F$, the approximate solution $\bar X=(\bar v,\bar w,\bar p,\bar s)$ and the operators  $A$, $A^{\dag}$ in hands, we now proceed to derive computable bounds $Y$, $Z_0$, $Z_1$ and $Z_2$ satisfying~\eqref{def:Y}-\eqref{def:Z_2} (for $\X=\X_\nu$).

\subsubsection{The bound $\bm{Y}$}
\label{sec:Y}
The definition of the bound $Y$  is rather straightforward, and we just consider
\begin{align}
\label{eq:Y_steady_states}
Y= \left\Vert AF(\bar X) \right\Vert_{\X_\nu}.
\end{align}
The key observation here is that $Y$  can be computed explicitly. Indeed, we recall that $\bar X\in \hat X^m$ is  a truncated sequence, i.e. $\bar X_k=(\bar v_k,\bar w_k,\bar p_k,\bar s_k)=(0,0,0,0)$ for all $k\geq m$. Therefore we have
\begin{align*}
&F_k^{(v)}(\bar X)=0,\quad \forall~k\geq m, \qquad &F_k^{(w)}(\bar X)=0,\quad \forall~k\geq 4m-3,\\
&F_k^{(p)}(\bar X)=0,\quad \forall~k\geq 3m-2, \qquad &F_k^{(s)}(\bar X)=0,\quad \forall~k\geq 3m-2,
\end{align*}
and thus $F(\bar X)$ only has a finite number of non zero coefficients. This is also true for $AF(\bar X)$ (thanks to the diagonal structure of the tail of $A$), and therefore $\left\Vert AF(\bar X) \right\Vert_{\X_\nu}$ can be evaluated on a computer. To be completely precise, what we mean by~\eqref{eq:Y_steady_states} is that a (sharp) upper bound of $\left\Vert AF(\bar X) \right\Vert_{\X_\nu}$ can be computed using interval arithmetic, and that we define $Y$ to be this upper bound. We are going to repeat this abuse of language whenever we define bounds that involve terms that have to be evaluated on a computer.

\subsubsection{The bound $\bm{Z_0}$}
\label{sec:Z0}

In this section we focus on getting a bound $Z_0$ satisfying~\eqref{def:Z_0}. Here and thereafter, when dealing with linear operators on $\X_\nu$, it is convenient to use a \emph{block notation}. For a linear operator $B:\X_\nu\to\X_\nu$, we consider the decomposition
\begin{equation*}
B=\begin{pmatrix}
B^{(v,v)} & B^{(v,w)} & B^{(v,p)} & B^{(v,s)} \\
B^{(w,v)} & B^{(w,w)} & B^{(w,p)} & B^{(w,s)} \\
B^{(p,v)} & B^{(p,w)} & B^{(p,p)} & B^{(p,s)} \\
B^{(s,v)} & B^{(s,w)} & B^{(s,p)} & B^{(s,s)}
\end{pmatrix},\quad {\rm each}\  B^{(i,j)}:\ell^1_\nu(\R)\to \ell^1_\nu(\R)
\end{equation*}
so that, for $X=(v,w,p,s)\in\X_\nu$,
\begin{equation*}
\left(BX\right)^{(v)}= B^{(v,v)}v + B^{(v,w)}w + B^{(v,p)}p + B^{(v,s)}s,
\end{equation*}
and similarly for the other components. Thus, recalling \eqref{eq:normXnu} and the operator norm \eqref{eq:oper_norm}, %Using this bock notation, we have that
\begin{align}\label{eq:norm_B_Xnu}
\left\Vert BX \right\Vert_{\X_\nu} &= \left\Vert \left(BX\right)^{(v)} \right\Vert_{\nu} + \left\Vert \left(BX\right)^{(w)} \right\Vert_{\nu} + \left\Vert \left(BX\right)^{(p)} \right\Vert_{\nu} + \left\Vert \left(BX\right)^{(s)} \right\Vert_{\nu} \\
&\leq \Theta_B^{(v)} \left\Vert v\right\Vert_\nu + \Theta_B^{(w)} \left\Vert w\right\Vert_\nu + \Theta_B^{(p)} \left\Vert p\right\Vert_\nu + \Theta_B^{(s)} \left\Vert s\right\Vert_\nu \\
&\leq \max\left[\Theta_B^{(v)},\Theta_B^{(w)},\Theta_B^{(p)},\Theta_B^{(s)}\right]\left\Vert X\right\Vert_{\X_\nu},
\end{align}
where
\begin{equation*}
\Theta_B^{(i)} = \VERT B^{(v,i)}\VERT_\nu + \VERT B^{(w,i)}\VERT_\nu + \VERT B^{(p,i)}\VERT_\nu + \VERT B^{(s,i)}\VERT_\nu, \quad \forall i\in \{v,w,p,s\}.
\end{equation*}
Therefore, we define
\begin{equation}
\label{eq:Z0_steady_states}
Z_0=\max\left[\Theta_{I-AA^{\dag}}^{(v)},\Theta_{I-AA^{\dag}}^{(w)},\Theta_{I-AA^{\dag}}^{(p)},\Theta_{I-AA^{\dag}}^{(s)}\right].
\end{equation}
Notice that, since the tail part of $A$ and $A^{\dag}$ are exact inverse of each other by definition, the tail part of $I-AA^{\dag}$ is zero. Therefore, each block in the decomposition of $I-AA^{\dag}$ has only finitely many non zero coefficients, and each $\Theta_{I-AA^{\dag}}^{(i)}$ can be computed using Lemma~\ref{lem:operator_norm_nu}, the supremum and the sum ranging only over finitely many coefficients.

\subsubsection{The bound $\bm{Z_1}$}
\label{sec:Z1}

In this section we focus on getting a bound $Z_1$ satisfying~\ref{def:Z_1}.
\begin{lemma}
Let $\hat\alpha^m_v,\hat\alpha^m_w,\hat\alpha^m_p,\hat\alpha^m_s$ be vectors in $\R^{4m}$ each, defined as

\begin{align*}
(\hat\alpha^m_v)_0 &= \begin{pmatrix}
0 \\
\Phi_0^m(-b_1(\bar p\ast\bar w),\nu) \\
\left\vert d_{12}\left(\bar p_0+2\sum_{k\geq 1}\bar p_k\right)\right\vert\frac{2}{\nu^m} \\
\Phi_0^m(-2a_2\bar v -b_2(\bar p\ast\bar w),\nu)
\end{pmatrix},\quad  (\hat\alpha^m_w)_0 =
\begin{pmatrix}
0 \\
\Phi_0^m(r_1\bar p-2a_1(\bar p\ast\bar p\ast\bar w)-b_1(\bar p\ast\bar v),\nu) \\
0 \\
\Phi_0^m(-b_2(\bar p\ast\bar v),\nu)
\end{pmatrix}, \\
(\hat\alpha^m_p)_0&=  \begin{pmatrix}
0 \\
\Phi_0^m(r_1\bar w-2a_1(\bar p\ast\bar w\ast\bar w)-b_1(\bar v\ast\bar w),\nu) \\
\left\vert d+d_{12}\left(\bar v_0+2\sum_{k\geq 1}\bar v_k\right)\right\vert\frac{2}{\nu^m} \\
\Phi_0^m(-b_2(\bar v\ast\bar w),\nu)
\end{pmatrix}, \quad (\hat\alpha^m_s)_0  =
\begin{pmatrix}
0 \\
0 \\
\Phi_0^m(d_{12}(\bar p\ast\bar p),\nu) \\
0
\end{pmatrix}
\end{align*}
and for  each $1\leq k<m$,
\begin{align*}
(\hat\alpha^m_v)_k&= \begin{pmatrix}
0 \\
\Phi_k^m(-b_1(\bar p\ast\bar w),\nu) \\
0 \\
\Phi_k^m(-2a_2\bar v -b_2(\bar p\ast\bar w),\nu)
\end{pmatrix}, \quad (\hat\alpha^m_w)_k= 
\begin{pmatrix}
0 \\
\Phi_k^m(r_1\bar p-2a_1(\bar p\ast\bar p\ast\bar w)-b_1(\bar p\ast\bar v),\nu) \\
0 \\
\Phi_k^m(-b_2(\bar p\ast\bar v),\nu)
\end{pmatrix}  \\
(\hat\alpha^m_p)_k&= \begin{pmatrix}
0 \\
\Phi_k^m(r_1\bar w-2a_1(\bar p\ast\bar w\ast\bar w)-b_1(\bar v\ast\bar w),\nu) \\
\Phi_k^m(2d_{12}(\bar s\star\bar p),\nu) \\
\Phi_k^m(-b_2(\bar v\ast\bar w),\nu)
\end{pmatrix}, \quad (\hat\alpha^m_s)_k=
\begin{pmatrix}
0 \\
0 \\
\Phi_k^m(d_{12}(\bar p\ast\bar p),\nu) \\
0
\end{pmatrix}.
\end{align*}
Define
\begin{align}
\label{eq:Z1_steady_states}
Z_1 &= \max\left[\left\Vert \vert  A\vert \hat\alpha^m_v \right\Vert_{\X_\nu},\left\Vert \vert  A\vert \hat\alpha^m_w \right\Vert_{\X_\nu},\left\Vert \vert  A\vert \hat\alpha^m_p \right\Vert_{\X_\nu},\left\Vert \vert  A\vert \hat\alpha^m_s \right\Vert_{\X_\nu}\right] \nonumber\\
&\quad +\max\left[\left(\frac{\left\Vert b_1(\bar p\ast \bar w)\right\Vert_{\nu}}{(\pi m)^2}+\frac{\left\Vert r_2-2a_2\bar v-b_2(\bar p\ast\bar w)\right\Vert_{\nu}}{d \pi m}\right), \right. \nonumber\\
&\qquad \qquad \left(\frac{\left\Vert r_1\bar p-2a_1(\bar p\ast\bar p\ast\bar w)-b_1(\bar p\ast\bar v)\right\Vert_{\nu}}{(\pi m)^2} +\frac{\left\Vert b_2(\bar p\ast\bar v)\right\Vert_{\nu}}{d \pi m}\right), \nonumber\\
&\qquad \qquad  \left(\frac{\left\Vert r_1\bar w-2a_1(\bar p\ast\bar w\ast\bar w)-b_1(\bar v\ast\bar w)\right\Vert_{\nu}}{(\pi m)^2} +\frac{\left\Vert 2d_{12}(\bar s\star\bar p)\right\Vert_{\nu}}{\pi m} + \frac{\left\Vert b_2(\bar v\ast\bar w)\right\Vert_{\nu}}{d \pi m}\right), \nonumber\\
&\left. \qquad \qquad \left(\frac{1}{\pi m} +\frac{\left\Vert d_{12}(\bar p\ast\bar p)\right\Vert_{\nu}}{\pi m}\right) \right].
\end{align}

Then
$$
Z_1\geq \left\VERT A\left(DF(\bar X)-A^{\dag}\right) \right\VERT_{\X_\nu}.
$$
\end{lemma}
\proof
According to~\eqref{def:Z_1}, we need a bound for  
\begin{equation*}
A\left(DF(\bar X)-A^{\dag}\right)X,
\end{equation*}
for $X\in B_{\X_\nu}(0,1)$. Denoting $U=\left(DF(\bar X)-A^{\dag}\right)X$ and using the triangular inequality, we have
\begin{align*}
\left\Vert A\left(DF(\bar X)-A^{\dag}\right)X \right\Vert_{\X_\nu} &\leq \left\Vert \vert A\vert \vert U\vert \right\Vert_{\X_\nu} \\
&\leq \Vert \vert A\vert \vert \hat U^m\vert \Vert_{\X_\nu} + \Vert \vert A\vert \vert \check U^m\vert \Vert_{\X_\nu},
\end{align*} 
where here and in the sequel, the absolute values must be understood component-wise. We point out that, since $A$ is built as a finite dimensional block $\hat A^m$  (acting on $\hat U^m$) and a diagonal tail, it follows that $\vert A\vert \vert \hat U^m\vert=\vert \hat A^m\vert \vert \hat U^m\vert$ is a finite vector, (it has non zero components only for $k<m$), whereas 
$\vert A\vert \vert \check U^m\vert$ has non zero components only for $k\geq m$. We  provide a bound for both terms separately.

At first, let us compute a bound on  
$\vert \hat U^m\vert$. Recalling from \eqref{eq:Adag_m} that $A^\dag$ is defined so   that  
\begin{equation*}
\widehat{A^{\dag}X}^m=D\hat F^m(\bar X)\hat X^m,
\end{equation*} 
it follows that  in computing $\hat U^m$ all the linear contributions of $X$ cancel out. Explicitly,  using Lemma~\ref{lem:convolution_truncation}, a meticulous though straightforward analysis gives
 \begin{equation*}
\vert \hat U^m\vert \leq (\hat\alpha^m_v) \left\Vert v\right\Vert_{\nu} + (\hat\alpha^m_w)\left\Vert w\right\Vert_{\nu} + (\hat\alpha^m_p) \left\Vert p\right\Vert_{\nu} + (\hat\alpha^m_s) \left\Vert s\right\Vert_{\nu}. %, \quad\forall~k<m,
\end{equation*} 
The vectors $\hat\alpha^m_i$ can each be seen as an element of $\R^{4m}$, or equivalently of $\X_\nu$ with coefficients equal to $0$ for all $k\geq m$. Inserting he previous inequality into  $\vert A\vert \vert \hat U^m\vert$, we have that
\begin{align}
\Vert \vert  A\vert \vert\hat U^m\vert \Vert_{\X_\nu} &\leq \left\Vert \vert  A\vert \hat\alpha^m_v \right\Vert_{\X_\nu}\left\Vert v\right\Vert_{\nu} + \left\Vert \vert  A\vert \hat\alpha^m_w \right\Vert_{\X_\nu}\left\Vert w\right\Vert_{\nu} + \left\Vert \vert  A\vert \hat\alpha^m_p \right\Vert_{\X_\nu}\left\Vert p\right\Vert_{\nu} + \left\Vert \vert  A\vert \hat\alpha^m_s \right\Vert_{\X_\nu}\left\Vert s\right\Vert_{\nu} \nonumber \\ 
& \leq \max\left[\left\Vert \vert  A\vert \hat\alpha^m_v \right\Vert_{\X_\nu},\left\Vert \vert  A\vert \hat\alpha^m_w \right\Vert_{\X_\nu},\left\Vert \vert  A\vert \hat\alpha^m_p \right\Vert_{\X_\nu},\left\Vert \vert  A\vert \hat\alpha^m_s \right\Vert_{\X_\nu}\right] \left\Vert X \right\Vert_{\X_\nu}, \label{eq:Z1_1}
\end{align}
the maximum being taken over terms that can all be evaluated on a computer.

For the tail part (i.e. for modes $k\geq m$), $A^{\dag}$ only cancels the diagonal dominant terms, and we have
\begin{align*}
\vert U_k\vert &\leq \begin{pmatrix}
0 \\
\left(\vert b_1 (\bar p\ast \bar w) \vert \ast \vert v\vert\right)_k \\
0 \\
\left(\vert r_2-2a_2\bar v-b_2(\bar p\ast\bar w)\vert \ast \vert v \vert\right)_k
\end{pmatrix} + 
\begin{pmatrix}
0 \\
\left(\vert r_1\bar p-2a_1(\bar p\ast\bar p\ast\bar w)-b_1(\bar p\ast\bar v)\vert \ast \vert w \vert\right)_k \\
0 \\
\left(\vert b_2(\bar p\ast\bar v)\vert \ast \vert w\vert\right)_k
\end{pmatrix} \\
& \quad + \begin{pmatrix}
0 \\
\left(\vert r_1\bar w-2a_1(\bar p\ast\bar w\ast\bar w)-b_1(\bar v\ast\bar w)\vert\ast\vert p\vert\right)_k \\
\left(\vert 2\alpha(\bar s\star\bar p)\vert \ast\vert p\vert\right)_k \\
\left(\vert b_2(\bar v\ast\bar w)\vert \ast \vert p\vert\right)_k
\end{pmatrix}  + 
\begin{pmatrix}
\vert s\vert_k \\
0 \\
\left(\vert d_{12}(\bar p\ast\bar p)\vert\star\vert s\vert\right)_k \\
0
\end{pmatrix}.
\end{align*}
Using Lemma~\ref{lem:convolution_algebra}, and the definition of the tail part of $A$, we get
\begin{align*}
\left\Vert \vert A\vert \vert \check U^m\vert \right\Vert_{\X_\nu} &\leq \left(\frac{\left\Vert b_1(\bar p\ast \bar w)\right\Vert_{\nu}}{(\pi m)^2}+\frac{\left\Vert r_2-2a_2\bar v-b_2(\bar p\ast\bar w)\right\Vert_{\nu}}{d \pi m}\right)\left\Vert v\right\Vert_{\nu} \\
&\quad + \left(\frac{\left\Vert r_1\bar p-2a_1(\bar p\ast\bar p\ast\bar w)-b_1(\bar p\ast\bar v)\right\Vert_{\nu}}{(\pi m)^2} +\frac{\left\Vert b_2(\bar p\ast\bar v)\right\Vert_{\nu}}{d \pi m}\right)\left\Vert w\right\Vert_{\nu} \\
&\quad + \left(\frac{\left\Vert r_1\bar w-2a_1(\bar p\ast\bar w\ast\bar w)-b_1(\bar v\ast\bar w)\right\Vert_{\nu}}{(\pi m)^2} +\frac{\left\Vert 2\alpha(\bar s\star\bar p)\right\Vert_{\nu}}{\pi m} + \frac{\left\Vert b_2(\bar v\ast\bar w)\right\Vert_{\nu}}{d \pi m}\right)\left\Vert p\right\Vert_{\nu} \\
&\quad + \left(\frac{1}{\pi m} +\frac{\left\Vert d_{12}(\bar p\ast\bar p)\right\Vert_{\nu}}{\pi m}\right)\left\Vert s\right\Vert_{\nu} \\
&\leq \max\left[\left(\frac{\left\Vert b_1(\bar p\ast \bar w)\right\Vert_{\nu}}{(\pi m)^2}+\frac{\left\Vert r_2-2a_2\bar v-b_2(\bar p\ast\bar w)\right\Vert_{\nu}}{d \pi m}\right), \right. \\
&\qquad \qquad \left(\frac{\left\Vert r_1\bar p-2a_1(\bar p\ast\bar p\ast\bar w)-b_1(\bar p\ast\bar v)\right\Vert_{\nu}}{(\pi m)^2} +\frac{\left\Vert b_2(\bar p\ast\bar v)\right\Vert_{\nu}}{d \pi m}\right), \\
&\qquad \qquad  \left(\frac{\left\Vert r_1\bar w-2a_1(\bar p\ast\bar w\ast\bar w)-b_1(\bar v\ast\bar w)\right\Vert_{\nu}}{(\pi m)^2} +\frac{\left\Vert 2d_{12}(\bar s\star\bar p)\right\Vert_{\nu}}{\pi m} + \frac{\left\Vert b_2(\bar v\ast\bar w)\right\Vert_{\nu}}{d \pi m}\right), \\
&\left. \qquad \qquad \left(\frac{1}{\pi m} +\frac{\left\Vert d_{12}(\bar p\ast\bar p)\right\Vert_{\nu}}{\pi m}\right) \right]\left\Vert X\right\Vert_{\X_\nu}.
\end{align*}
The sum of the latter estimate and \eqref{eq:Z1_1} provides the bound $Z_1$ \eqref{eq:Z1_steady_states}.
\qed

It is important to remark that all the $\Phi_k^m$ functions involved in the definition of $\hat\alpha^m_i$, $i\in\{v,w,p,s\}$,  take as arguments sequences that only have a finite number of non zero coefficients and that are  given in terms of the numerical guess $\bar X$. Therefore, all the vectors of coefficients  and the bound $Z_1$ can be rigorously and explicitly computed. 

\subsubsection{The bound $\bm{Z_2}$}

In this section we focus on defining a bound $Z_2$ satisfying~\eqref{def:Z_2}. 

\begin{lemma}
Consider the {\it block} decomposition of  $A$ and the associated coefficients $\Theta_A$, as introduced in Section~\ref{sec:Z0}.
Define the quantities
\begin{align*}
&\alpha_{v,v'}=2a_2\Theta_A^{(s)},\quad \alpha_{w,w'}=2a_1\left\Vert \bar p\ast\bar p\right\Vert_\nu \Theta_A^{(w)},\quad \alpha_{p,p'}=2a_1\left\Vert \bar w\ast\bar w\right\Vert_\nu \Theta_A^{(w)} + 2d_{12}\left\Vert \bar s\right\Vert_\nu \Theta_A^{(p)},\\
&\alpha_{v,w'}=b_1\left\Vert \bar p\right\Vert_\nu\Theta_A^{(w)} + b_2\left\Vert \bar p\right\Vert_\nu\Theta_A^{(s)},\quad \alpha_{v,p'}=b_1\left\Vert \bar w\right\Vert_\nu\Theta_A^{(w)} + d_{12}\Theta_A^{(p)} + b_2\left\Vert \bar w\right\Vert_\nu\Theta_A^{(s)},\\
&\alpha_{w,p'}=\left\Vert r_1-4a_1\bar w\ast \bar p-b_1\bar v\right\Vert_{\nu}\Theta_A^{(w)} + b_2\left\Vert \bar v\right\Vert_\nu\Theta_A^{(s)},\quad \alpha_{p,s'}=2d_{12}\left\Vert \bar p\right\Vert_\nu\Theta_A^{(p)},
\end{align*}
and
\begin{equation*}
\alpha_{v,w',p''} = b_1\Theta_A^{(w)} + b_2\Theta_A^{(s)},\quad 
\alpha_{w,w',p''} = 4a_1\left\Vert\bar p\right\Vert_{\nu}\Theta_A^{(w)},
\end{equation*}
\begin{equation*} 
\alpha_{w,p',p''} = 4a_1\left\Vert\bar w\right\Vert_{\nu}\Theta_A^{(w)},\quad
\alpha_{p,p',s''} = 2d_{12}\Theta_A^{(p)},
\end{equation*}
and
\begin{equation*}
\alpha_{w,,w',p'',p'''} = 4a_1\Theta_A^{(w)}.
\end{equation*}
Define 
\begin{align}
\label{eq:Z2_steady_states}
Z_2(r) &= \max\left[\alpha_{v,v'},\alpha_{w,w'},\alpha_{p,p'},\alpha_{v,w'},\alpha_{v,p'},\alpha_{w,p'},\alpha_{p,s'}\right] \nonumber\\
&\quad + \frac{1}{2}\max\left[\alpha_{v,w',p''},\alpha_{w,w',p''},\alpha_{w,p',p''},\alpha_{p,p',s''}\right] r \nonumber\\
&\quad + \frac{1}{6}\alpha_{w,w',p'',p'''} r^2.
\end{align}
Then 
$$
 rZ_2(r)\geq \left\VERT A\left(DF(X)-DF(\bar X)\right) \right\VERT_\X  \quad \forall~X\in\B_\X(\bar X,r).
$$
\end{lemma}
\proof
Consider the expansion:
\begin{align*}
A\left(DF(\bar X+X')-DF(\bar X)\right)X &= AD^2F(\bar X)(X',X) \\
&\quad + \frac{1}{2}AD^3F(\bar X)(X',X',X) \\
&\quad + \frac{1}{6}AD^4F(\bar X)(X',X',X',X).
\end{align*}
We provide bounds for each term on the right hand side, uniform for all $X\in B_{\X_\nu}(0,1)$ and $X'\in B_{\X_\nu}(0,r)$.

For the quadratic term, we have that
\begin{align*}
\left\Vert D^2F^{(v)}(\bar X)(X',X) \right\Vert_{\nu} = &  \,0,\\
\left\Vert D^2F^{(w)}(\bar X)(X',X) \right\Vert_{\nu} \leq &\, b_1\left\Vert \bar w\right\Vert_{\nu}\left(\left\Vert v\right\Vert_{\nu}\left\Vert p'\right\Vert_{\nu}+\left\Vert v'\right\Vert_{\nu}\left\Vert p\right\Vert_{\nu}\right)\\
& + b_1\left\Vert \bar p\right\Vert_{\nu}\left(\left\Vert v\right\Vert_{\nu}\left\Vert w'\right\Vert_{\nu}+\left\Vert v'\right\Vert_{\nu}\left\Vert w\right\Vert_{\nu}\right) \\
&+ 2a_1\left\Vert \bar w^2\right\Vert_{\nu}\left\Vert p\right\Vert_{\nu}\left\Vert p'\right\Vert_{\nu} + 2a_1\left\Vert \bar p^2\right\Vert_{\nu}\left\Vert w\right\Vert_{\nu}\left\Vert w'\right\Vert_{\nu} \\
& + \left\Vert r_1-4a_1\bar w\ast \bar p-b_1\bar v\right\Vert_{\nu} \left(\left\Vert w\right\Vert_{\nu}\left\Vert p'\right\Vert_{\nu}+\left\Vert w'\right\Vert_{\nu}\left\Vert p\right\Vert_{\nu}\right)\\
\left\Vert D^2F^{(p)}(\bar X)(X',X) \right\Vert_{\nu} \leq &\,  d_{12}\Big(2\left\Vert \bar s\right\Vert_{\nu}\left\Vert p\right\Vert_{\nu}\left\Vert p'\right\Vert_{\nu} + 2\left\Vert \bar p\right\Vert_{\nu}\left(\left\Vert p\right\Vert_{\nu}\left\Vert s'\right\Vert_{\nu}+\left\Vert p'\right\Vert_{\nu}\left\Vert s\right\Vert_{\nu}\right) \\
&\quad\quad  +\left\Vert v\right\Vert_{\nu}\left\Vert p'\right\Vert_{\nu}+\left\Vert v'\right\Vert_{\nu}\left\Vert p\right\Vert_{\nu}\Big),\\
\left\Vert D^2F^{(s)}(\bar X)(X',X) \right\Vert_{\nu} \leq &\, 2a_2\left\Vert v\right\Vert_{\nu}\left\Vert v'\right\Vert_{\nu} + b_2 \Big( \left\Vert \bar p\right\Vert_{\nu}\left(\left\Vert v\right\Vert_{\nu}\left\Vert w'\right\Vert_{\nu}+\left\Vert v'\right\Vert_{\nu}\left\Vert w\right\Vert_{\nu}\right)\\
&\phantom{\, 2a_2\left\Vert v\right\Vert_{\nu}\left\Vert v'\right\Vert_{\nu}+ b_2 \Big(} + \left\Vert \bar w\right\Vert_{\nu}\left(\left\Vert v\right\Vert_{\nu}\left\Vert p'\right\Vert_{\nu}+\left\Vert v'\right\Vert_{\nu}\left\Vert p\right\Vert_{\nu}\right) \\
&\phantom{\, 2a_2\left\Vert v\right\Vert_{\nu}\left\Vert v'\right\Vert_{\nu}+ b_2 \Big(}   + \left\Vert \bar v\right\Vert_{\nu}\left(\left\Vert w\right\Vert_{\nu}\left\Vert p'\right\Vert_{\nu}+\left\Vert w'\right\Vert_{\nu}\left\Vert p\right\Vert_{\nu}\right)\Big).
\end{align*}
According to \eqref{eq:norm_B_Xnu} and by rearrangements of the several terms, it follows
\begin{align*}
\left\Vert AD^2F(\bar X)(X',X)\right\Vert_{\X_\nu} &\leq \alpha_{v,v'}\left\Vert v\right\Vert_{\nu}\left\Vert v'\right\Vert_{\nu} + \alpha_{w,w'}\left\Vert w\right\Vert_{\nu}\left\Vert w'\right\Vert_{\nu} + \alpha_{p,p'}\left\Vert p\right\Vert_{\nu}\left\Vert p'\right\Vert_{\nu} \\
&\quad + \alpha_{v,w'}\left(\left\Vert v\right\Vert_{\nu}\left\Vert w'\right\Vert_{\nu}+\left\Vert v'\right\Vert_{\nu}\left\Vert w\right\Vert_{\nu}\right) + \alpha_{v,p'}\left(\left\Vert v\right\Vert_{\nu}\left\Vert p'\right\Vert_{\nu}+\left\Vert v'\right\Vert_{\nu}\left\Vert p\right\Vert_{\nu}\right) \\
&\quad + \alpha_{w,p'}\left(\left\Vert w\right\Vert_{\nu}\left\Vert p'\right\Vert_{\nu}+\left\Vert w'\right\Vert_{\nu}\left\Vert p\right\Vert_{\nu}\right) + \alpha_{p,s'}\left(\left\Vert p\right\Vert_{\nu}\left\Vert s'\right\Vert_{\nu}+\left\Vert p'\right\Vert_{\nu}\left\Vert s\right\Vert_{\nu}\right) \\
&\leq \max\left[\alpha_{v,v'},\alpha_{w,w'},\alpha_{p,p'},\alpha_{v,w'},\alpha_{v,p'},\alpha_{w,p'},\alpha_{p,s'}\right]\\
&\quad \times \left(\left\Vert v\right\Vert_{\nu}+\left\Vert w\right\Vert_{\nu}+\left\Vert p\right\Vert_{\nu}+\left\Vert s\right\Vert_{\nu}\right) \left(\left\Vert v'\right\Vert_{\nu}+\left\Vert w'\right\Vert_{\nu}+\left\Vert p'\right\Vert_{\nu}+\left\Vert s'\right\Vert_{\nu}\right)\\
&\leq \max\left[\alpha_{v,v'},\alpha_{w,w'},\alpha_{p,p'},\alpha_{v,w'},\alpha_{v,p'},\alpha_{w,p'},\alpha_{p,s'}\right] r\left\Vert X\right\Vert_{\X_\nu}.
\end{align*}
The same procedure  applied to the higher order derivative gives
\begin{align*}
\left\Vert AD^3F(\bar X)(X',X',X)\right\Vert_{\X_\nu} &\leq  \max\left[\alpha_{v,w',p''},\alpha_{w,w',p''},\alpha_{w,p',p''},\alpha_{p,p',s''}\right] r^2 \left\Vert X\right\Vert_{\X_\nu},
\end{align*}
and
\begin{align*}
\left\Vert AD^4F(\bar X)(X',X',X',X)\right\Vert_{\X_\nu} &\leq  \alpha_{w,w',p'',p'''} r^3 \left\Vert X\right\Vert_{\X_\nu}.
\end{align*}
Combining the above estimates, it follows that
$$
rZ_2(r)\geq \left\Vert A\left(DF(\bar X+X')-DF(\bar X)\right)X\right\Vert _{\X_\nu}, \quad \forall~X\in B_{\X_\nu}(0,1),\  \forall~X'\in B_{\X_\nu}(0,r).
$$
\qed

Notice that the computation of $\Theta_A^{(i)}$ requires the computation of $\left\VERT A^{(i,j)}\right\VERT_\nu$. Contrarily to the situation in Section~\ref{sec:Z0}, the tail of $A^{(i,j)}$ is not zero, in case $i=j$. However, since it has a diagonal structure, we can still explicitly compute the operator norm of each block of $A$. For instance
\begin{align*}
\left\VERT A^{(v,v)}\right\VERT_\nu &= \sup_{j\geq 0}\frac{1}{\nu^{j}}\sum_{k\geq 0}|A^{(v,v)}(k,j)|\nu^{k}\\
&= \max\left[\max_{0\leq j< m}\frac{1}{\nu^{j}}\sum_{0\leq k <m}|A^{(v,v)}(k,j)|\nu^{k},\sup_{j\geq m}\frac{1}{\nu^{j}}|-(\pi j)^{-1}|\nu^{j}\right] \\
&= \max\left[\max_{0\leq j< m}\frac{1}{\nu^{j}}\sum_{0\leq k <m}|A^{(v,v)}(k,j)|\nu^{k},\frac{1}{\pi m}\right].
\end{align*}

\subsection{Existence of steady States: the radii polynomial}

We now collect all the ingredients required to prove the existence of steady states  into a unique proposition.
\begin{proposition}
\label{prop:radii_pol} 
For $\nu>1$, let the space $\X_\nu=\big( \ell^1_\nu(\R)\big)^4$ be  endowed with the norm \eqref{eq:normXnu} and let $F$, $\bar X$, $A$, $A^{\dag}$  be as defined in section~\ref{sec:fun_F} and section~\ref{sec:fixed_point}. % to \ref{sec:framework_steady_states}. 
Let the bounds $Y$, $Z_0$, $Z_1$ and $Z_2$ be defined in~\eqref{eq:Y_steady_states}, \eqref{eq:Z0_steady_states}, \eqref{eq:Z1_steady_states} and \eqref{eq:Z2_steady_states} respectively, and rigorously computed. 
 \begin{itemize}
 \item[i)] If there exists $r>0$ such that
\begin{equation*}
P(r)=Z_2(r)r^2 -\left(1-\left(Z_0+Z_1\right)\right)r +Y<0,
\end{equation*}
then there exists a unique zero of $F$ in $\B_{\X_\nu}(\bar X,r)$.
\item[ii)] Let the functions $\bar v(x)$, $\bar w(x)$ and $\bar u(x)$ be 
$$
\bar w(x)=\bar w_0+2\sum_{k= 1}^{m-1}\bar w_k\cos(k\pi x), \quad  \bar v(x)=\bar v_0+2\sum_{k= 1}^{m-1}\bar v_k\cos(k\pi x).
$$
\begin{equation*}
\bar u(x)=\bar w(w)\bar p(x)=(\bar p\ast\bar w)_0+2\sum_{k= 1}^{2m-2}(\bar p\ast\bar w)_k\cos(k\pi x),
\end{equation*}
If $P(r)<0$ and $\inf_{x\in[0,1]}\bar w(x)-r>0$ and  $\inf_{x\in[0,1]}\bar v(x)-r>0$, then 
there exists of a smooth solution $(u(x),v(x))$ to~\eqref{eq:steady_states} so that
$$
|v(x)-\bar v(x)|<r, \quad |u(x)-\bar u(x)|<\left(\left\Vert\bar w\right\Vert_\nu+\left\Vert\bar p\right\Vert_\nu\right)r+\frac{r^2}{4}, \qquad \forall x\in[0,1].
$$
\end{itemize}

\end{proposition}
\begin{proof}
$i)$ The definition of the bounds  implies that the assumptions~\eqref{def:Y}-\eqref{def:Z_2} are satisfied. Theorem~\ref{th:radii_pol} then yields the existence and uniqueness of a zero for $F$ in $\B_{\X_\nu}(\bar X,r)$. The injectivity of $A$ follows for free from the fact that $P(r)<0$. Indeed it is necessary that $Z_0<1$ which means $\left\VERT I-AA^{\dag} \right\VERT_\X<1$. Note that the tail parts of $A$ and $A^\dag$ are analytically defined in a way that $1> \left\VERT I-AA^{\dag} \right\VERT_\X=\left\VERT \hat I^m-\hat A^m \hat{A^{\dag}}^m \right\VERT_\X$. The latter implies that both  $\hat A^m$ and $\hat{A^{\dag}}^m$ are  invertible, and thus $A$ is injective because its diagonal tail is made of non-zero coefficients.

$ii)$ Let $X=(v,w,p,s)$ be the unique zero of $F$ in   $\B_{\X_\nu}(\bar X,r)$. Thus $\|v-\bar v\|_\nu\leq r$ and $\|w-\bar w\|_\nu<r$. Since for any $a\in \ell^1_\nu$ we have that  $\|a\|_\nu\geq \sup_x|a(x)|$, it follows that $v(x)\geq \inf_x \bar v(x)-r>0$. The same holds for $w(x)$. For Lemma~\ref{lem:rigorous_justification} it follows  the existence of a smooth solution to~\eqref{eq:steady_states}. The error bound between $u$ and $\bar u$ is proven in Section~\ref{sec:c_j}.
\end{proof}

\section{Results about the existence of steady states}
\label{sec:results_steady_states}

In this section we present the computer-assisted proof of existence of steady states solutions stated in Theorem~\ref{th:exists_13}.
 
Each solution that is represented on Figure~\ref{fig:proved_diagram} was validated using the procedure described at the end of Section~\ref{sec:validated_numerics}. In particular, we computed each solution numerically, implemented the bounds described in Section~\ref{sec:bounds_steady_states}, and then \emph{successfully} applied Proposition~\ref{prop:radii_pol} to validate the numerical solution. By successfully we mean that we found a positive $r$ such that $P(r)<0$ and checked that $\inf_{x\in[0,1]}\bar w(x)-r>0$ and  $\inf_{x\in[0,1]}\bar v(x)-r>0$ (with the notation of Proposition~\ref{prop:radii_pol}).
The numerical data as well as the {\it Matlab} codes to perform the proofs and some documentation are available  at~\cite{website}. To make the computation of the bounds rigorous by controlling round-off errors, the interval arithmetic package {\it Intlab} \cite{Rum99} has been used. 
The computations presented here have been run on a laptop with a processor Intel Core i7 (2.50Ghz) and 8GB of RAM.

\medskip

\noindent{\it Proof of Theorem~\ref{th:exists_13}.}
In the script  \verb+script_proof_branch_steadystates.m+ fix the values of the parameters $r_1=5$, $r_2=2$, $a_1=3$, $a_2=3$, $b_1=1$, $b_2=1$, $d_{12}=3$. The parameter $d_1=d_2=d$ is intended as the bifurcation parameter. Choose a value for the finite dimensional projection $m$ and a value for the norm weight $\nu>1$. Also select  a branch of solutions (for the names of the several branches we refer to the documentation and the \verb+readme+ file). The script loads the numerical data, computes the required bounds and verifies the existence of an interval $\mathcal I=(r_1,r_2)$  such that $P(r)<0$ for any $r\in \mathcal I$. If $\mathcal I$ is not empty then the conditions $\inf_{x\in[0,1]}\bar w(x)-r>0$ and  $\inf_{x\in[0,1]}\bar v(x)-r>0$ are checked. In case of successful computation, Proposition~\ref{prop:radii_pol} implies the existence of the solutions.
The values for $m$ and $\nu$ that allow the rigorous computation of all the branches depicted in the Figure~\ref{fig:proved_diagram} are available in the documentation. 

The script \verb+script_proof_steadystate_and_instability.m+ concerns the existence of steady states for a fixed value of $d$. It is used to prove the existence of 13 solutions at values $d=0.005$. 
 Figure~\ref{fig:13solutions} shows the numerical data for the 13 steady states solutions. In Table~\ref{table:data_steadystates} we detail the values for  $m$ and  $\nu$ used in the proof and the resulting  validation radius $r$ (the script also aims at computing unstable eigenvalues, see Section~\ref{sec:results_instability}). 

\begin{figure}[h!]
\vspace{-0.5cm}
\centering
\subfigure[$v(0)=0.3973$]{\includegraphics[width=4cm]{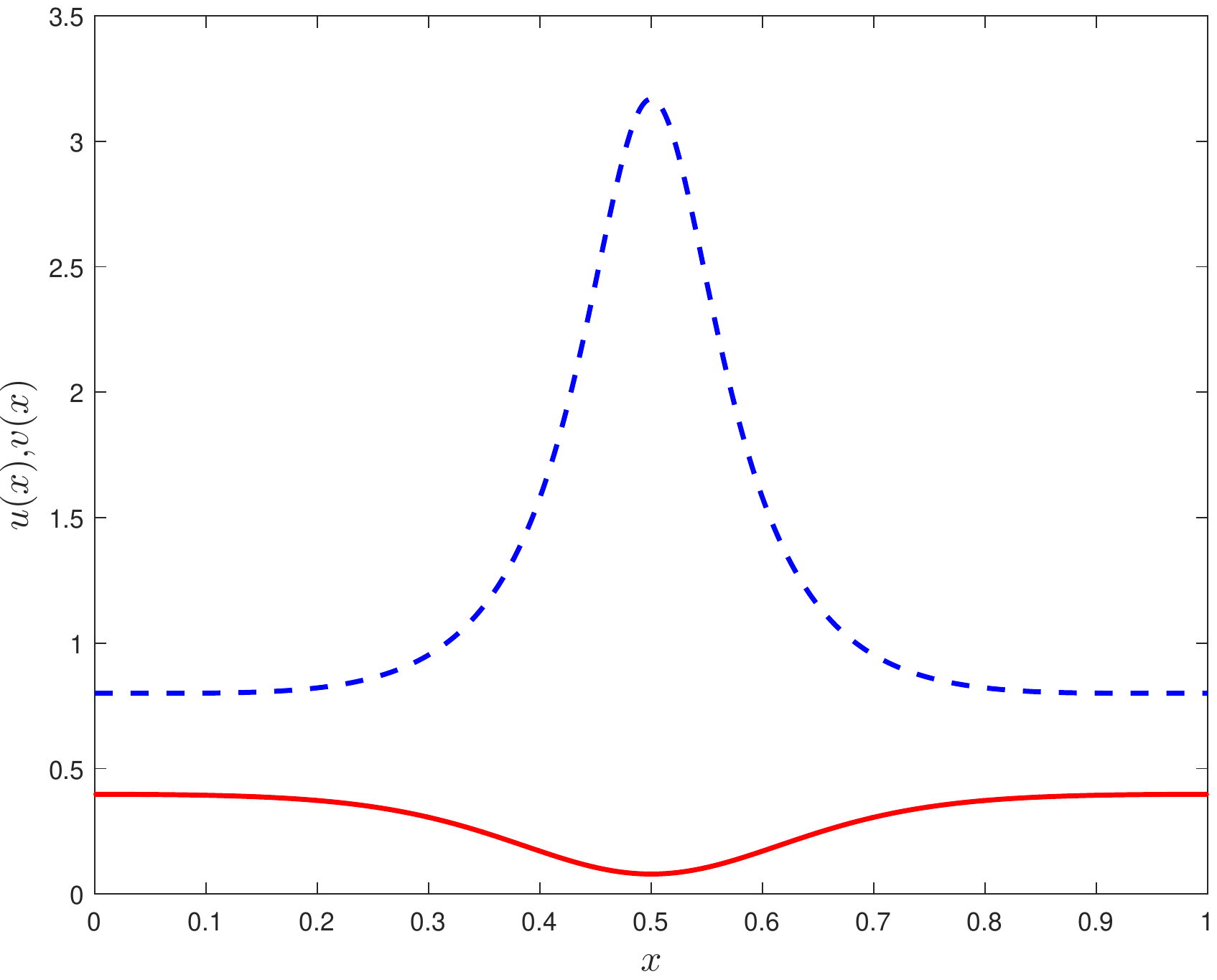}}
\subfigure[$v(0)=0.3632$]{\includegraphics[width=4cm]{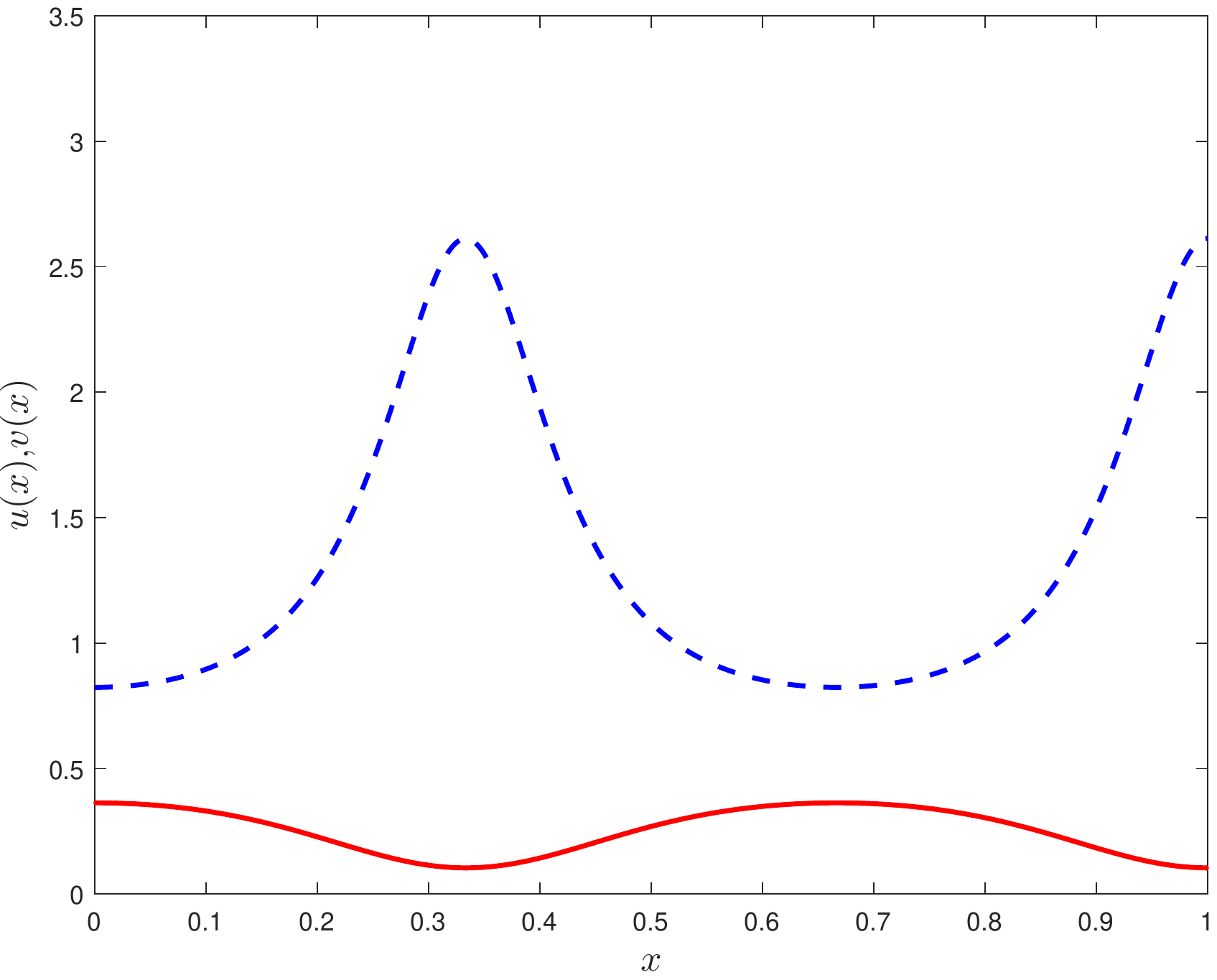}} 
\subfigure[$v(0)=0.3284$]{\includegraphics[width=4cm]{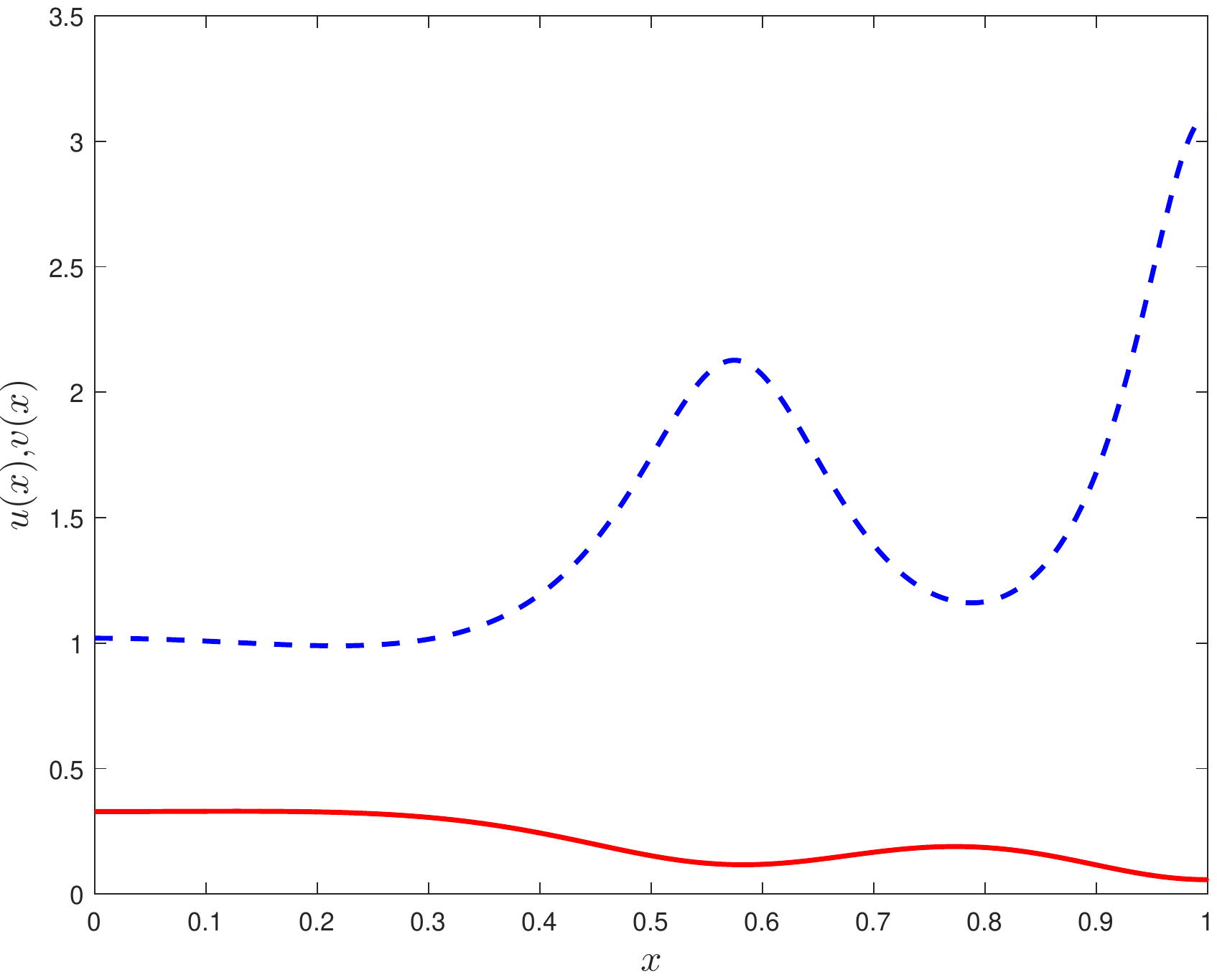}}
\subfigure[$v(0)=0.2721$]{\includegraphics[width=4cm]{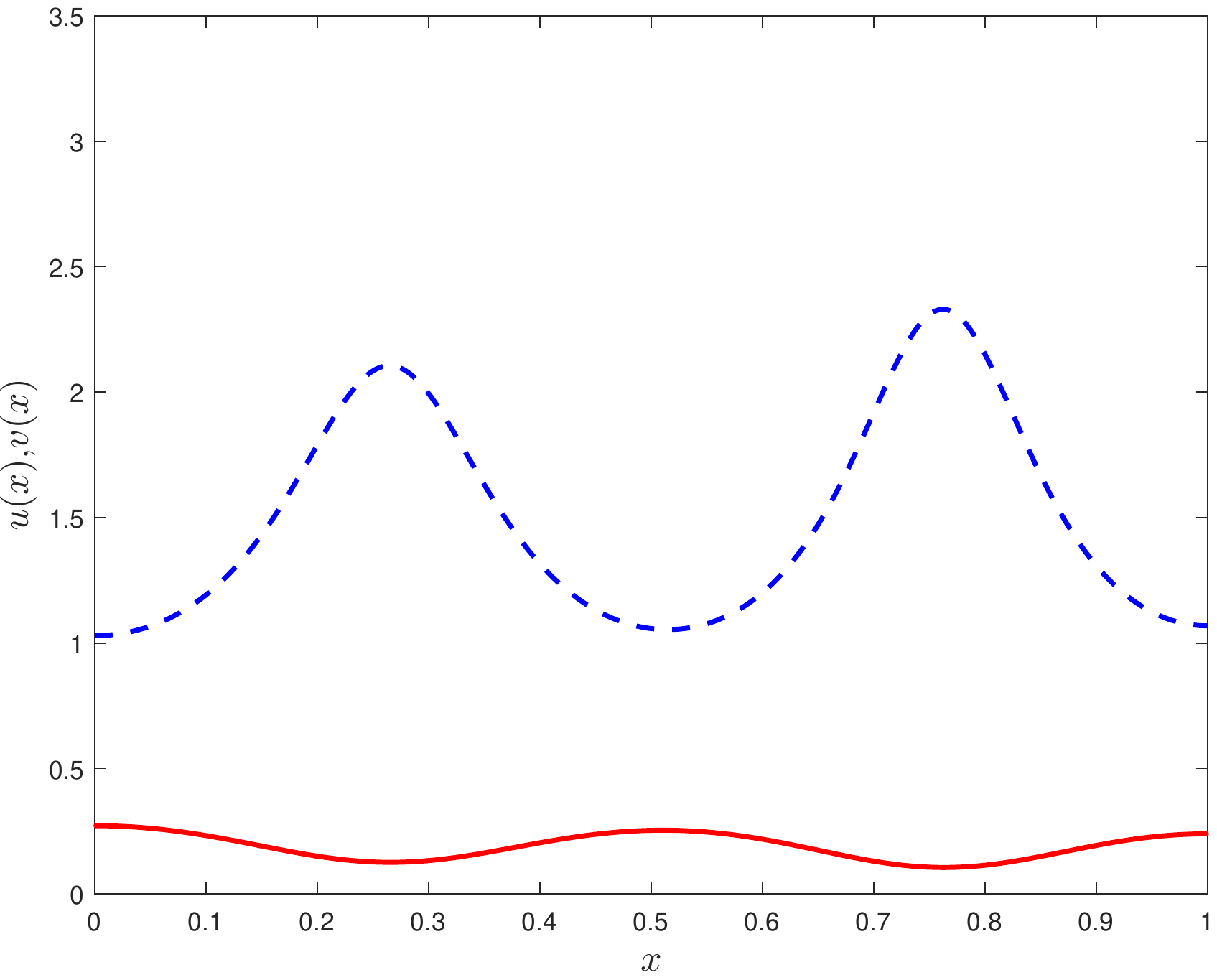}} 
\subfigure[$v(0)=0.2565$]{\includegraphics[width=4cm]{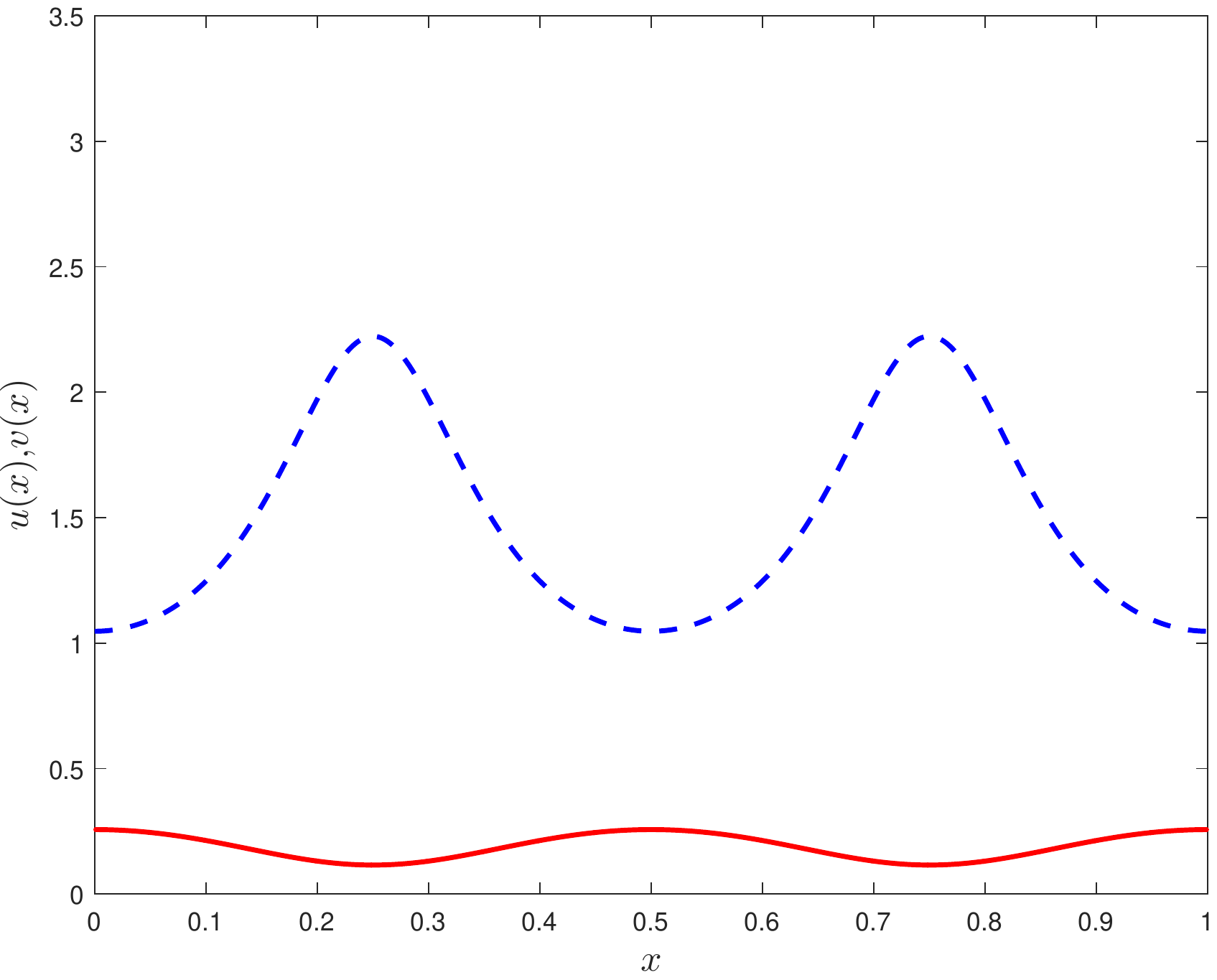}}
\subfigure[$v(0)=0.2397$]{\includegraphics[width=4cm]{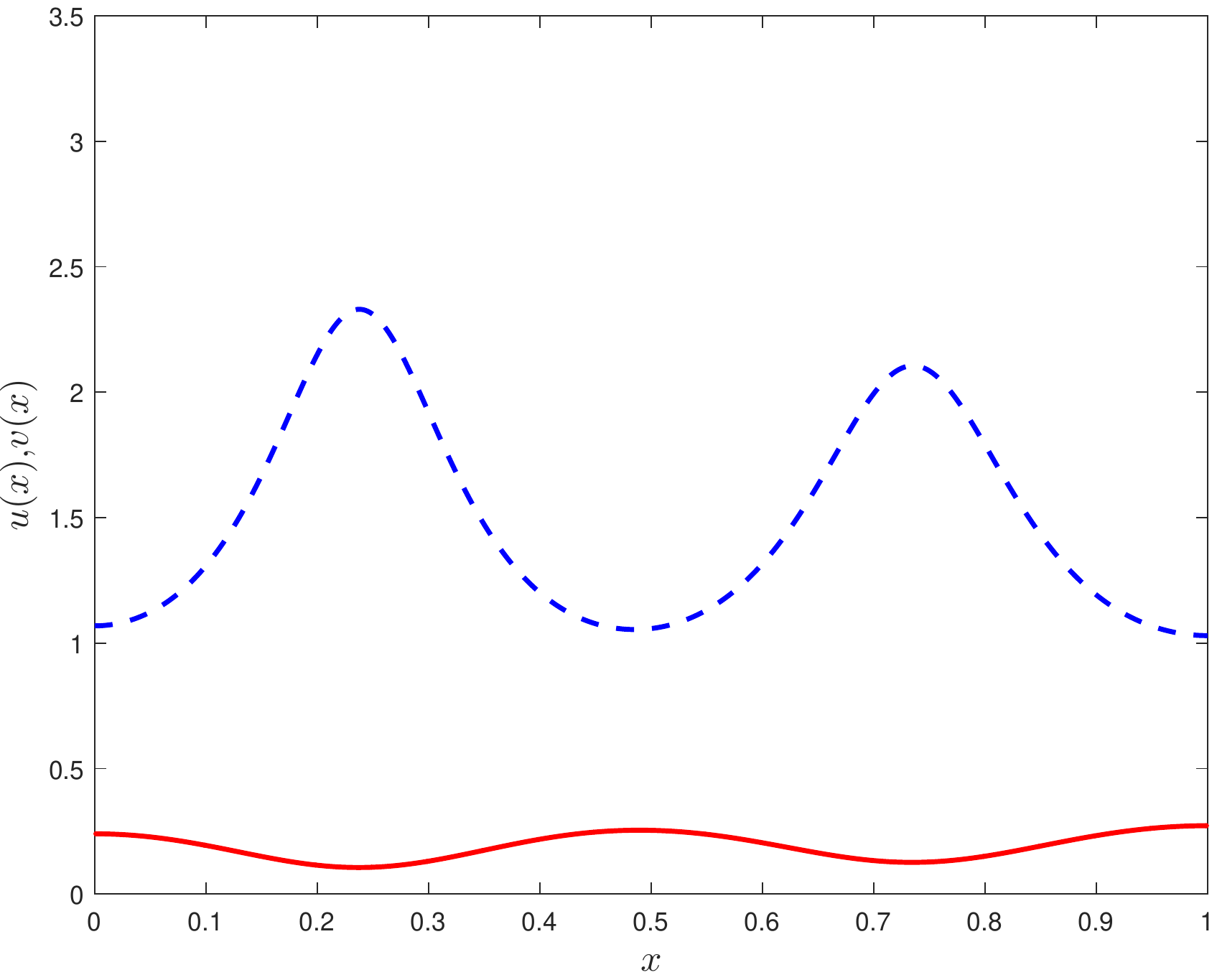}} 
\subfigure[$v(0)=0.1357$]{\includegraphics[width=4cm]{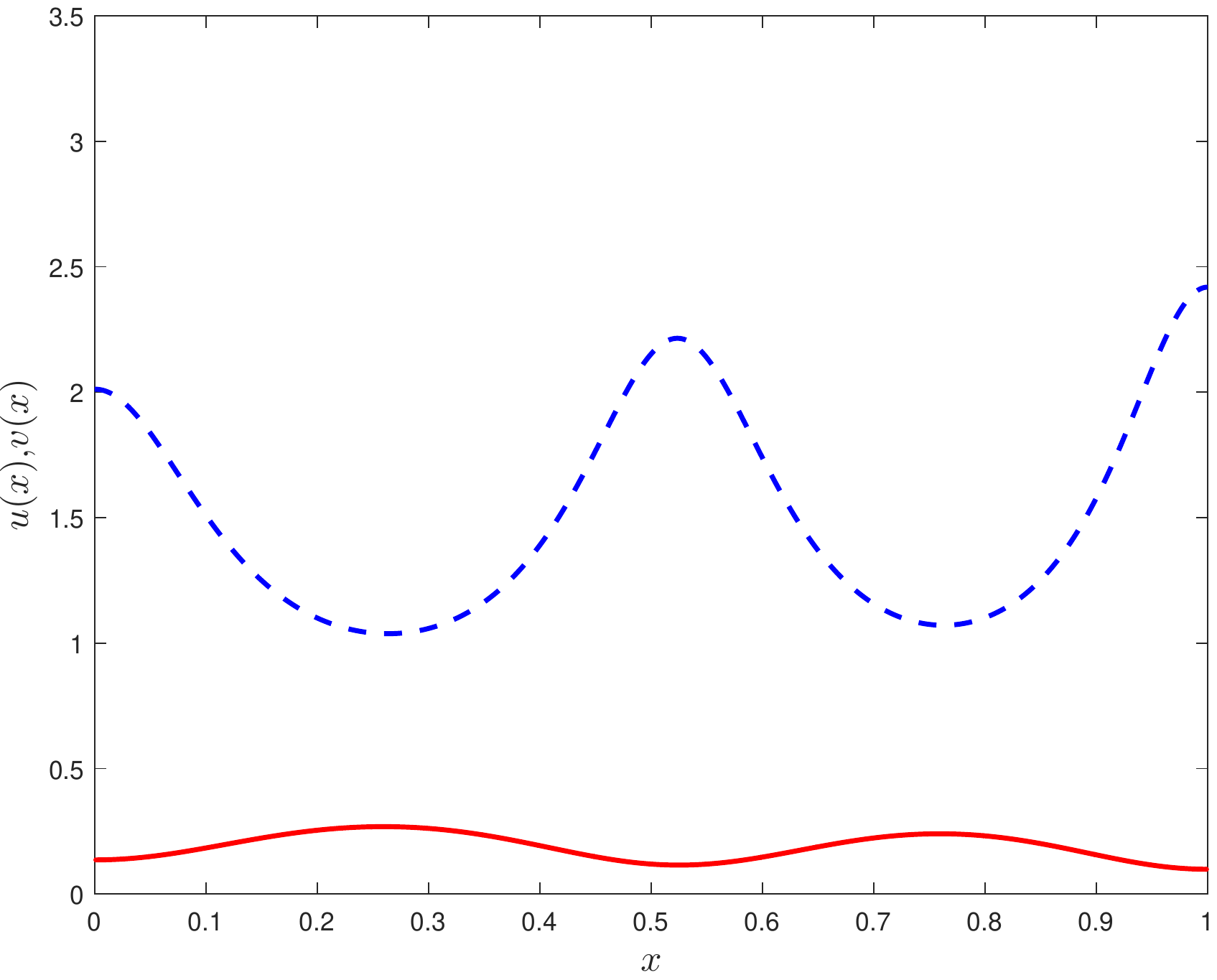}}
\subfigure[$v(0)=0.1250$]{\includegraphics[width=4cm]{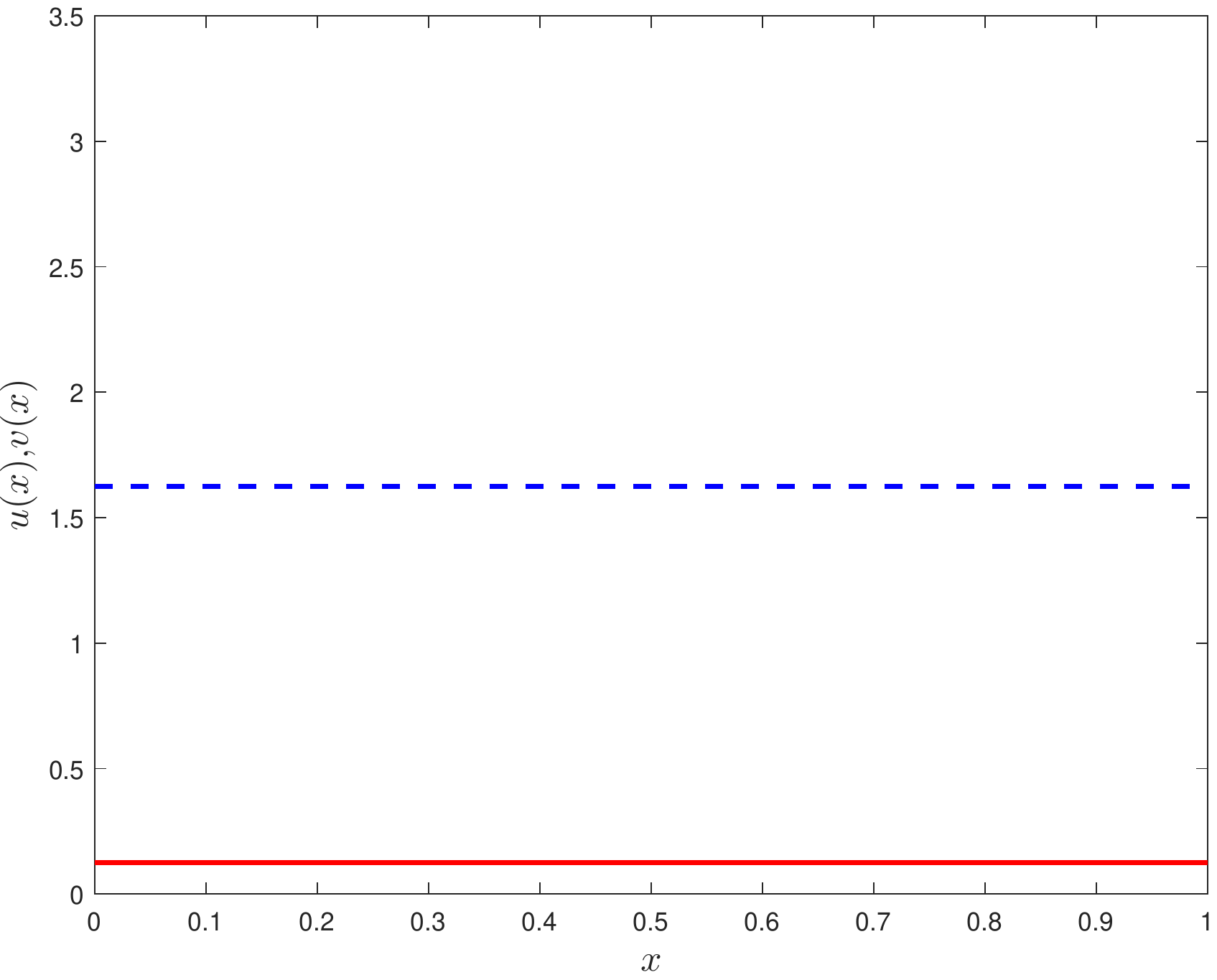}} 
\subfigure[$v(0)=0.11523$]{\includegraphics[width=4cm]{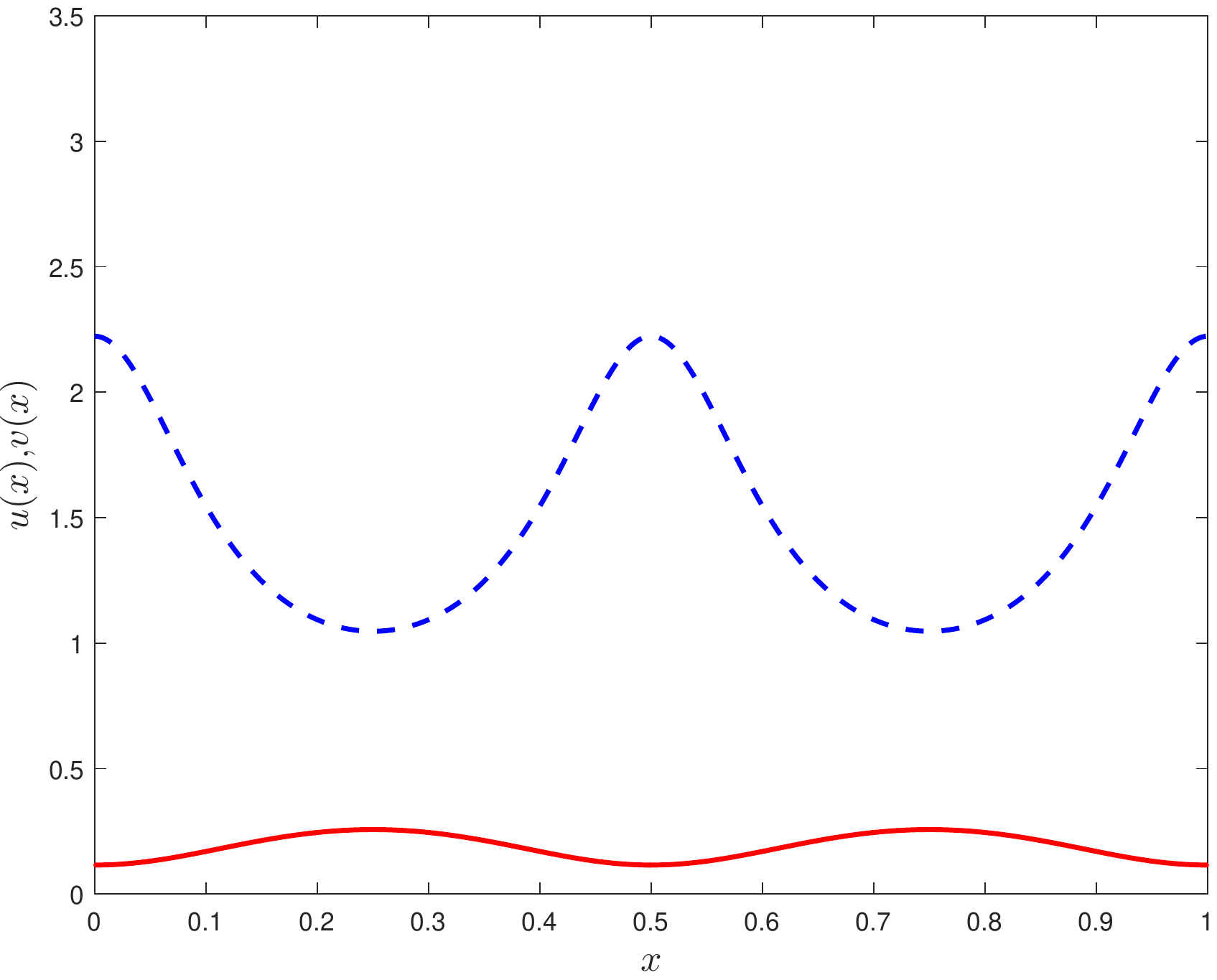}}
\subfigure[$v(0)=0.1040$]{\includegraphics[width=4cm]{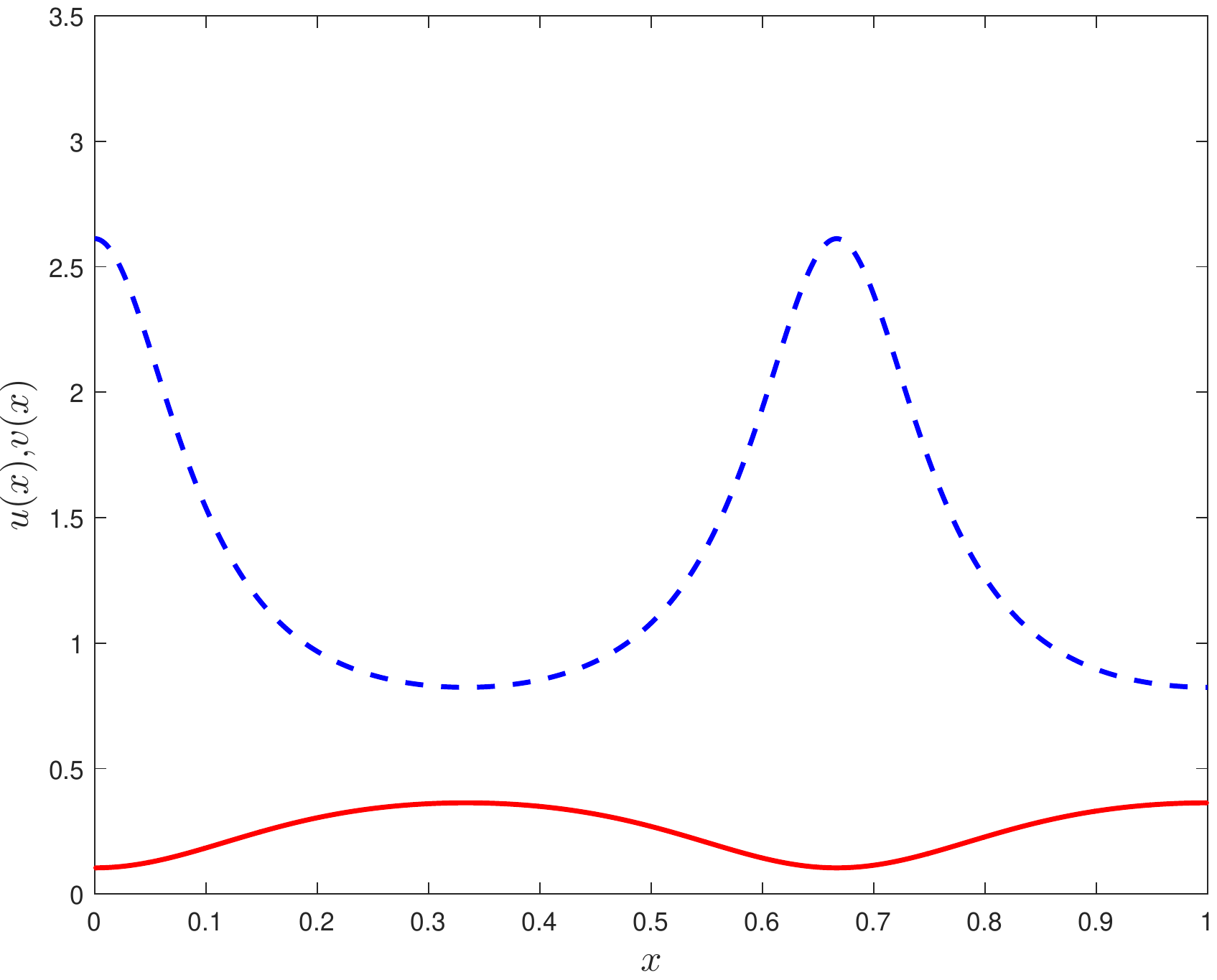}} 
\subfigure[$v(0)=0.0984$]{\includegraphics[width=4cm]{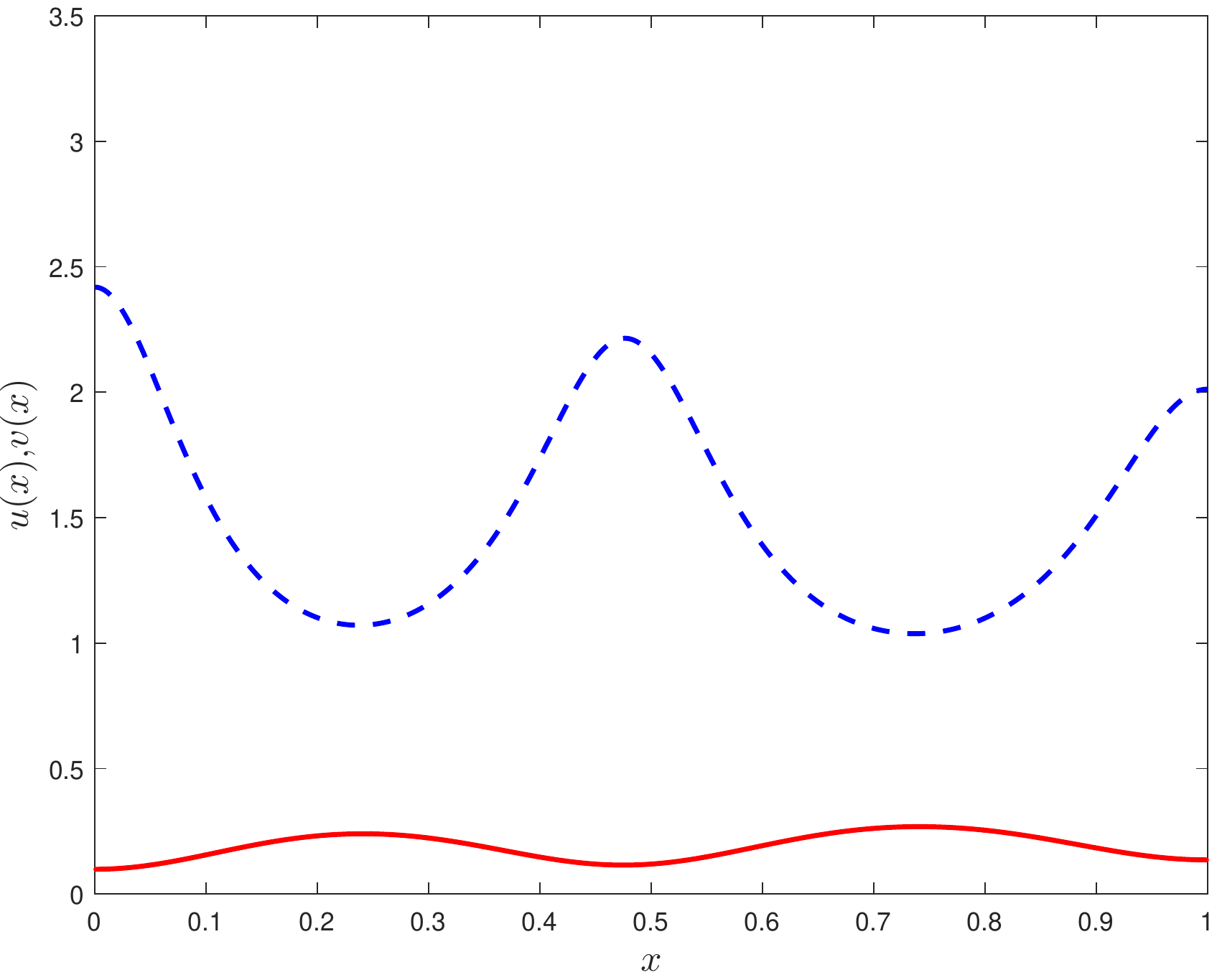}}
\subfigure[$v(0)=0.0793$]{\includegraphics[width=4cm]{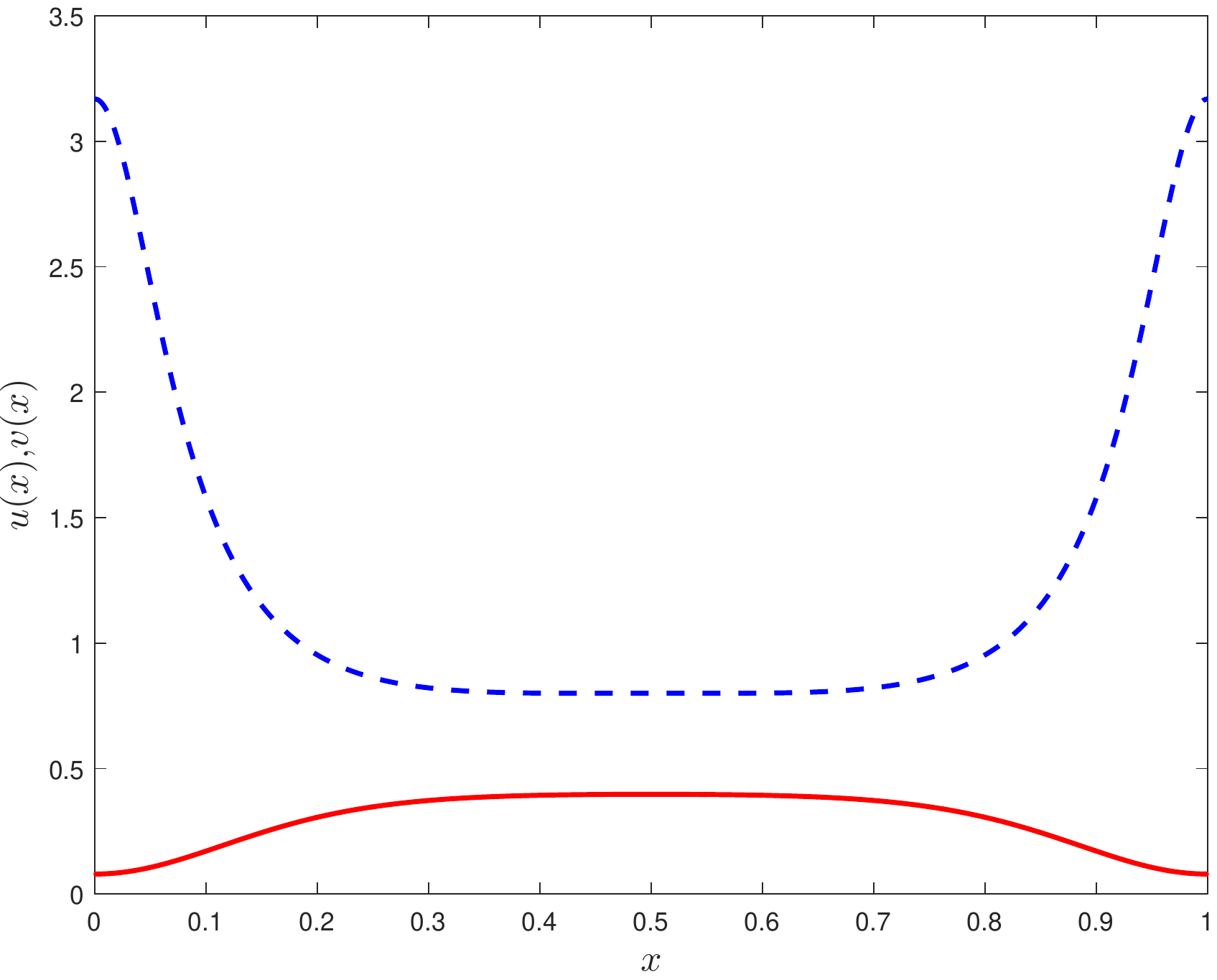}} 
\subfigure[$v(0)=0.0569$]{\includegraphics[width=4cm]{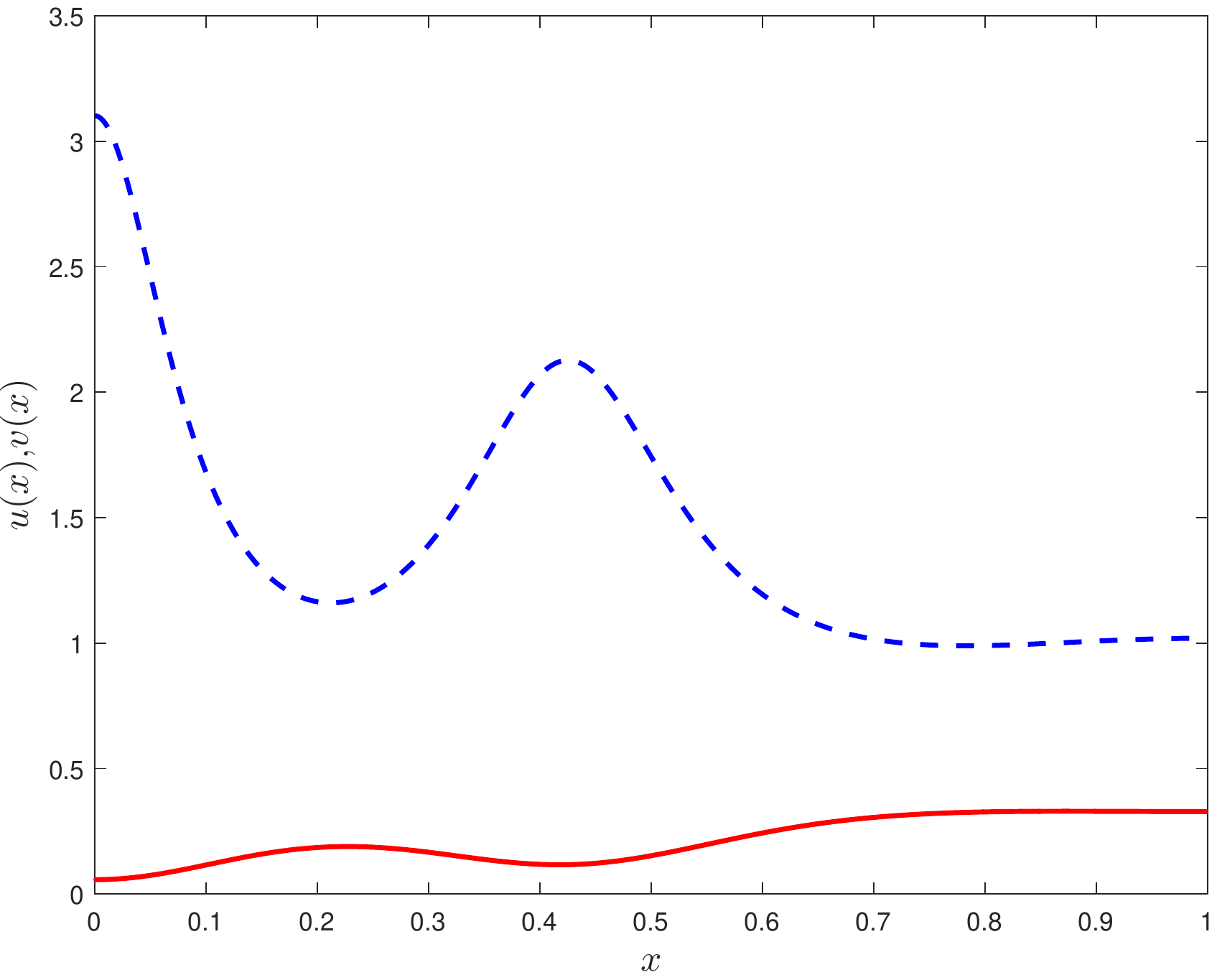}}
\vspace{-0.1cm}
\caption{The $13$ solutions announced in Theorem~\ref{th:exists_13} for $d=0.005$. They can be replaced on the bifurcation diagram using the value of $v(0)$. $u$ is represented in dashed blue, and $v$ in red. We give additional information about the proof for each of these solutions in Table~\ref{table:data_steadystates}, and discuss their stability in Section~\ref{sec:results_instability}.}
\label{fig:13solutions}
\end{figure}

\renewcommand{\arraystretch}{1.2}
\begin{table}[ht!]
\centering
\begin{tabular}{|c|c|c|c|}
\hline
Label for the solution (see Figure~\ref{fig:13solutions}) & $m$ used for the proof & $\nu$ used for the proof & Validation radius \\
\hline
(a) & 500 & 1.06 & $2.5968\times 10^{-11}$ \\
\hline
(b) & 500 & 1.06 & $9.8961\times 10^{-12}$ \\
\hline
(c) & 500 & 1.06 & $7.2076\times 10^{-11}$ \\
\hline
(d) & 500 & 1.06 & $7.8228\times 10^{-11}$ \\
\hline
(e) & 500 & 1.06 & $5.7165\times 10^{-12}$ \\
\hline
(f) & 500 & 1.06 & $1.0104\times 10^{-10}$ \\
\hline
(g) & 700 & 1.055 & $7.7146\times 10^{-11}$ \\
\hline
(h) & 500 & 1.06 & $2.9001\times 10^{-12}$ \\
\hline
(i) & 500 & 1.06 & $5.0578\times 10^{-12}$ \\
\hline
(j) & 500 & 1.06 & $6.4651\times 10^{-12}$ \\
\hline
(k) & 700 & 1.055 & $8.1462\times 10^{-11}$ \\
\hline
(l) & 500 & 1.06 & $1.6680\times 10^{-11}$ \\
\hline
(m) & 500 & 1.06 & $4.2505\times 10^{-11}$ \\
\hline
\end{tabular}
\caption{For each solution displayed in Figure~\ref{fig:13solutions}, we give the dimension $m$ that was used for the finite dimensional projection, the weight $\nu$ that was chosen for the space $\X_\nu$, and a validation radius $r$ for which the proof is successful, with those parameters $m$ and $\nu$.}
\label{table:data_steadystates}
\end{table}
\renewcommand{\arraystretch}{1}

\section{Framework for the instability of steady states}
\label{sec:framework_instability}

In this section, we focus on the stability of the steady states whose existence we proved in the Sections~\ref{sec:framework_steady_states} to \ref{sec:results_steady_states}. More precisely, we consider  %the evolution equation associated to~\eqref{eq:steady_states}, that is 
the cross-diffusion system~\eqref{eq:with_crossdiff} in the triangular case $d_{21}=0$ and with space dimension one ($\Omega=(0,1)$),
\begin{equation}
\label{eq:evolution}
\left\{
\begin{aligned}
&\frac{\partial u}{\partial t} = \frac{\partial^2}{\partial x^2}\left((d_1+d_{12} v)u\right)+(r_1-a_1u-b_1v)u, \quad &\text{on }\R_+\times(0,1),\\
&\frac{\partial v}{\partial t} = d_2\frac{\partial^2 v}{\partial x^2}+(r_2-b_2u-a_2 v)v, \quad &\text{on }\R_+\times(0,1),\\
&\frac{\partial u}{\partial x}(t,0)=0=\frac{\partial u}{\partial x}(t,1), \quad &\text{on }\R_+, \\
&\frac{\partial v}{\partial x}(t,0)=0=\frac{\partial v}{\partial x}(t,1), \quad &\text{on }\R_+.
\end{aligned}
\right.
\end{equation}
Let  $(u,v)=(u(x),v(x))$ be a steady state of~\eqref{eq:evolution}.
The linearization  of~\eqref{eq:evolution}, around the steady state $(u,v)$ yields the  eigenvalue problem
\begin{equation}
\label{eq:expanded_eigen}
\left\{
\begin{aligned}
&d_1\xi''+d_{12}\left(v''\xi+2v'\xi'+v\xi''+\eta'' u+2\eta' u' +\eta u''\right)+r_1\xi-2a_1u\xi-b_1(v\xi+\eta u)=\lambda \xi, \quad &\text{on }(0,1),\\
&d_2\eta''+r_2\eta-b_2(u\eta+\xi v)-2a_2 v\eta=\lambda \eta, \quad &\text{on }(0,1),\\
&\xi'(0)=\xi'(1)=0, \\
&\eta'(0)=\eta'(1)=0,
\end{aligned}
\right.
\end{equation}
where the functions $(\xi,\eta)=(\xi(x),\eta(x))$ form the  eigenfunction and $\lambda$ is the eigenvalue. As announced in the introduction, we aim at proving that most of the steady states obtained in Section~\ref{sec:results_steady_states} are  unstable, by showing that  the eigenproblem~\eqref{eq:expanded_eigen} admits an unstable eigenvalue, i.e. there exists a solution $\left((\xi,\eta),\lambda\right)$ of~\eqref{eq:expanded_eigen} such that $\Re(\lambda)>0$. The  approach is similar to the one used in Section~\ref{sec:framework_steady_states}: we first reformulate~\eqref{eq:expanded_eigen} into an equivalent problem more amenable to validated numerics, and then use Theorem~\ref{th:radii_pol} to prove the existence of an unstable eigenvalue. For this, we again follow the algorithm  outlined at the end of Section \ref{sec:validated_numerics}.

\subsection{Proof of instability: the function $\bm{F}$}\label{sec:F2}

As we did for the steady states, we look for the eigenfunctions $(\xi,\eta)$ as cosine series.  We point out that the steady state $(u,v)$ in~\eqref{eq:expanded_eigen} are non constant functions which have been obtained in Sections~\ref{sec:framework_steady_states} and~\ref{sec:results_steady_states}. If we directly expand~\eqref{eq:expanded_eigen} on the  Fourier basis, the dominant terms would not be diagonal. We take care of this issue by transforming~\eqref{eq:expanded_eigen} into an equivalent \emph{generalized eigenvalue problem} which has an autonomous second order terms. 

The two equations in~\eqref{eq:expanded_eigen} can be rewritten as
\renewcommand{\arraystretch}{1.2}
\begin{equation}
\label{eq:eigen_matrix}
M_1\begin{pmatrix}
\xi'' \\ \eta ''
\end{pmatrix}
+M_2\begin{pmatrix}
\xi' \\ \eta'
\end{pmatrix}
+M_3\begin{pmatrix}
\xi \\ \eta
\end{pmatrix}
=
\lambda\begin{pmatrix}
\xi \\ \eta
\end{pmatrix},
\end{equation}
where
\begin{equation*}
M_1=\begin{pmatrix}
d_1+d_{12}v & d_{12}u \\ 
0 & d_2
\end{pmatrix},\ 
M_2=\begin{pmatrix}
2d_{12} v' & 2d_{12} u'\\
0 & 0 
\end{pmatrix},\ 
M_3=\begin{pmatrix}
d_{12} v''+r_1-2a_1u-b_1v & d_{12} u''-b_1u\\
-b_2v & r_2-b_2u-2a_2 v 
\end{pmatrix}.
\end{equation*}
Introducing $p=\frac{1}{d_1+d_{12}v}$ as in Section~\ref{sec:framework_steady_states}, and knowing that $p(x)>0$ for any $x\in[0,1]$, we can express the inverse of $M_1$ as:
\begin{equation*}
M_1^{-1}=\begin{pmatrix}
\frac{1}{d_1+d_{12}v} & -\frac{d_{12} u}{d_2(d_1+d_{12} v)}\\
0 & \frac{1}{d_2}
\end{pmatrix}
=\begin{pmatrix}
p & -\frac{d_{12}up}{d_2}\\
0 & \frac{1}{d_2}
\end{pmatrix}.
\end{equation*}
We multiply~\eqref{eq:eigen_matrix} by $M_1^{-1}$ and obtain the following equivalent formulation for~\eqref{eq:expanded_eigen}:
\begin{equation}
\label{eq:eigen_with_c}
\left\{
\begin{aligned}
&\xi'' + c_1\xi' + c_2\eta' + c_3\xi + c_4\eta + c_5\lambda\xi + c_6\lambda\eta = 0, \quad &\text{on }(0,1),\\
&\eta'' + c_7\xi + c_8\eta + c_9\lambda\eta = 0, \quad &\text{on }(0,1),\\
&\xi'(0)=\xi'(1)=0, \\
&\eta'(0)=\eta'(1)=0,
\end{aligned}
\right.
\end{equation}
where the functions $\left(c_j\right)_{1\leq j \leq 9}$ depend on the steady state $(u,v)$ (and on the parameters of the cross-diffusion system), and are given by
\begin{equation*}
c_1=2d_{12}pv',\quad c_2=2d_{12}pu',\quad c_3=(r_1-2a_1u-b_1v+d_{12}v'')p +\frac{d_{12} b_2}{d_2}uvp,
\end{equation*}
\begin{equation*}
c_4=(d_{12}u''-b_1u)p-\frac{d_{12}}{d}up(r_2-b_2u-2a_2v),\quad c_5=-p,\quad c_6=\frac{d_{12}}{d_2}up,
\end{equation*}
\begin{equation}
\label{eq:def_cj}
c_7=-\frac{b_2}{d_2}v,\quad c_8=\frac{1}{d_2}(r_2-b_2u-2a_2v),\quad c_9=-\frac{1}{d_2}.
\end{equation}

\subsubsection{The algebraic system in Fourier space and the $\bm{F=0}$ formulation}

Expanding the eigenfunctions in cosine series
\begin{equation*}
\xi(x)=\xi_0+2\sum_{k\geq 1}\xi_k\cos(k\pi x), \quad \eta(x)=\eta_0+2\sum_{k\geq 1}\eta_k\cos(k\pi x),
\end{equation*}
and inserting these expansions in~\eqref{eq:eigen_with_c}, we end up with
\begin{equation}
\label{eq:algebraic_eigen}
\left\{
\begin{aligned}
&-(\pi k)^2 \xi_k-(c_1\bullet K\xi)_k-(c_2\bullet K\eta)_k+(c_3\ast\xi)_k+(c_4\ast\eta)_k+\lambda(c_5\ast\xi)_k+\lambda(c_6\ast\eta)_k=0 ,\quad &\forall~k\in\N,\\
&-(\pi k)^2 \eta_k+(c_7\ast\xi)_k+(c_8\ast\eta)_k+\lambda(c_9\ast\eta)_k=0 ,\quad &\forall~k\in\N.
\end{aligned}
\right.
\end{equation}

Again, we identify the functions $\xi$, $\eta$ and $c_j$ with their sequence of Fourier coefficients. 

\begin{remark}
\textbf{Reset of some notations.} To maintain the same   notations as in  Theorem~\ref{th:radii_pol}, we are going to redefine the appropriate $\X$, $X$, $F$, $\bar X$, $A$ and $A^{\dag}$ corresponding to the eigenproblem~\eqref{eq:eigen_with_c}. Henceforth, we forget about the definition of this symbols that was given in Section~\ref{sec:framework_steady_states}, and give new ones in the sequel.
\end{remark}

For $\gamma>1$, we define $\X_\gamma=\left(\ell^1_\gamma\right)^2\times\C$ and denote by $X=(\xi,\eta,\lambda)$ any element in $\X_\gamma$. We point out that  in Section~\ref{sec:framework_steady_states} we could restrict ourselves to real sequences because we were only looking for real solutions, but this is no longer the case here since we may encounter complex conjugate eigenvalues and eigenvectors. We endow $\X_\gamma$ with the norm
\begin{equation*}
\left\Vert X\right\Vert_{\X_\gamma} = \left\Vert \xi\right\Vert_\gamma + \left\Vert \eta\right\Vert_\gamma + \left\vert \lambda\right\vert,
\end{equation*}
which makes it a Banach space. We then fix an index $k_0\in\N$ and define the function $F=(F^{(\xi)},F^{(\eta)},F^{(\lambda)})$ acting on $\X_{\gamma}$ by
\begin{align*}
\label{eq:def_F_eigen}
&F_k^{(\xi)}(X)=-(\pi k)^2 \xi_k-(c_1\bullet K\xi)_k-(c_2\bullet K\eta)_k+(c_3\ast\xi)_k+(c_4\ast\eta)_k+\lambda(c_5\ast\xi)_k+\lambda(c_6\ast\eta)_k, &\forall~k\in \N, \nonumber\\
&F_k^{(\eta)}(X)=-(\pi k)^2 \eta_k+(c_7\ast\xi)_k+(c_8\ast\eta)_k+\lambda(c_9\ast\eta)_k, &\forall~k\in \N, \nonumber\\
&F^{(\lambda)}(X)=\xi_{k_0}-1.
\end{align*}
Notice that the only difference between $F(X)=0$ and system~\eqref{eq:algebraic_eigen} is the equation $\xi_{k_0}=1$. The role of this additional constraint is to normalise the eigenfunction and hence to isolate the potential solutions of $F$. Indeed, we can not hope to successfully use Theorem~\ref{th:radii_pol}, which is based on a contraction argument, if the zeros of $F$ are not isolated. We point out that many different conditions could have been added to isolate the solution, and that this specific choice is rather arbitrary.

We now state the precise link between $F$ and our stability problem.
\begin{lemma}
\label{lem:rigorous_eigen}
Assume that $(u,v)$ is a positive stationary solution of~\eqref{eq:evolution} and that there exists $\gamma>1$ such that the Fourier coefficients of the functions $\left(c_j\right)_{1\leq j\leq 9}$, defined in~\eqref{eq:def_cj}, belong to $\ell^1_\gamma(\R)$. Fix $k_0\in\N$ and suppose that there exists $X=(\xi,\eta,\lambda)\in\X_\gamma$, with $\Re(\lambda)>0$, such that $F(X)=0$. Then the steady state $(u,v)$ is linearly unstable.
\end{lemma}
\begin{proof}
As for Lemma~\ref{lem:rigorous_justification}, the proof just consists in checking that the regularity (i.e. the fact that the Fourier coefficients belongs to $\ell^1_\gamma$) of the solution $X$ and of the data $\left(c_j\right)_{1\leq j\leq 9}$ allows to rigorously backtrack the manipulations made to obtain $F$ from the eigenproblem~\eqref{eq:expanded_eigen}. 
\end{proof}
We point out that the assumption that $u$ and $v$ are positive, is in fact only needed here to ensure that $p$ is well defined. Concerning the assumption on the functions $\left(c_j\right)_{1\leq j\leq 9}$, the method developed in Sections~\ref{sec:framework_steady_states} to \ref{sec:results_steady_states} naturally provides us with steady states $(u,v)$ for which the Fourier coefficients of $u$, $v$ (and $p$) belong to $\ell^1_\nu$, for some $\nu>1$. However, since some of the $c_j$ involve derivatives of $u$ and $v$, we can only get that their Fourier coefficients belong to $\ell^1_\gamma$ for $\gamma<\nu$. We give more details and explicit estimates right below.

\subsubsection{About the functions $\bm{c_j}$}
\label{sec:c_j}

The function $F$ depends on the Fourier coefficients of $\left(c_j\right)_{1\leq j\leq 9}$, which themselves depend on the steady state $(u,v)$. The method described in Sections~\ref{sec:framework_steady_states} to \ref{sec:results_steady_states} (in particular Proposition~\ref{prop:radii_pol}) provides us with an approximate steady state, in the form of Fourier sequences $(\bar v,\bar w,\bar p, \bar s)$, together with a validation radius $r_\nu$ which gives an upper bound of the distance (in the $\ell^1_\nu$ norm) between the approximate steady state and the genuine one. 

\begin{remark}
From now on, we denote   $r_\nu$ the validation radius obtained in the computation of  the steady states. This validation radius was simply denoted by $r$ in Section~\ref{sec:framework_steady_states} and~\ref{sec:results_steady_states}, but this new notation should avoid possible confusions with the validation radius that we are going to consider for the eigenvalue problem.
\end{remark}

More precisely, the steady state (in the $(v,w,p,s)$ coordinates) is proved to exist in the form
\begin{equation*}
v=\bar v+\varepsilon_v,\quad w=\bar w+\varepsilon_w, \quad p=\bar p+\varepsilon_p,\quad s=\bar s+\varepsilon_s,
\end{equation*}
where the $(\bar v,\bar w,\bar p, \bar s)$ are finite Fourier sequences that we have explicitly on our computer, and $\left\Vert\varepsilon_v\right\Vert_\nu + \left\Vert\varepsilon_w\right\Vert_\nu + \left\Vert\varepsilon_p\right\Vert_\nu + \left\Vert\varepsilon_s\right\Vert_\nu\leq r_\nu$, where $r_\nu$ is the radius provided by the proof. % (for which the radii polynomial $P(r_\nu)$ is negative).

Therefore, we can also represent each $c_j$ as
\begin{equation*}
c_j=\bar c_j+ \varepsilon_j,
\end{equation*}
where $\bar c_j$ is a finite sequence of Fourier coefficients and $\left\Vert\varepsilon_j\right\Vert_\gamma \leq \epsilon_j(\gamma)$. In this section, we provide formulas for the finite sequences $\bar c_j$ and the upper bounds $\epsilon_j(\gamma)$ on the distance (in $\ell^1_\gamma$) between $\bar c_j$ and $c_j$ .

\medskip

The first step is to provide an enclosure for $u$  in the form  $u=\bar u + \varepsilon_u$. Remembering that $u(x)=p(x)w(x)$, and hence $u=p*w$, we have
\begin{equation*}
u=\bar p\ast\bar w + \bar p\ast\varepsilon_w + \varepsilon_p\ast \bar w + \varepsilon_p\ast \varepsilon_w.
\end{equation*}
Therefore, we can define $\bar u=\bar p\ast \bar w$, and $\varepsilon_u=\bar p\ast\varepsilon_w + \varepsilon_p\ast \bar w + \varepsilon_p\ast \varepsilon_w$. By Lemma~\ref{lem:convolution_algebra}, a bound for the norm of $\varepsilon_u$ is given as $\left\Vert\varepsilon_u\right\Vert_\nu\leq (\left\Vert\bar p\right\Vert_\nu+\left\Vert\bar w\right\Vert_\nu)r_\nu+\frac{r_\nu^2}{4}$, where we used that $\left\Vert\varepsilon_p\right\Vert_\nu \left\Vert\varepsilon_w\right\Vert_\nu \leq \frac{1}{4}(\left\Vert\varepsilon_p\right\Vert_\nu + \left\Vert\varepsilon_w\right\Vert_\nu)^2$. For further use, we define $\epsilon_u=(\left\Vert\bar p\right\Vert_\nu+\left\Vert\bar w\right\Vert_\nu)r_\nu+\frac{r_\nu^2}{4}$.

The second step is to derive estimates on derivatives. 
\begin{definition}
\label{def:K}
Denote by $K$ the (unbounded) linear operator such that, for all $z$ in $\ell^1_\nu$
\begin{equation*}
\left(Kz\right)_k = \pi kz_k,\quad \forall k\geq 0.
\end{equation*}
\end{definition}
Up to a change of sign (depending on whether we consider the cosine or sine expansion), $Kz$ is nothing but the sequence of Fourier coefficients of the derivative of the function $z(x)$.
\begin{lemma}
\label{lem:estimate_derivative}
Let $1<\gamma<\nu$ and $z\in\ell^1_\nu$. Then
\begin{equation*}
\left\Vert Kz\right\Vert_\gamma\leq \Upsilon_{\gamma,\nu}^1 \left\Vert z\right\Vert_\nu, \quad \left\Vert K^2z\right\Vert_\gamma\leq \Upsilon_{\gamma,\nu}^2 \left\Vert z\right\Vert_\nu,
\end{equation*}
where
\begin{equation*}
\Upsilon_{\gamma,\nu}^1 = \left\{
\begin{aligned}
&\frac{\gamma}{\nu},\quad &\text{if }\gamma<e^{-1}\nu\\
&\frac{e^{-1}}{\ln\frac{\nu}{\gamma}},\quad &\text{otherwise,}
\end{aligned}
\right.
\qquad\qquad\text{and}\qquad\qquad
\Upsilon_{\gamma,\nu}^2 = \left\{
\begin{aligned}
&\frac{\gamma}{\nu},\quad &\text{if }\gamma<e^{-2}\nu\\
&\left(\frac{2e^{-1}}{\ln\frac{\nu}{\gamma}}\right)^2,\quad &\text{otherwise.}
\end{aligned}
\right.
\end{equation*}
\end{lemma}
\begin{proof}
We estimate
\begin{align*}
\left\Vert Kz\right\Vert_\gamma &= 2\sum_{k\geq 1} k \left\vert z_k\right\vert\gamma^k \\
&= 2\sum_{k\geq 1} k\left(\frac{\gamma}{\nu}\right)^k \left\vert z_k\right\vert\nu^k \\
&\leq \left\Vert z\right\Vert_\nu \sup\limits_{k\geq 1} k\left(\frac{\gamma}{\nu}\right)^k.
\end{align*}
The constant $\Upsilon_{\gamma,\nu}^1$ is  the maximum (on $[1,+\infty)$) of the function $k\mapsto k\left(\frac{\gamma}{\nu}\right)^k$. Similarly, $\Upsilon_{\gamma,\nu}^2$ is the maximum (on $[1,+\infty)$) of the function $k\mapsto k^2\left(\frac{\gamma}{\nu}\right)^k$. 
\end{proof}

For $v=\bar v+\varepsilon_v$ we have  $Kv=K\bar v+K\varepsilon_v$. The sequence $K\bar v$ can be rigorously computed, being  finite dimensional, while by the Lemma~\ref{lem:estimate_derivative} $\left\Vert K\varepsilon_v\right\Vert_\gamma\leq \Upsilon_{\gamma,\nu}^1 r_\nu$. The same argument used to derive $\bar u$ and $\epsilon_u$ and the application of the  Lemma~\ref{lem:estimate_derivative} when requested, provide the following expressions for $\bar c_j$, and $\epsilon_j(\gamma)$:
\begin{equation*}
\bar c_1=-2d_{12}(K\bar v\star \bar p),\quad  \bar c_2=-2d_{12}(K\bar u\star \bar p),\quad \bar c_3=((r_1-2a_1\bar u-b_1\bar v+d_{12}K\bar s)\ast \bar p) +\frac{d_{12} b_2}{d_2}(\bar u\ast \bar v\ast \bar p),
\end{equation*}
\begin{equation*}
\bar c_4=((-d_{12}K^2\bar u-b_1\bar u)\ast\bar p)-\frac{d_{12}}{d}(\bar u\ast\bar p\ast (r_2-b_2\bar u-2a_2\bar v)),\quad \bar c_5=-\bar p,\quad \bar c_6=\frac{d_{12}}{d_2}(\bar u\ast \bar p),
\end{equation*}
\begin{equation*}
\bar c_7=-\frac{b_2}{d_2}\bar v,\quad \bar c_8=\frac{1}{d_2}(r_2-b_2\bar u-2a_2\bar v),\quad \bar c_9=-\frac{1}{d_2},
\end{equation*} 
together with
\begin{equation*}
\epsilon_1(\gamma)=2d_{12}(\left\Vert K\bar v\right\Vert_\gamma r_\nu+\left\Vert \bar p\right\Vert_\gamma\Upsilon_{\gamma,\nu}^1 r_\nu+\Upsilon_{\gamma,\nu}^1 r_\nu^2),\quad \epsilon_2(\gamma)=2d_{12}(\left\Vert K\bar u\right\Vert_\gamma r_\nu+\left\Vert \bar p\right\Vert_\gamma\Upsilon_{\gamma,\nu}^1 \epsilon_u+\Upsilon_{\gamma,\nu}^1 \epsilon_u r_\nu),
\end{equation*}
\begin{align*}
\epsilon_3(\gamma) &= \left\Vert r_1-2a_1\bar u-b_1\bar v+d_{12}K\bar s \right\Vert_\gamma r_\nu +(2a_1\epsilon_u+b_1r_\nu+d_{12}\Upsilon_{\gamma,\nu}^1r_\nu)\left\Vert\bar p \right\Vert_\gamma + (2a_1\epsilon_u+b_1r_\nu+d_{12}\Upsilon_{\gamma,\nu}^1r_\nu)r_\nu\\
&\quad + \frac{d_{12}b_2}{d_2}(\left\Vert \bar u\ast\bar v\right\Vert_\gamma r_\nu+\left\Vert \bar u\ast\bar p\right\Vert_\gamma r_\nu+\left\Vert \bar v\ast\bar p\right\Vert_\gamma \epsilon_u+\left\Vert \bar u\right\Vert_\gamma r_\nu^2+\left\Vert \bar v\right\Vert_\gamma \epsilon_u r_\nu+\left\Vert \bar p\right\Vert_\gamma \epsilon_u r_\nu+\epsilon_ur_\nu^2),
\end{align*}
\begin{align*}
\epsilon_4(\gamma) &= \left\Vert -d_{12}K^2\bar u-b_1\bar u\right\Vert_\gamma r_\nu + (d_{12}\Upsilon_{\gamma,\nu}^2\epsilon_u+b_1\epsilon_u)\left\Vert \bar p\right\Vert_\gamma + (d_{12}\Upsilon_{\gamma,\nu}^2\epsilon_u+b_1\epsilon_u)r_\nu\\
&\quad  + \frac{d_{12}}{d_2}\left(\left\Vert \bar u\ast\bar p\right\Vert_\gamma(b_2\epsilon_u+2a_2r_\nu) + \left\Vert \bar u\ast(r_2-b_2\bar u-2a_2\bar v)\right\Vert_\gamma r_\nu \right.\\
&\qquad \qquad \left. + \left\Vert\bar p\ast(r_2-b_2\bar u-2a_2\bar v)\right\Vert_\gamma \epsilon_u +\epsilon_ur_\nu(b_2\epsilon_u+2a_2r_\nu)\right),
\end{align*}
\begin{equation*}
\epsilon_5=r_\nu,\quad \epsilon_6(\gamma)=\frac{d_{12}}{d_2}(\left\Vert \bar u\right\Vert_\gamma r_\nu+\left\Vert \bar u\ast\bar p\right\Vert_\gamma\epsilon_u+\epsilon_u r_\nu),\quad \epsilon_7=\frac{b_2}{d_2}r_\nu,\quad \epsilon_8=\frac{1}{d_2}(b_2\epsilon_u+2a_2r_\nu),\quad \epsilon_9=0.
\end{equation*}

\subsection{Proof of instability: the operators $\bm{A}$ and $\bm{A^\dag}$}\label{sec:A2}

We now introduce the approximate solution and linear operators needed to apply Theorem~\ref{th:radii_pol} in the context of the eigenproblem~\eqref{eq:expanded_eigen}. 

Compared to the situation of Section~\ref{sec:fixed_point}, we have here an additional difficulty due to the fact that the function $F$ depends on the coefficients $\left(c_j\right)_{1\leq j\leq 9}$, which (as detailed above) are only known \emph{up to an error bound}. This motivates the splitting of the function $F$ into two parts, one containing the known terms $\bar c_j$ and the other one containing the remainder terms $\varepsilon_j$. More precisely, we define $\bar F$ by
\begin{align}
\label{eq:Fbar}
&\bar F_k^{(\xi)}(X)=-(\pi k)^2 \xi_k-(\bar c_1\bullet K\xi)_k-(\bar c_2\bullet K\eta)_k+(\bar c_3\ast\xi)_k+(\bar c_4\ast\eta)_k+\lambda(\bar c_5\ast\xi)_k+\lambda(\bar c_6\ast\eta)_k, &\forall~k\in \N, \nonumber\\
&\bar F_k^{(\eta)}(X)=-(\pi k)^2 \eta_k+(\bar c_7\ast\xi)_k+(\bar c_8\ast\eta)_k+\lambda(\bar c_9\ast\eta)_k, &\forall~k\in \N, \nonumber\\
&\bar F^{(\lambda)}(X)=\xi_{k_0}-1,
\end{align}
and $\E_F$ as
\begin{align}
\label{eq:EF}
&\left(\E_F\right)_k^{(\xi)}(X)=-(\varepsilon_1\bullet K\xi)_k-(\varepsilon_2\bullet K\eta)_k+(\varepsilon_3\ast\xi)_k+(\varepsilon_4\ast\eta)_k+\lambda(\varepsilon_5\ast\xi)_k+\lambda(\varepsilon_6\ast\eta)_k, &\forall~k\in \N, \nonumber\\
&\left(\E_F\right)_k^{(\eta)}(X)=(\varepsilon_7\ast\xi)_k+(\varepsilon_8\ast\eta)_k+\lambda(\varepsilon_9\ast\eta)_k, &\forall~k\in \N, \nonumber\\
&\left(\E_F\right)^{(\lambda)}(X)=0,
\end{align}
so that $F=\bar F+\E_F$. 

Then, extending again the notations introduced in Definition~\ref{def:truncation}, we denote
\begin{equation*}
\hat X^n=(\hat\xi^n,\hat\eta^n,\lambda).
\end{equation*}
Notice that the truncation parameter $n$ does not need to be (and in practice is not) the same as the truncation parameter $m$ used for the steady states. However, we require $n>k_0$, so that the isolating condition that we imposed is incorporated in the finite dimensional projection. We also define
\begin{equation*}
\hat F^n = \left(\left(\bar F^{(\xi)}_k\right)_{0\leq k<n},\left(\bar F^{(\eta)}_k\right)_{0\leq k<n},\bar F^{(\lambda)}\right).
\end{equation*}
We consider $\hat F^n$ as acting on truncated sequences $\hat X^n$ only, so that we can see it as a function mapping $\C^{2n+1}$ to itself. Therefore, finding $\hat X^n$ such that $\hat F^n(\hat X^n)=0$ is a finite dimensional problem that can be solved numerically. Notice that crucially, $\hat F$ is a finite dimensional projection of $\bar F$ rather than of $F$, so it only depends on coefficients that are known explicitly.

We now assume that we have computed numerically a zero of $\hat F^n$, and denote it $\bar X$. The next step is to define $A^{\dag}$ and $A$. Again, we are going to take for $A^{\dag}$ an approximation of $DF(\bar X)$, with a diagonal tail. More precisely, we  define $A^{\dag}$ (acting on $X=(\xi,\eta,\lambda)\in\X_{\gamma}$), as
\begin{equation*}
\widehat{A^{\dag}X}^n= D\hat F^n(\bar X)\hat X^n,
\end{equation*}
and
\begin{equation*}
\left(A^{\dag}X\right)_k = \left(-(\pi k)^{2} \xi_k,-(\pi k)^{2} \eta_k\right),\quad \forall~k\geq n.
\end{equation*}
Then, we consider $\hat A^n$ a numerically computed inverse of $D\hat F^n(\bar X)$ and define $A$ (acting on $X=(\xi,\eta,\lambda)\in\X_{\gamma}$), as
\begin{equation*}
\widehat{AX}^n= \hat A^n\hat X^n,
\end{equation*}
and
\begin{equation*}
\left(AX\right)_k = \left(-(\pi k)^{-2} \xi_k,-(\pi k)^{-2} \eta_k\right),\quad \forall~k\geq n.
\end{equation*}
\begin{remark}
As in Remark~\eqref{rem:diagonal}, we point out the diagonal dominant behaviour of the derivative $DF(\bar X)$ is the result of some preliminary manipulations done on the differential system, in this case the multiplication of the eigenproblem~\eqref{eq:expanded_eigen} by $M_1^{-1}$. 
\end{remark}

The definition of the tail part of $A$ and the fact that $\ell^1_\gamma$ is an algebra for both convolution products $\ast$ and $\bullet$ (see Lemma~\ref{lem:convolution_algebra}) ensure that $AF$ does map $\X_\gamma$ into itself as requested to apply Theorem~\ref{th:radii_pol}. 

Adopting a similar bloc notation as introduced in Section~\ref{sec:Z0}, we write
\begin{equation*}
A=\begin{pmatrix}
A^{(\xi,\xi)} & A^{(\xi,\eta)} & A^{(\xi,\lambda)} \\
A^{(\eta,\xi)} & A^{(\eta,\eta)} & A^{(\eta,\lambda)}\\   A^{(\lambda,\xi)} & A^{(\lambda,\eta)} & A^{(\lambda,\lambda)}
\end{pmatrix},
\end{equation*}
and define
\begin{equation*}
\Theta_A^{(\xi)} = \VERT A^{(\xi,\xi)}\VERT_\gamma + \VERT A^{(\eta,\xi)}\VERT_\gamma + \VERT A^{(\lambda,\xi)}\VERT_\gamma, 
\quad 
\Theta_A^{(\eta)} = \VERT A^{(\xi,\eta)}\VERT_\gamma + \VERT A^{(\eta,\eta)}\VERT_\gamma + \VERT A^{(\lambda,\eta)}\VERT_\gamma,
\end{equation*}
\begin{equation*}
\Theta_A^{(\lambda)} = \Vert A^{(\xi,\lambda)}\Vert_\gamma + \Vert A^{(\eta,\lambda)}\Vert_\gamma + \vert A^{(\lambda,\lambda)}\vert.
\end{equation*}
\subsection{Proof of instability: the bounds $\bm{Y}$ and $\bm{Z_i(r)}$}
\label{sec:bounds_instability}

Consider $\X_\gamma$, $F$, $\bar X=(\bar\xi,\bar\eta,\bar\lambda)$, $A$, $A^{\dag}$ as defined in Sections~\ref {sec:F2}-\ref{sec:A2}. Now we  derive computable bounds $Y$, $Z_0$, $Z_1$ and $Z_2$ satisfying~\eqref{def:Y}-\eqref{def:Z_2} (for $\X=\X_\gamma$).

\subsubsection{The bound $\bm{Y}$}
\label{sec:Y_instability}

\begin{lemma}
Define
\begin{align}
\label{eq:Y_instability}
Y &= \left\Vert A\bar F(\bar X)\right\Vert_{\X_\gamma} + \Theta_A^{(\xi)}\left(\left\Vert K\bar \xi\right\Vert_\gamma\epsilon_1(\gamma) + \left\Vert K\bar \eta\right\Vert_\gamma\epsilon_2(\gamma) + \left\Vert \bar \xi\right\Vert_\gamma(\epsilon_3(\gamma) +\vert\bar\lambda\vert\epsilon_5(\gamma)) + \left\Vert \bar \eta\right\Vert_\gamma(\epsilon_4(\gamma) +\vert\bar\lambda\vert\epsilon_6(\gamma))\right) \nonumber\\
&\quad + \Theta_A^{(\eta)}\left(\left\Vert \bar \xi\right\Vert_\gamma\epsilon_7(\gamma) + \left\Vert \bar \eta\right\Vert_\gamma(\epsilon_8(\gamma) +\vert\bar\lambda\vert\epsilon_9(\gamma))\right).
\end{align}
Then $Y\geq \left\Vert A F(\bar X)\right\Vert_\gamma$.
\end{lemma}
\begin{proof}
Using the splitting $F=\bar F + \E_F$ introduced in~\eqref{eq:Fbar}-\eqref{eq:EF}, we bound separately $\left\Vert A\bar F(\bar X)\right\Vert_\gamma$ and $\left\Vert A\E_F(\bar X)\right\Vert_\gamma$. %Similarly to the situation of Section~\ref{sec:Y}, notice that
 $A\bar F(\bar X)$ only has finitely many non zero coefficients, therefore $\left\Vert A\bar F(\bar X)\right\Vert_\gamma$ can be evaluated on a computer (using interval arithmetic to control the round-off errors). 

Concerning the second term, we have that
\begin{align*}
\left\Vert \E_F^{(\xi)}(\bar X)\right\Vert_\gamma &\leq \left\Vert K\bar \xi\right\Vert_\gamma\epsilon_1(\gamma) + \left\Vert K\bar \eta\right\Vert_\gamma\epsilon_2(\gamma) + \left\Vert \bar \xi\right\Vert_\gamma(\epsilon_3(\gamma) +\vert\bar\lambda\vert\epsilon_5(\gamma)) + \left\Vert \bar \eta\right\Vert_\gamma(\epsilon_4(\gamma) +\vert\bar\lambda\vert\epsilon_6(\gamma)),\\
\left\Vert \E_F^{(\eta)}(\bar X)\right\Vert_\gamma &\leq  \left\Vert \bar \xi\right\Vert_\gamma\epsilon_7(\gamma) + \left\Vert \bar \eta\right\Vert_\gamma(\epsilon_8(\gamma) +\vert\bar\lambda\vert\epsilon_9(\gamma)).
\end{align*}
Thus
\begin{align*}
\left\Vert A\E_F(\bar X)\right\Vert_{\X_\gamma} &\leq \Theta_A^{(\xi)}\left(\left\Vert K\bar \xi\right\Vert_\gamma\epsilon_1(\gamma) + \left\Vert K\bar \eta\right\Vert_\gamma\epsilon_2(\gamma) + \left\Vert \bar \xi\right\Vert_\gamma(\epsilon_3(\gamma) +\vert\bar\lambda\vert\epsilon_5(\gamma)) + \left\Vert \bar \eta\right\Vert_\gamma(\epsilon_4(\gamma) +\vert\bar\lambda\vert\epsilon_6(\gamma))\right) \\
&\quad + \Theta_A^{(\eta)}\left(\left\Vert \bar \xi\right\Vert_\gamma\epsilon_7(\gamma) + \left\Vert \bar \eta\right\Vert_\gamma(\epsilon_8(\gamma) +\vert\bar\lambda\vert\epsilon_9(\gamma))\right).
\end{align*}
The sum of the last expression and $\left\Vert A\bar F(\bar X)\right\Vert_\gamma$ gives $Y$.
\end{proof}

\subsubsection{The bound $\bm{Z_0}$}
\label{sec:Z0_instability}
Arguing exactly as in Section~\ref{sec:Z0}, we define
\begin{equation}
\label{eq:Z0_instability}
Z_0=\max\left[\Theta_{I-AA^{\dag}}^{(\xi)},\Theta_{I-AA^{\dag}}^{(\eta)},\Theta_{I-AA^{\dag}}^{(\lambda)}\right].
\end{equation}

\subsubsection{The bound $\bm{Z_1}$}
\label{sec:Z1_instability}

Now we focus on providing a bound $Z_1$ satisfying~\eqref{def:Z_1}. 
\begin{lemma}
Let $\hat\alpha^n_\xi,\hat\alpha^n_\eta$ be vectors in $\C^{2n+1}$ each, defined as
\begin{equation*}
\left(\hat\alpha^n_\xi\right)_0 = \begin{pmatrix}
\Phi_0^n(\bar c_3-K\bar c_1+\bar \lambda \bar c_5,\gamma) \\
\Phi_0^n(\bar c_7,\gamma) \\
0
\end{pmatrix},\quad 
\left(\hat\alpha^n_\eta\right)_0 = \begin{pmatrix}
\Phi_0^n(\bar c_4-K\bar c_2 + \bar \lambda\bar c_6\gamma) \\
\Phi_0^n(\bar c_8+\bar \lambda\bar c_9,\gamma) \\          0
\end{pmatrix},
\end{equation*}
and for all $1\leq k<n$
\begin{equation*}
\left(\hat\alpha^n_\xi\right)_k = \begin{pmatrix}
k\Phi_k^n(\bar c_1,\gamma)+\Phi_k^n(\bar c_3-K\bar c_1+\bar \lambda \bar c_5,\gamma) \\
\Phi_k^n(\bar c_7,\gamma) 
\end{pmatrix},\quad 
\left(\hat\alpha^n_\eta\right)_k = \begin{pmatrix}
k\Phi_k^n(\bar c_2,\gamma)+\Phi_k^n(\bar c_4-K\bar c_2 + \bar \lambda\bar c_6\gamma) \\
\Phi_k^n(\bar c_8+\bar \lambda\bar c_9,\gamma)
\end{pmatrix}.
\end{equation*}
Let the operator $\tilde K$ acting on $\X_\gamma$ defined as
\begin{equation*}
\tilde K=\begin{pmatrix}
K & 0 & 0\\
0 & 0 & 0\\
0 & 0 & 0
\end{pmatrix},
\end{equation*}
where $K$ is the same operator as in Definition~\ref{def:K}.

Choose $\tilde\gamma\in(\gamma,\nu)$ and 
define
\begin{align}
\label{eq:Z1_instability}
Z_1 &= \max\left[\left\Vert \vert  A\vert \hat\alpha^n_\xi \right\Vert_{\X_\gamma},\left\Vert \vert  A\vert \hat\alpha^n_\eta \right\Vert_{\X_\gamma}\right] \nonumber\\
&\quad + \max\left[\frac{\left\Vert\bar c_1 \right\Vert_\gamma}{\pi n}+\frac{\left\Vert\bar c_3-K\bar c_1+\bar\lambda\bar c_5 \right\Vert_\gamma+\left\Vert\bar c_7 \right\Vert_\gamma
}{(\pi n)^2}, \frac{\left\Vert\bar c_2 \right\Vert_\gamma}{\pi n}+\frac{\left\Vert\bar c_4-K\bar c_2+\bar\lambda\bar c_6 \right\Vert_\gamma+\left\Vert\bar c_8 +\bar\lambda\bar c_9 \right\Vert_\gamma
}{(\pi n)^2}, \right. \nonumber\\
&\qquad\qquad \left. \frac{\left\Vert\bar c_5\ast\bar\xi\right\Vert_\gamma + \left\Vert\bar c_6\ast\bar\eta\right\Vert_\gamma + \left\Vert\bar c_9\ast\bar\eta\right\Vert_\gamma}{(\pi n)^2} \right] \nonumber\\
&\quad + \max\left[\Theta_A^{(\xi)}\left(\Upsilon_{\gamma,\tilde\gamma}^1\epsilon_1(\tilde\gamma) + \epsilon_3(\gamma) + \vert\bar\lambda\vert\epsilon_5(\gamma)\right) + \Theta_A^{(\eta)}\epsilon_7(\gamma), \right. \nonumber\\
&\qquad\qquad \left. \Theta_A^{(\xi)}\left(\Upsilon_{\gamma,\tilde\gamma}^1\epsilon_2(\tilde\gamma) + \epsilon_4(\gamma) + \vert\bar\lambda\vert\epsilon_6(\gamma)\right) + \Theta_A^{(\eta)}\left(\epsilon_8(\gamma) + \vert\bar\lambda\vert\epsilon_9(\gamma)\right), \right. \nonumber\\
&\qquad\qquad \left. \Theta_A^{(\xi)}\left(\left\Vert\bar\xi\right\Vert_\gamma\epsilon_5(\gamma) + \left\Vert\bar\eta\right\Vert_\gamma\epsilon_6(\gamma)\right) + \Theta_A^{(\eta)}\left\Vert\bar\eta\right\Vert_\gamma\epsilon_9(\gamma) \right] \nonumber\\
&\quad + \Theta_{A\tilde K}^{(\xi)}\max\left[\epsilon_1(\gamma),\epsilon_2(\gamma)\right].
\end{align}
then 
\begin{equation*}
Z_1\geq \left\VERT A\left(DF(\bar X)-A^{\dag}\right) \right\VERT_{\X_\gamma}
\end{equation*}
\end{lemma}
\begin{proof}
Using the splitting of $F$ introduced in in~\eqref{eq:Fbar}-\eqref{eq:EF} we have
\begin{equation*}
\left\VERT A\left(DF(\bar X)-A^{\dag}\right) \right\VERT_{\X_\gamma} \leq \left\VERT A\left(D\bar F(\bar X)-A^{\dag}\right) \right\VERT_{\X_\gamma} + \left\VERT AD\E_F(\bar X) \right\VERT_{\X_\gamma}.
\end{equation*}

The procedure to obtain a bound for the  first part on the r.h.s is  similar to the one followed in  Section~\ref{def:Z_1} (since  the remainder terms $\left(\varepsilon_j\right)_{1\leq j\leq 9}$ are not involved), therefore we skip most of  the details. Let $X\in\B_{\X_\gamma}(0,1)$ and introduce $U=\left(D\bar F(\bar X)-A^{\dag}\right)X$. We have
\begin{align*}
\left\Vert A\left(D\bar F(\bar X)-A^{\dag}\right)X \right\Vert_{\X_\gamma} &\leq \left\Vert \vert A\vert \vert U\vert \right\Vert_{\X_\gamma} \\
&\leq \Vert \vert A\vert \vert \hat U^n\vert \Vert_{\X_\gamma} + \Vert \vert A\vert \vert \check U^n\vert \Vert_{\X_\gamma},
\end{align*} 
and we provide a bound for each term separately. For both of them, it is helpful to notice that
\begin{equation}
\label{eq:derivative_product}
(\bar c_1\bullet K\xi)=-K(\bar c_1\star\xi)+(K\bar c_1\ast\xi),
\end{equation}
and similarly for $(\bar c_2\bullet K\eta)$, this identity being nothing but $\bar c_1\xi'=(\bar c_1 \xi)' - \bar c_1'\xi$ written for the Fourier sequences. 
Using~\eqref{eq:derivative_product} and  proceeding exactly as in Section~\ref{def:Z_1}, we obtain
\begin{align}\label{eq:Z1part1}
\Vert \vert  A\vert \vert\hat U^n\vert \Vert_{\X_\gamma} \leq \max\left[\left\Vert \vert  A\vert \hat\alpha^n_\xi \right\Vert_{\X_\gamma},\left\Vert \vert  A\vert \hat\alpha^n_\eta \right\Vert_{\X_\gamma}\right]\left\Vert X\right\Vert_{\X_\gamma}.
\end{align} 
For the tail part, using Lemma~\ref{lem:convolution_algebra}, and the definition of the tail part of $A$, we compute
\begin{align}
\left\Vert \vert A\vert \vert \check U^n\vert \right\Vert_{\X_\gamma} &\leq \max\left[\frac{\left\Vert\bar c_1 \right\Vert_\gamma}{\pi n}+\frac{\left\Vert\bar c_3-K\bar c_1+\bar\lambda\bar c_5 \right\Vert_\gamma+\left\Vert\bar c_7 \right\Vert_\gamma
}{(\pi n)^2}, \right. \nonumber\\
&\qquad\qquad \left. \frac{\left\Vert\bar c_2 \right\Vert_\gamma}{\pi n}+\frac{\left\Vert\bar c_4-K\bar c_2+\bar\lambda\bar c_6 \right\Vert_\gamma+\left\Vert\bar c_8 +\bar\lambda\bar c_9 \right\Vert_\gamma
}{(\pi n)^2}, \right. \nonumber \\
&\qquad\qquad \left. \frac{\left\Vert\bar c_5\ast\bar\xi\right\Vert_\gamma + \left\Vert\bar c_6\ast\bar\eta\right\Vert_\gamma + \left\Vert\bar c_9\ast\bar\eta\right\Vert_\gamma}{(\pi n)^2} \right]\left\Vert X\right\Vert_{\X_\gamma}. \label{eq:Z1part2}
\end{align}

It remains to estimate $\left\VERT AD\E_F(\bar X) \right\VERT_{\X_\gamma}$. We want to proceed in a similar fashion as in Section~\ref{sec:Y_instability} where we computed a bound  for  $\left\Vert A\E_F(\bar X) \right\Vert_{\X_\gamma}$. However, we have to be slightly more careful, since $\left\VERT D\E_F(\bar X) \right\VERT_{\X_\gamma}$  is not finite (because of the $\pi k$ terms coming from the first order derivatives). Therefore, using again~\eqref{eq:derivative_product}, we separate the unbounded contributions and decompose $\E_F$ into $\E^1_F+\E^2_F$, where
\begin{align*}
&\left(\E^1_F\right)_k^{(\xi)}(X)=-(K\varepsilon_1\ast \xi)_k-(K\varepsilon_2\ast \eta)_k+(\varepsilon_3\ast\xi)_k+(\varepsilon_4\ast\eta)_k+\lambda(\varepsilon_5\ast\xi)_k+\lambda(\varepsilon_6\ast\eta)_k, &\forall~k\in \N, \nonumber\\
&\left(\E^1_F\right)_k^{(\eta)}(X)=(\varepsilon_7\ast\xi)_k+(\varepsilon_8\ast\eta)_k+\lambda(\varepsilon_9\ast\eta)_k, &\forall~k\in \N, \nonumber\\
&\left(\E^2_F\right)^{(\lambda)}(X)=0,
\end{align*}
and
\begin{align*}
&\left(\E^2_F\right)_k^{(\xi)}(X)=\pi k(\varepsilon_1\star \xi)_k+\pi k(\varepsilon_2\star \eta)_k, &\forall~k\in \N, \nonumber\\
&\left(\E^2_F\right)_k^{(\eta)}(X)=0, &\forall~k\in \N, \nonumber\\
&\left(\E^1_F\right)^{(\lambda)}(X)=0.
\end{align*}
and we provide bounds on $\left\Vert AD\E^1_F(\bar X)X \right\Vert_{\X_\gamma}$ and $\left\Vert AD\E^2_F(\bar X)X \right\Vert_{\X_\gamma}$, for $X=(\xi,\eta,\lambda)\in\B_{\X_\gamma}(0,1)$. The smoothing effect of $A$ will make the second bound finite.

First, consider
\begin{align*}
\left\Vert D\left(\E^1_F\right)^{(\xi)}(\bar X)X \right\Vert_{\X_\gamma} &\leq \left(\left\Vert K\varepsilon_1\right\Vert_\gamma + \epsilon_3(\gamma) + \vert\bar\lambda\vert\epsilon_5(\gamma)\right)\left\Vert\xi\right\Vert_\gamma + \left(\left\Vert K\varepsilon_2\right\Vert_\gamma + \epsilon_4(\gamma) + \vert\bar\lambda\vert\epsilon_6(\gamma) \right)\left\Vert\eta\right\Vert_\gamma \\
&\quad + \left(\left\Vert\bar\xi\right\Vert_\gamma\epsilon_5(\gamma) + \left\Vert\bar\eta\right\Vert_\gamma\epsilon_6(\gamma)\right)\vert\lambda\vert,\\
\left\Vert D\left(\E^1_F\right)^{(\eta)}(\bar X)X \right\Vert_{\X_\gamma} &\leq \epsilon_7(\gamma)\left\Vert\xi\right\Vert_\gamma + \left(\epsilon_8(\gamma) + \vert\bar\lambda\vert\epsilon_9(\gamma) \right)\left\Vert\eta\right\Vert_\gamma + \left\Vert\bar\eta\right\Vert_\gamma\epsilon_9(\gamma)\vert\lambda\vert.
\end{align*}
Note that explicit upper bound for the $\left\Vert K\varepsilon_1\right\Vert_\gamma$ and $\left\Vert K\varepsilon_2\right\Vert_\gamma$ are required. For this, let $\tilde\gamma$ be such that $\gamma<\tilde\gamma<\nu$ and use Lemma~\ref{lem:estimate_derivative} to obtain
\begin{equation*}
\left\Vert K\varepsilon_1\right\Vert_\gamma\leq \Upsilon_{\gamma,\tilde\gamma}^1\epsilon_1(\tilde\gamma),\qquad \left\Vert K\varepsilon_2\right\Vert_\gamma\leq \Upsilon_{\gamma,\tilde\gamma}^1\epsilon_2(\tilde\gamma).
\end{equation*}
Finally, again using the bloc notation, we have
\begin{align}
\left\Vert AD\E^1_F(\bar X)X \right\Vert_{\X_\gamma} &\leq \Theta_A^{(\xi)}\left\Vert D\left(\E^1_F\right)^{(\xi)}(\bar X)X \right\Vert_{\X_\gamma} + \Theta_A^{(\eta)}\left\Vert D\left(\E^1_F\right)^{(\eta)}(\bar X)X \right\Vert_{\X_\gamma} + \Theta_A^{(\lambda)}\left\Vert D\left(\E^1_F\right)^{(\lambda)}(\bar X)X \right\Vert_{\X_\gamma}\nonumber \\
&\leq \left(\Theta_A^{(\xi)}\left(\Upsilon_{\gamma,\tilde\gamma}^1\epsilon_1(\tilde\gamma) + \epsilon_3(\gamma) + \vert\bar\lambda\vert\epsilon_5(\gamma)\right) + \Theta_A^{(\eta)}\epsilon_7(\gamma)\right)\left\Vert\xi\right\Vert_\gamma\nonumber \\
&\quad + \left(\Theta_A^{(\xi)}\left(\Upsilon_{\gamma,\tilde\gamma}^1\epsilon_2(\tilde\gamma) + \epsilon_4(\gamma) + \vert\bar\lambda\vert\epsilon_6(\gamma)\right) + \Theta_A^{(\eta)}\left(\epsilon_8(\gamma) + \vert\bar\lambda\vert\epsilon_9(\gamma)\right)\right)\left\Vert\eta\right\Vert_\gamma\nonumber \\
&\quad + \left(\Theta_A^{(\xi)}\left(\left\Vert\bar\xi\right\Vert_\gamma\epsilon_5(\gamma) + \left\Vert\bar\eta\right\Vert_\gamma\epsilon_6(\gamma)\right) + \Theta_A^{(\eta)}\left\Vert\bar\eta\right\Vert_\gamma\epsilon_9(\gamma)\right)\vert\lambda\vert \nonumber\\
&\leq \max\left[\Theta_A^{(\xi)}\left(\Upsilon_{\gamma,\tilde\gamma}^1\epsilon_1(\tilde\gamma) + \epsilon_3(\gamma) + \vert\bar\lambda\vert\epsilon_5(\gamma)\right) + \Theta_A^{(\eta)}\epsilon_7(\gamma), \right. \nonumber \\
&\qquad\qquad \left. \Theta_A^{(\xi)}\left(\Upsilon_{\gamma,\tilde\gamma}^1\epsilon_2(\tilde\gamma) + \epsilon_4(\gamma) + \vert\bar\lambda\vert\epsilon_6(\gamma)\right) + \Theta_A^{(\eta)}\left(\epsilon_8(\gamma) + \vert\bar\lambda\vert\epsilon_9(\gamma)\right), \right. \nonumber\\
&\qquad\qquad \left. \Theta_A^{(\xi)}\left(\left\Vert\bar\xi\right\Vert_\gamma\epsilon_5(\gamma) + \left\Vert\bar\eta\right\Vert_\gamma\epsilon_6(\gamma)\right) + \Theta_A^{(\eta)}\left\Vert\bar\eta\right\Vert_\gamma\epsilon_9(\gamma) \right] \left\Vert X\right\Vert_{\X_\gamma}. \label{eq:Z1part3}
\end{align}
To deal with $\left\Vert AD\E^2_F(\bar X)X \right\Vert_{\X_\gamma}$, we introduce $\tilde\E_F^2$ defined as
\begin{align*}
&\left(\tilde\E^2_F\right)_k^{(\xi)}(X)=(\varepsilon_1\star \xi)_k+(\varepsilon_2\star \eta)_k, &\forall~k\in \N, \nonumber\\
&\left(\tilde\E^2_F\right)_k^{(\eta)}(X)=0, &\forall~k\in \N, \nonumber\\
&\left(\tilde\E^2_F\right)^{(\lambda)}(X)=0,
\end{align*}
to get
\begin{equation*}
AD\E^2_F(\bar X)X = A\tilde K D\tilde\E^2_F(\bar X)X.
\end{equation*}
Now we can estimate
\begin{equation*}
\left\Vert D\left(\tilde\E^2_F\right)^{(\xi)}(\bar X)X \right\Vert_{\X_\gamma} \leq \epsilon_1(\gamma)\left\Vert\xi\right\Vert_\gamma + \epsilon_2(\gamma)\left\Vert\eta\right\Vert_\gamma,
\end{equation*}
so to obtain
\begin{equation}\label{eq:Z1part4}
\left\Vert AD\E^2_F(\bar X)X \right\Vert_{\X_\gamma}\leq \Theta_{A\tilde K}^{(\xi)}\max\left[\epsilon_1(\gamma),\epsilon_2(\gamma)\right]\left\Vert X\right\Vert_{\X_\gamma}.
\end{equation}
Notice that  $\Theta_{A\tilde K}^{(\xi)}$ is finite and can be compute explicitly, because the tail of $A$ is diagonal an decreases like $(\pi k)^{-2}$. For instance
\begin{align*}
\left\VERT (A\tilde K)^{(\xi,\xi)}\right\VERT_\gamma &= \sup_{j\geq 0}\frac{1}{\gamma^{j}}\sum_{k\geq 0}\pi j|A^{(\xi,\xi)}(k,j)|\gamma^{k}\\
&= \max\left[\max_{0\leq j< m}\frac{\pi j}{\gamma^{j}}\sum_{0\leq k <m}|A^{(\xi,\xi)}(k,j)|\gamma^{k},\sup_{j\geq m}\frac{1}{\gamma^{j}}\pi j|-(\pi j)^{-2}|\gamma^{j}\right] \\
&= \max\left[\max_{0\leq j< m}\frac{\pi j}{\gamma^{j}}\sum_{0\leq k <m}|A^{(\xi,\xi)}(k,j)|\gamma^{k},\frac{1}{\pi m}\right].
\end{align*}
The sum of all contributions \eqref{eq:Z1part1}-\eqref{eq:Z1part4} gives the required $Z_1$.
\end{proof}

\subsubsection{The bound $\bm{Z_2}$}

Since $F$ is quadratic, we have that for all $X'\in\B_{\X_\gamma}(0,r)$ and  $X\in\B_{\X_\gamma}(0,1)$
\begin{equation*}
A\left(DF(\bar X+X')-DF(\bar X)\right)X = AD^2F(\bar X)(X,X').
\end{equation*}
Direct computations give 
\begin{align*}
D^2F^{(\xi)}(\bar X)(X,X')&=\lambda(c_5\ast\xi')+\lambda'(c_5\ast\xi)+\lambda(c_6\ast\eta')+\lambda'(c_6\ast\eta),\\
D^2F^{(\eta)}(\bar X)(X,X')&=\lambda(c_9\ast\eta')+\lambda'(c_9\ast\eta),\\
D^2F^{(\lambda)}(\bar X)(X,X')&=0,
\end{align*}
therefore
\begin{align*}
\left\Vert AD^2F(\bar X)(X,X')\right\Vert_{\X_\gamma}&\leq \Theta_A^{(\xi)}\left(\left\Vert\bar c_5\right\Vert_{\gamma}+\epsilon_5(\gamma)\right)\left(\left\Vert\xi\right\Vert_\gamma\vert\lambda'\vert + \left\Vert\xi'\right\Vert_\gamma\vert\lambda\vert\right) \\
&\quad + \left(\Theta_A^{(\xi)}\left(\left\Vert\bar c_6\right\Vert_{\gamma}+\epsilon_6(\gamma)\right) + \Theta_A^{(\eta)}\left(\left\Vert\bar c_9\right\Vert_{\gamma}+\epsilon_9(\gamma)\right)\right)\left(\left\Vert\eta\right\Vert_\gamma\vert\lambda'\vert + \left\Vert\eta'\right\Vert_\gamma\vert\lambda\vert\right) \\
&\leq \max\left[\Theta_A^{(\xi)}\left(\left\Vert\bar c_5\right\Vert_{\gamma}+\epsilon_5(\gamma)\right) , \Theta_A^{(\xi)}\left(\left\Vert\bar c_6\right\Vert_{\gamma}+\epsilon_6(\gamma)\right) + \Theta_A^{(\eta)}\left(\left\Vert\bar c_9\right\Vert_{\gamma}+\epsilon_9(\gamma)\right)\right]\\ &\quad\times \left(\left\Vert\xi\right\Vert_\gamma + \left\Vert\eta\right\Vert_\gamma + \vert\lambda\vert\right)\left(\left\Vert\xi'\right\Vert_\gamma + \left\Vert\eta'\right\Vert_\gamma + \vert\lambda'\vert\right) \\
&= \max\left[\Theta_A^{(\xi)}\left(\left\Vert\bar c_5\right\Vert_{\gamma}+\epsilon_5(\gamma)\right) , \Theta_A^{(\xi)}\left(\left\Vert\bar c_6\right\Vert_{\gamma}+\epsilon_6(\gamma)\right) + \Theta_A^{(\eta)}\left(\left\Vert\bar c_9\right\Vert_{\gamma}+\epsilon_9(\gamma)\right)\right]r\left\Vert X\right\Vert_{\X_\gamma}.
\end{align*}
Thus, we define
\begin{equation}
\label{eq:Z2_instability}
Z_2=\max\left[\Theta_A^{(\xi)}\left(\left\Vert\bar c_5\right\Vert_{\gamma}+\epsilon_5(\gamma)\right) , \Theta_A^{(\xi)}\left(\left\Vert\bar c_6\right\Vert_{\gamma}+\epsilon_6(\gamma)\right) + \Theta_A^{(\eta)}\left(\left\Vert\bar c_9\right\Vert_{\gamma}+\epsilon_9(\gamma)\right)\right].
\end{equation}

\subsection{Proof of instability: The radii polynomial}

We now collect all the bounds developed above into a statement about the stability of the steady states.
\begin{proposition}
\label{prop:radii_pol_instability}
Let $\nu>1$. Assume to have computed finite sequences of Fourier coefficients $\bar v$, $\bar w$, $\bar p$, $\bar s$ and $r_\nu>0$ such that there exists a unique $(v,w,p,s)\in\left(\ell^1_\nu(\R)\right)^4$ that solves~\eqref{eq:algebraic_system} and satisfies
\begin{equation*}
\left\Vert v-\bar v\right\Vert_\nu+\left\Vert w-\bar w\right\Vert_\nu+\left\Vert p-\bar p\right\Vert_\nu+\left\Vert s-\bar s\right\Vert_\nu \leq r_\nu.
\end{equation*}
Choose $1<\gamma<\nu$ and let $\X_\gamma$, $F$, $\bar X$, $A$, $A^{\dag}$ be as  in Sections~\ref{sec:F2}-\ref{sec:A2}. Suppose to have computed the bounds $Y$, $Z_0$, $Z_1$ and $Z_2$, defined in~\eqref{eq:Y_instability}, \eqref{eq:Z0_instability}, \eqref{eq:Z1_instability} and \eqref{eq:Z2_instability} respectively. 
If there exists $r>0$ such that
\begin{equation*}
P(r)=Z_2(r)r^2 -\left(1-\left(Z_0+Z_1\right)\right)r +Y<0,
\end{equation*}
then there exists a unique zero of $F$ in $\B_{\X_\gamma}(\bar X,r)$.  If moreover $\Re(\bar\lambda)>r$ then the steady state $(u,v)$, $u=pw$, is unstable.

\end{proposition}
\begin{proof}
It follows as application of Theorem \ref{th:radii_pol} and Lemma~\ref{lem:rigorous_eigen}
\end{proof}

\section{Results about the instability of steady states}
\label{sec:results_instability}

In this section, we give some details about the proof of Theorem~\ref{th:unstable_11}. We recall that the parameters of~\eqref{eq:steady_states} are fixed as $\Omega=(0,1)$, $r_1=5$, $r_2=2$, $a_1=3$, $a_2=3$, $b_1=1$, $b_2=1$, $d_{12}=3$, and that $d_1=d_2=d$ is left as the bifurcation parameter. For each solution represented by a blue dot on Figure~\ref{fig:stability_diagram}, we proved the existence of an unstable eigenvalue, using the procedure described at the end of Section~\ref{sec:validated_numerics}. In particular, for each of these steady states we computed numerically an eigenvalue with positive real part, implemented the bounds described in Section~\ref{sec:bounds_instability}, and then \emph{successfully} applied Proposition~\ref{prop:radii_pol_instability} to validate the numerical eigenvalue. By successfully we mean that we found a positive $r$ such that $P(r)<0$ and checked that $\Re(\bar\lambda)>r$. For the steady states displayed previously in Figure~\ref{fig:13solutions}, we detail in Table~\ref{table:data_instability} what is the value of the unstable eigenvalue, what dimension $n$ was used for the finite dimensional projection, what $\gamma$ was chosen for the space $\X_\gamma$, and give a validation radius $r$ for which the proof is succesfull with those parameters (for the steady states where an unstable eigenvalue was actually found).

\medskip

\noindent{\it Proof of Theorem ~\ref{th:unstable_11}.}
In the script \verb+script_proof_branch_instability.m+ fix the values of the parameters $r_1=5$, $r_2=2$, $a_1=3$, $a_2=3$, $b_1=1$, $b_2=1$, $d_{12}=3$. The parameter $d_1=d_2=d$ is intended as the bifurcation parameter.  Choose a value for the finite dimensional projection $n$ and a value for the norm weight $\gamma>1$. Also select  a branch of solutions (for the names of the several branches we refer to the documentation and the \verb+readme+ file). The script loads the numerical data, computes the required bounds and verifies the existence of an interval $\mathcal I=(r_1,r_2)$  such that $P(r)<0$ for any $r\in \mathcal I$. If $\mathcal I$ is not empty then the condition $\Re(\bar\lambda)>r$ is checked. In case of successful computation, Proposition~\ref{prop:radii_pol_instability} implies that the concerned steady state is unstable.
The values for $n$ and $\gamma$ that allow the rigorous computation of all the branches depicted in the Figure~\ref{fig:proved_diagram} are available in the documentation. 

The script \verb+script_proof_steadystate_and_instability.m+ concerns the existence of steady states for a fixed value of $d$. It is used to prove the existence of 13 solutions at values $d=0.005$. 
 Figure~\ref{fig:13solutions} shows the numerical data for the 13 steady states solutions. In Table~\ref{table:data_instability} we detail the values for $n$ and  $\gamma$ used in the proof and the resulting  validation radius $r$. 
\qed

\renewcommand{\arraystretch}{1.2}
\begin{table}[ht!]
\centering
\begin{tabular}{|c|c|c|c|c|}
\hline
Label for the solution (see Figure~\ref{fig:13solutions}) & unstable eigenvalue $\lambda$ & $n$  & $\gamma$ & Validation radius \\
\hline
(a) & no unstable eigenvalue found & -- & -- & -- \\
\hline
(b) & $0.0153$ & 1000 & 1.0001 & $3.2807\times 10^{-7}$ \\
\hline
(c) & $0.2050 \pm 0.1673i$ & 1000 & 1.0001 & $7.8894\times 10^{-6}$ \\
\hline
(d) & $0.0463 \pm 0.0524i$ & 1000 & 1.0001 & $8.0803\times 10^{-6}$ \\
\hline
(e) & $0.0844$ & 1000 & 1.0001 & $2.5572\times 10^{-7}$ \\
\hline
(f) & $0.0463 \pm 0.0524i$ & 1100 & 1.0001 & $1.2238\times 10^{-5}$ \\
\hline
(g) & $0.0570 \pm 0.0390i$ & 1400 & 1.0001 & $1.0099\times 10^{-5}$ \\
\hline
(h) & $0.2743$ & 1000 & 1.0001 & $4.3919\times 10^{-9}$ \\
\hline
(i) & $0.0844$ & 1000 & 1.0001 & $2.2639\times 10^{-7}$ \\
\hline
(j) & $0.0153$ & 1000 & 1.0001 & $2.1499\times 10^{-7}$ \\
\hline
(k) & $0.0570 \pm 0.0390i$ & 1500 & 1.0001 & $1.093\times 10^{-5}$ \\
\hline
(l) & no unstable eigenvalue found & -- & -- & -- \\
\hline
(m) & $0.2050 \pm 0.1673i$ & 1000 & 1.0001 & $4.0795\times 10^{-6}$ \\
\hline
\end{tabular}
\caption{For each steady state displayed in Figure~\ref{fig:13solutions}, when an unstable eigenvalue is found we give the dimension $n$ that was used for the finite dimensional projection, the weight $\gamma$ that was chosen for the space $\X_\gamma$, and a validation radius $r$ for which the proof of the eigenvalue is successful, with those parameters $n$ and $\gamma$.}
\label{table:data_instability}
\end{table}
\renewcommand{\arraystretch}{1}

\section*{Acknowledgments}

MB acknowledge partial support from the french ``ANR blanche'' project Kibord: ANR-13-BS01-0004. Both authors also wish to acknowledge the hospitality of the Lorentz Center in Leiden, which hosted the workshop \emph{Computational Proofs for Dynamics in PDEs} during which this project started.

\end{document}